\documentclass[a4paper,12pt]{amsart}
\usepackage{amssymb, amsmath, graphicx}
\usepackage[curve]{xypic}
\usepackage{enumerate}
\usepackage{tikz-cd}  

\DeclareRobustCommand{\gobblefour}[5]{}
\newcommand*{\SkipTocEntry}{\addtocontents{toc}{\gobblefour}}

\providecommand{\abs}[1]{\lvert#1\rvert}

\providecommand{\id}{\textnormal{id}}

\providecommand{\Ker}{\textnormal{Ker}}
\providecommand{\IIm}{\textnormal{Im}}
\providecommand{\Coker}{\textnormal{Coker}}

\providecommand{\Cyl}{\textnormal{Cyl}}
\providecommand{\Hom}{\textnormal{Hom}}

\providecommand{\Ker}{\textnormal{Ker}}

\providecommand{\ch}{\textnormal{ch}}

\providecommand{\cpt}{\textnormal{cpt}}

\providecommand{\vol}{\textnormal{vol}}

\providecommand{\fl}{\textnormal{fl}}

\providecommand{\N}{\mathbb{N}}
\providecommand{\Z}{\mathbb{Z}}

\providecommand{\R}{\mathbb{R}}
\providecommand{\C}{\mathbb{C}}
\providecommand{\CC}{\mathcal{C}}
\providecommand{\Cc}{\mathfrak{C}}

\providecommand{\HH}{\mathcal{H}}
\providecommand{\Hh}{\mathbb{H}}
\providecommand{\M}{\mathcal{M}}

\providecommand{\h}{\mathfrak{h}}

\providecommand{\cl}{\textnormal{cl}}
\providecommand{\ex}{\textnormal{ex}}
\providecommand{\op}{\textnormal{op}}
\providecommand{\dR}{\textnormal{dR}}
\providecommand{\Td}{\textnormal{Td}}
\providecommand{\colim}{\textnormal{colim}}
\providecommand{\A}{\mathcal{A}}

\providecommand{\cov}{\textnormal{cov}}
\providecommand{\Bock}{\textnormal{Bock}}
\providecommand{\ppar}{\textnormal{par}}
\providecommand{\K}{\mathcal{K}}

\tikzset{
	curvarr/.style={
		to path={ -- ([xshift=2ex]\tikztostart.east)
			|- (#1) [near end]\tikztonodes
			-| ([xshift=-2ex]\tikztotarget.west)
			-- (\tikztotarget)}
	}
}

\makeatletter
\newcommand\tint{\mathop{\mathpalette\tb@int{t}}\!\int}
\newcommand\bint{\mathop{\mathpalette\tb@int{b}}\!\int}
\newcommand\tb@int[2]{%
  \sbox\z@{$\m@th#1\int$}%
  \if#2t%
    \rlap{\hbox to\wd\z@{%
      \hfil
      \vrule width .35em height \dimexpr\ht\z@+1.4pt\relax depth -\dimexpr\ht\z@+1pt\relax
      \kern.05em % a small correction on the top
    }}
  \else
    \rlap{\hbox to\wd\z@{%
      \vrule width .35em height -\dimexpr\dp\z@+1pt\relax depth \dimexpr\dp\z@+1.4pt\relax
      \hfil
    }}
  \fi
}
\makeatother

\allowdisplaybreaks

\begin{document}

\title[Relative diff.\ cohomology and generalized CS characters]{Relative differential cohomology and generalized Cheeger-Simons characters}
\author{Fabio Ferrari Ruffino - Juan Carlos Rocha Barriga}
\address{Departamento de Matem\'atica - Universidade Federal de S\~ao Carlos - Rod.\ Washington Lu\'is, Km 235 - C.P.\ 676 - 13565-905 S\~ao Carlos, SP, Brasil}
\email{ferrariruffino@gmail.com, juanrocha@dm.ufscar.br}
\thanks{Fabio Ferrari Ruffino was supported by FAPESP (Funda\c{c}\~ao de Amparo \`a Pesquisa do Estado de S\~ao Paulo), processo 2014/03721-3.}

\begin{abstract}
We provide a suitable axiomatic framework for differential cohomology in the relative case and we deduce the corresponding long exact sequences. We also construct the relative version of the generalized Cheeger-Simons characters and we define the integration map when the fibre has a boundary.
\end{abstract}

\maketitle

\newtheorem{Theorem}{Theorem}[section]
\newtheorem{Lemma}[Theorem]{Lemma}
\newtheorem{Corollary}[Theorem]{Corollary}
\newtheorem{ThmDef}[Theorem]{Theorem - Definition}

\theoremstyle{definition}
\newtheorem{Rmk}[Theorem]{Remark}
\newtheorem{Rmks}[Theorem]{Remarks}
\newtheorem{Def}[Theorem]{Definition}
\newtheorem{Not}[Theorem]{Notation}

%%%%%%%%%%%%%%%%%%%%%%%%%%%%%%%%%%%%%%%%%%%%%%%%%%%%%%%%%%%%%%%%%%%%%%%%%%%%%%%%%%%%%%%%%%%%%%%

\tableofcontents

%%%%%%%%%%%%%%%%%%%%%%%%%%%%%%%%%%%%%%%%%%%%%%%%%%%%%%%%%%%%%%%%%%%%%%%%%%%%%%%%%%%%%%%%%%%%%%%

\section{Introduction}

A differential extension of a cohomology theory is a refinement of its restriction to the category of smooth manifolds, obtained enriching it with differential information. The basic example is provided by complex line bundles with connection on a manifold $X$. In this case, the topological information is completely grasped by the first Chern class, that is an element of $H^{2}(X; \Z)$. Adding a connection we refine such a class, providing a piece of information partially described by the corresponding curvature. The latter is a closed 2-form on $X$ with integral periods, whose de-Rham cohomology class is the real image of the first Chern class. Equivalently, a line bundle with connection can be characterized by the pair $(\chi, \omega)$, where $\chi$ is the corresponding holonomy map, defined on smooth $1$-cycles with values in $\R/\Z$, and $\omega$ is the curvature. The pair $(\chi, \omega)$ satisfies a Stokes-type formula, i.e., the value of $\chi$ on a boundary $\partial \Sigma$ is equal to the integral of $\omega$ on $\Sigma$ modulo $\Z$.

The description through pairs of the form $(\chi, \omega)$ can be easily generalized to any degree, i.e., we can consider a map $\chi$, defined on smooth $(n-1)$-cycles with values in $\R/\Z$, and a closed $n$-form $\omega$ such that the Stokes-type formula keeps on holding. It is quite easy to construct from these data the corresponding first Chern class, belonging to $H^{n}(X; \Z)$, whose real image is the de-Rham cohomology class of $\omega$ (it follows that $\omega$ has integral periods). A pair of this form is called \emph{Cheeger-Simons character}. This notion, introduced in \cite{CS} in order to find obstructions to conformal embeddings of Riemannian manifolds into the Euclidean space, provides a model for the differential refinement of singular cohomology. Another interesting model is provided by the Deligne cohomology \cite{Brylinski}. In this case the differential refinement is implemented considering the set of transition functions of a line bundle, thought of as a $\rm\check{C}$ech cocyle, and adding the local potentials that describe the connection. This model turns out to be isomorphic to the previous one. A more systematic approach towards a theory of differential refinements of singular cohomology was developed by Simons and Sullivan in \cite{SS}. There, the authors proposed a suitable axiomatic framework and settled the problem of uniqueness of the extension.

Beyond singular cohomology, the differential extension of K-theory has been studied intensively too (see \cite{FL, BS2, SS2, BS3} and many others). In the more general setting of differential refinements of arbitrary cohomology theories, foundational work was laid down by Hopkins and Singer in \cite{HS}, where the authors proposed a model that produces a differential extension out of any cohomology theory represented by a spectrum. Such a model has been completed, describing in detail the $S^{1}$-integration and the multiplicative structure, in \cite{Upmeier}. A suitable axiomatic description of this general notion of differential refinement can be found in \cite{BS}, where Bunke and Schick proved, under rather mild hypotheses, that the extension of a fixed theory is unique.

Since a cohomology theory is usually defined on pairs of spaces, it is natural to construct also the relative version of a differential refinement. About singular cohomology, in the paper \cite{BT} the authors introduced two different definitions of relative Cheeger-Simons character. Such constructions and their main properties are described in detail in \cite{BB}. In the case of Deligne cohomology, the analogous definitions have been summarized in \cite{FR3} and generalized to any cohomology theory in \cite{FR2}, adapting to the relative case the Hopkins-Singer model.\footnote{The case of parallel relative classes had already been considered in \cite{Upmeier}.} Moreover, in \cite{FR2} we have shown how to construct some families of long exact sequences naturally associated to a differential extension.

In this paper we generalize to the relative case the axiomatic framework for differential cohomology, deducing the long exact sequences directly from the axioms. Moreover, adapting the technique of Bunke and Schick, we show that, under the same hypotheses of \cite{BS}, the extension is unique. After that, we generalize to any cohomology theory the notion of Cheeger-Simons character, extending to the relative case the construction shown in \cite{FR}. Finally, we define the integration map for bundles whose fibre has a boundary.

The paper is organized as follows. In section \ref{AxiomsDC} we state the axioms for relative differential cohomology and the first properties that can be deduced from them. In section \ref{LongExSeqSec} we construct the corresponding long exact sequences and we prove the exactness in each position. In section \ref{ExUniqSec} we show that, under the same hypotheses of \cite{BS}, there exists a unique relative differential extension of a fixed cohomology theory. In section \ref{ComplementsSec} we define the differential extension with compact support and we show how to construct a long exact sequence completely made by differential cohomology groups. In section \ref{SecOrientInt} we start introducing the material that will be necessary to construct the relative generalized Cheeger-Simons characters. In particular, we recall the notion of differential orientation (for maps and manifolds) and the construction of the integration map. In section \ref{RelIntSec} we define the integration map for classes relative to the boundary and we show how to integrate such a relative class to the point. In section \ref{SecFlatPCS} we define the relative Cheeger-Simons characters, starting from the flat case. Finally, in section \ref{SecIntBd} we show how to define the (relative) integration map for bundles whose fibres have non-empty boundary.

\section{Axioms for relative differential cohomology}\label{AxiomsDC}

We are going to state the axioms of relative differential cohomology. When we use the expression ``relative cohomology'', we mean that we are considering the cohomology groups of any map of spaces, not necessarily an embedding. Thus, we start with a brief review of the axioms of (topological) cohomology for maps.

\SkipTocEntry \subsection{Relative cohomology}\label{RelCohomSec}

Let $\CC$ be the category of spaces with the homotopy type of a CW-complex or of a finite CW-complex. We call $\CC_{+}$ the category whose objects are the ones of $\CC$ with a marked point, and whose morphisms are the continuous functions that respect the marked points. Taking the quotient of the morphisms of $\CC$ and $\CC_{+}$ up to homotopy (relative to the marked point in $\CC_{+}$), we get the categories $\HH\CC$ and $\HH\CC_{+}$. Moreover, we denote by $\CC_{2}$ the category of morphism of $\CC$, defined in the following way:
\begin{itemize}
	\item an object of $\mathcal{C}_{2}$ is a morphism $\rho \colon A \to X$ of $\CC$ (i.e.\ a continuous function between objects of $\CC$);
	\item given two objects $\eta \colon B \to Y$ and $\rho \colon A \to X$, a morphism from $\eta$ to $\rho$ is a pair of continuous functions $f \colon Y \to X$ and $g \colon B \to A$, making the following diagram commutative:
	\begin{equation}\label{MorphismDiagram}
		\xymatrix{
		B \ar[r]^{\eta} \ar[d]_{g} & Y \ar[d]^{f} \\
		A \ar[r]^{\rho} & X.
	}\end{equation}
\end{itemize}
We set $I := [0, 1]$ and we call $\id_{I} \colon I \to I$ the identity map. A \emph{homotopy} between two morphisms $(f_{0}, g_{0}), (f_{1}, g_{1}) \colon \eta \to \rho$ is a morphism $(F, G) \colon \eta \times \id_{I} \to \rho$, such that, for $i = 0, 1$, we have $(F\vert_{X \times \{i\}}, G\vert_{A \times \{i\}}) = (f_{i}, g_{i})$. Taking the quotient of the morphisms of $\CC_{2}$ by homotopy we define the category $\HH\CC_{2}$. There are the following natural embeddings of categories:
\begin{itemize}
	\item $\CC \hookrightarrow \CC_{+}$ and $\HH\CC \hookrightarrow \HH\CC_{+}$, defined identifying an object $X$ with $(X_{+}, \infty)$, where $X_{+} = X \sqcup \{\infty\}$;
	\item $\CC_{+} \hookrightarrow \CC_{2}$ and $\HH\CC_{+} \hookrightarrow \HH\CC_{2}$, defined identifying the object $(X, x_{0})$ with the morphism $\rho \colon pt \to X$ such that $\rho(pt) = x_{0}$;
	\item by composition, we get the embeddings $\CC \hookrightarrow \CC_{2}$ and $\HH\CC \hookrightarrow \HH\CC_{2}$; we can also define these embeddings identifying $X$ with the empty function $\emptyset \to X$, if we consider the empty set as a manifold.
\end{itemize}
Finally, there are two natural functors $\Pi \colon \CC_{2} \to \CC$ and $\Pi \colon \HH\CC_{2} \to \HH\CC$, defined in the following way: if $\rho \colon A \to X$ is an object, then $\Pi(\rho) = A$; if $\eta \colon B \to Y$ is another object and $(f, g) \colon \rho \to \eta$ is a morphism, then $\Pi(f, g) = g$.

We call $\A_{\Z}$ the category of $\Z$-graded abelian groups. A \emph{cohomology theory} on $\CC_{2}$ is defined by a functor $h^{\bullet} \colon \HH\CC_{2} \to \A_{\Z}$ and a morphism of functors $\beta^{\bullet} \colon h^{\bullet} \circ \Pi \to h^{\bullet+1}$, satisfying the following axioms:
\begin{enumerate}
	\item \emph{Long exact sequence:} the functor $h^{\bullet}$ and the morphism of functors $\beta^{\bullet}$ define a functor from $\HH\CC_{2}$ to the category of long exact sequences of abelian groups, that assigns to an object $\rho \colon A \to X$ the sequence:
		\[\xymatrix{
		\cdots \ar[r] & h^{n}(\rho) \ar[r]^(.48){\pi^{*}} & h^{n}(X) \ar[r]^{\rho^{*}} & h^{n}(A) \ar[r]^(.45){\beta} & h^{n+1}(\rho) \ar[r] & \cdots
	}\]
($\pi$ being the natural morphism from $\emptyset \to X$ to $\rho \colon A \to X$) and to a morphism the corresponding morphism of exact sequences.
	\item \emph{Excision:} if $i \colon Z \hookrightarrow A$ and $j \colon A \hookrightarrow X$ are embeddings such that the closure of $j(i(Z))$ is contained in the interior of $j(A)$, then the morphism
		\[\xymatrix{
		A \setminus i(Z) \ar[rr]^(0.45){j'} \ar[d]_{\iota'} & & X \setminus j(i(Z)) \ar[d]^{\iota} \\
		A \ar[rr]^{j} & & X
	}\]
	induces an isomorphism between $h^{\bullet}(j)$ and $h^{\bullet}(j')$.
\end{enumerate}
If the objects of $\CC$ have the homotopy type of a finite CW-complex this is enough, otherwise we must add the multiplicativity axiom \cite{Milnor}.

Such a definition of cohomology theory is equivalent to the usual one on pairs of spaces or on spaces with a marked point. In fact, starting from a reduced cohomology theory on $\HH\CC_{+}$, the cohomology groups of a morphism $\rho \colon A \to X$ are defined as the reduced ones of the cone $C(\rho) := X \sqcup_{A} CA$, and the axioms are satisfied. Vice-versa, if we start from the axioms on the category $\HH\CC_{2}$, we can prove that $h^{\bullet}(\rho)$ is naturally isomorphic to $\tilde{h}^{\bullet}(C(\rho))$, hence the theory is the unique possible extension to $\HH\CC_{2}$ of a reduced cohomology theory on $\HH\CC_{+}$. In fact, we consider the cylinder $\Cyl(\rho) := X \sqcup_{A} \Cyl(A)$ and the following commutative diagram:
\begin{equation}\label{DiagCylCone}
	\xymatrix{
	\{*\} \ar@{^(->}[r]^{v} & C(\rho) \\
	A \ar@{^(->}[r]^(.42){i_{1}} \ar[u]^{p'} \ar[d]_{\id} & \Cyl(\rho) \ar[u]_{p} \ar[d]^{\pi} \\
	A \ar[r]^{\rho} & X.
}
\end{equation}
The point $\{*\}$ is the vertex of the cone. The projection $\pi$ shrinks the cylinder of $A$ on the base and the projection $p$ collapses the upper base of the cylinder to the vertex of the cone. Finally, the embedding $i_{1}$ sends $A$ to the upper base of the cylinder. Since $i_{1}$ is a cofibration, the projection $p$, that collapses $A$ to a point, induces an isomorphism in relative cohomology $h^{\bullet}(i_{1}) \simeq \tilde{h}^{\bullet}(C(\rho))$. Moreover, since $\pi$ is a homotopy equivalence, the pair $(\pi, \id)$ induces, by the five lemma applied to the corresponding long exact sequences, an isomorphism $h^{\bullet}(i_{1}) \simeq h^{\bullet}(\rho)$. Composing the two isomorphisms we get $h^{\bullet}(\rho) \simeq \tilde{h}^{\bullet}(C(\rho))$. Such an isomorphism is natural. In fact, given two maps $\rho \colon A \to X$ and $\eta \colon B \to Y$ and a morphism $(k, h) \colon \rho \to \eta$, from the induced morphism between the two diagrams \eqref{DiagCylCone} of $\rho$ and $\eta$, we see that the following diagram commutes:
\begin{equation}\label{IsoConeNatural}
\xymatrix{
	h^{\bullet}(\eta) \ar[rr]^{(k, h)^{*}} \ar[d]_{\simeq} & & h^{\bullet}(\rho) \ar[d]^{\simeq} \\
	\tilde{h}^{\bullet}(C(\eta)) \ar[rr]^{C(k, h)^{*}} & & \tilde{h}^{\bullet}(C(\rho)).
} \end{equation}
In particular, if $C(k, h)$ is a homotopy equivalence, then $(k, h)^{*}$ is an isomorphism, even if $(k, h)$ is not a homotopy equivalence in the category $\CC_{2}$.

In order to introduce products, we call $\mathcal{R}_{\Z}$ the category of $\Z$-graded commutative rings. There is a natural forgetful functor $\mathcal{R}_{\Z} \to \mathcal{A}_{\Z}$, that we apply when needed, without writing it explicitly. The cohomology theory $h^{\bullet}$ is called \emph{multiplicative} if it can be refined to a functor $h^{\bullet} \colon \HH\CC_{2} \to \mathcal{R}_{\Z}$ , in such a way that the product satisfies a suitable compatibility condition with the morphisms $\beta^{\bullet}$. The isomorphism $h^{\bullet}(\rho) \simeq \tilde{h}^{\bullet}(C(\rho))$ is a \emph{ring} isomorphism, hence the product in relative cohomology is canonically induced by the one on the corresponding reduced cohomology theory.

Finally, given a morphism $\rho \colon A \to X$, the group $h^{\bullet}(\rho)$ has a natural right module structure over $h^{\bullet}(X)$:
\begin{equation}\label{AbsRelMod}
	\cdot \, \colon h^{\bullet}(\rho) \otimes_{\Z} h^{\bullet}(X) \to h^{\bullet}(\rho)
\end{equation}
defined as follows. We compute the product $h^{\bullet}(\rho) \otimes h^{\bullet}(X) \simeq \tilde{h}^{\bullet}(C(\rho)) \otimes \tilde{h}^{\bullet}(X_{+}) \to \tilde{h}^{\bullet}(C(\rho) \wedge X_{+}) \simeq \tilde{h}^{\bullet}(C(\rho \times \id_{X})) \simeq h^{\bullet}(\rho \times \id_{X})$. Then we apply the pull-back via the diagonal morphism
	\[\xymatrix{
		A \ar[rr]^{\rho} \ar[d]_{(\id_{A}, \rho)} & & X \ar[d]^{\Delta_{X}} \\
		A \times X \ar[rr]^{\rho \times \id_{X}} & & X \times X.
}\]
We could construct \eqref{AbsRelMod} directly from the axioms, without passing through the cone of $\rho$, but it would be a little bit longer.

\SkipTocEntry \subsection{Fibre-wise integration and Stokes theorem}

Given a smooth map of manifolds $\rho \colon A \to X$, we call $\Omega^{\bullet}(\rho)$ the cochain complex $\Omega^{\bullet}(X) \oplus \Omega^{\bullet-1}(A)$ with coboundary $d(\omega, \eta) = (d\omega, \rho^{*}\omega - d\eta)$. We get the following short exact sequence of chain complexes:
\begin{equation}\label{ShortSeqForms}
	\xymatrix{
	0 \ar[r] & (\Omega^{\bullet-1}(A), -d^{\bullet-1}) \ar[r]^(.57){i} & (\Omega^{\bullet}(\rho), d^{\bullet}) \ar[r]^(.48){\pi} & (\Omega^{\bullet}(X), d^{\bullet}) \ar[r] & 0,
}
\end{equation}
where $i(\eta) = (0, \eta)$ and $\pi(\omega, \eta) = \omega$. The complex $\Omega^{\bullet}(\rho)$ has a natural right module structure over $\Omega^{\bullet}(X)$, defined by:
\begin{equation}\label{ModWedge}
	(\omega, \eta) \wedge \xi := (\omega \wedge \xi, \eta \wedge \rho^{*}\xi).
\end{equation}
We get correctly that $d((\omega, \eta) \wedge \xi) = d(\omega, \eta) \wedge \xi + (-1)^{\abs{\omega}} (\omega, \eta) \wedge d\xi$.

Let us fix the following data:
\begin{itemize}
	\item a smooth map $\rho \colon A \to X$ between compact manifolds, possibly with boundary;
	\item two fibre bundles $f \colon Y \to X$ and $g \colon B \to A$ with $n$-dimensional compact oriented fibres, possibly with boundary;
	\item a morphism of fibre bundles $\tilde{\rho} \colon B \to Y$ covering $\rho$ and inducing a diffeomorphism in each fibre;\footnote{Such a morphism is equivalent to a bundle isomorphism between $B$ and $\rho^{*}Y$.}
	\item an orientation of the bundle $f$, inducing an orientation of $g$.
\end{itemize}
The map $\rho$ is not necessarily neat (e.g., it can be the embedding of $\partial X$ in $X$). As well, $f$ and $g$ are not required to be neat (surely they are not when the fibres have non-empty boundary). This implies that neither $\tilde{\rho}$ is neat in general, but it is in each fibre, since it is a diffeomorphism. We get the following diagram:
\begin{equation}\label{DiagramFibInt}
\xymatrix{
	B \ar[r]^{\tilde{\rho}} \ar[d]_{g} & Y \ar[d]^{f} \\
	A \ar[r]^{\rho} & X.
}
\end{equation}

\begin{Not}\label{NotIntegral} Considering for example the fibration $f \colon Y \to X$, we denote the fibre-wise integration of a differential form $\omega \in \Omega^{\bullet}(Y)$ by $\int_{Y/X} \omega$, with the usual convention $\bigl(\int_{Y/X} \omega\bigr)_{x}(v_{1}, \ldots, v_{p}) := \int_{Y_{x}}\omega(\tilde{v}_{1}, \ldots, \tilde{v}_{p}, \,\cdot\,, \ldots, \,\cdot\,)$, where $df(\tilde{v}_{i}) = v_{i}$. We denote by $\tint_{Y/X} \omega$ the integration with the opposite convention, i.e., $\bigl(\tint_{Y/X} \omega\bigr)_{x}(v_{1}, \ldots, v_{p}) := \int_{Y_{x}}\omega(\,\cdot\,, \ldots, \,\cdot\,, \tilde{v}_{1}, \ldots, \tilde{v}_{p})$, hence $\tint_{Y/X} \omega = (-1)^{n(\abs{\omega} - 1)} \int_{Y/X} \omega$,\footnote{The sign should be $n(\abs{\omega} - n)$, but $n$ and $n^{2}$ have the same parity.} $n$ being the dimension of the fibres. In particular, if $Y = X \times F$ (or $Y = F \times X$) is the trivial fibration, $\xi$ is a form on $X$ and $\vol_{F}$ is a volume form on $F$, then $\int_{X \times F/X} \xi \wedge \vol_{F} = \xi$, while $\tint_{X \times F/X} \vol_{F} \wedge \xi  = \xi$. It follows that, when the fibre has no boundary, $\int_{Y/X}d\omega = d\int_{Y/X}\omega$, while $\tint_{Y/X}d\omega = (-1)^{n}d\tint_{Y/X}\omega$.
\end{Not}

We define the fibre-wise integration of a relative form $(\omega, \eta) \in \Omega^{\bullet}(\tilde{\rho})$ in the following way, depending on the convention:
\begin{align}
	& \label{FiberInt1} \int_{\tilde{\rho}/\rho} (\omega, \eta) := \left( \int_{Y/X} \omega, \int_{B/A} \eta \right) \\
	& \label{FiberInt2} \tint_{\tilde{\rho}/\rho} (\omega, \eta) := \left( \tint_{Y/X} \omega, (-1)^{n} \tint_{B/A} \eta \right).
\end{align}
It follows that $\tint_{\tilde{\rho}/\rho} (\omega, \eta) = (-1)^{n(\abs{\omega} - 1)} \int_{\tilde{\rho}/\rho} (\omega, \eta)$. If the fibres of $f$ and $g$ have boundary, we call $\partial f \colon Y' \to X$ and $\partial g \colon B' \to A$ the fibre bundles induced by the restrictions of $f$ and $g$ to the union of the boundaries of the fibres (if $A$ and $X$ have no boundary, then $Y' = \partial Y$ and $B' = \partial B$); moreover, we call $\partial\tilde{\rho} \colon B' \to Y'$ the corresponding restriction of $\tilde{\rho}$. We get a diagram analogous to \eqref{DiagramFibInt}. The following relative version of Stokes theorem holds \cite[formula (82) p.\ 165]{BB}:
\begin{align}
	& \label{RelativeStokes1} d\int_{\tilde{\rho}/\rho} (\omega, \eta) = \int_{\tilde{\rho}/\rho} d(\omega, \eta) + (-1)^{\abs{\omega} + n} \int_{\partial \tilde{\rho}/\rho} (\omega, \eta) \\
	& \label{RelativeStokes2} (-1)^{n}d\tint_{\tilde{\rho}/\rho} (\omega, \eta) = \tint_{\tilde{\rho}/\rho} d(\omega, \eta) - \tint_{\partial \tilde{\rho}/\rho} (\omega, \eta).
\end{align}

\SkipTocEntry \subsection{Differential extension}\label{RelDiffCoh}

Let $\M$ be the category of smooth manifolds or of smooth compact manifolds (even with boundary), and let $\mathcal{A}_{\Z}$ be the category of $\Z$-graded abelian groups. We consider a cohomology theory $h^{\bullet}$, defined on a category including $\M$. We use the following notation:
	\[\h^{\bullet} := h^{\bullet}(\{pt\}) \qquad\qquad \h^{\bullet}_{\R} := \h^{\bullet} \otimes_{\Z} \R.
\]
We consider the category $\M_{2}$ of morphisms of $\M$. For any object $\rho \colon A \to X$ of $\mathcal{M}_{2}$, we call $\ch \colon h^{\bullet}(\rho) \to H^{\bullet}_{\dR}(\rho; \h^{\bullet}_{\R})$ the generalized Chern character \cite[sec.\ 4.8 p.\ 383]{HS}.\footnote{In \cite{HS} only the absolute Chern character is discussed, using the language of spectra. Anyway, since the spaces defining a spectrum are pointed (in order to define the structure maps), we can easily define the Chern character in reduced cohomology. Considering the reduced cohomology of the cone, we get the relative Chern character.} We follow \cite[sec.\ 1]{BS}, adapting the construction to the relative case.
\begin{Def}\label{RelDiffExt} A \emph{relative differential extension} of $h^{\bullet}$ is a functor $\hat{h}^{\bullet} \colon \mathcal{M}_{2}^{\op} \to \mathcal{A}_{\Z}$, together with the following natural transformations of $\mathcal{A}_{\Z}$-valued functors:
\begin{itemize}
	\item $I \colon \hat{h}^{\bullet}(\rho) \to h^{\bullet}(\rho)$;
	\item $R \colon \hat{h}^{\bullet}(\rho) \to \Omega_{\cl}^{\bullet}(\rho; \h^{\bullet}_{\R})$, called \emph{curvature};
	\item $a \colon \Omega^{\bullet-1}(\rho; \h^{\bullet}_{\R})/\IIm(d) \to \hat{h}^{\bullet}(\rho)$,
\end{itemize}
such that:
\begin{enumerate}[{$\qquad$}({A}1)]
	\item $R \circ a = d$;
	\item the following diagram is commutative:
	\begin{equation}\label{ComDiagDC}
		\xymatrix{
		\hat{h}^{\bullet}(\rho) \ar[r]^{I} \ar[d]_{R} & h^{\bullet}(\rho) \ar[d]^{\ch} \\
		\Omega_{\cl}^{\bullet}(\rho; \h^{\bullet}_{\R}) \ar[r]^(.46){\dR} & H^{\bullet}_{\dR}(\rho; \h^{\bullet}_{\R});
	}
	\end{equation}
	\item the following sequence is exact:
	\begin{equation}\label{ExSeqDC}
		\xymatrix{
		h^{\bullet-1}(\rho) \ar[r]^(.35){\ch} & \Omega^{\bullet-1}(\rho; \h^{\bullet}_{\R})/\IIm(d) \ar[r]^(.65){a} & \hat{h}^{\bullet}(\rho) \ar[r]^{I} & h^{\bullet}(\rho) \ar[r] & 0;
	} \end{equation}
	\item calling $\cov(\rho)$ the second component of the curvature $R(\rho)$ and $\pi$ the natural morphism from $\emptyset \to X$ to $\rho \colon A \to X$, the following diagram is commutative:
		\[\xymatrix{
		\hat{h}^{\bullet}(\rho) \ar[r]^{\pi^{*}} \ar[d]_{\cov} & \hat{h}^{\bullet}(X) \ar[d]^{\rho^{*}} \\
		\Omega^{\bullet-1}(A) \ar[r]^{a} & \hat{h}^{\bullet}(A).
	}\]
\end{enumerate}
We also call $\hat{h}^{\bullet}$ \emph{relative differential cohomology theory}.
\end{Def}

\begin{Not}\label{NotXA} When $\rho \colon A \hookrightarrow X$ is a closed embedding, we denote $\hat{h}^{\bullet}(\rho)$ also by $\hat{h}^{\bullet}(X, A)$. Restricting to the case $A = \emptyset$, we obtain an absolute differential extension of $h^{\bullet}$ as usually defined in the literature \cite{BS}.
\end{Not}

A class $\hat{\alpha} \in \hat{h}^{n}(\rho)$ is \emph{flat} when $R(\hat{\alpha}) = 0$. Considering flat classes, we get the functor $\hat{h}^{\bullet}_{\fl} \colon \M_{2}^{\op} \to \mathcal{A}_{\Z}$. Thus, we get the following commutative hexagon:
\begin{equation}\label{CommHex1}
	\resizebox{0.9\textwidth}{!}{
	\xymatrix{
	& \Omega^{\bullet-1}(\rho; \h^{\bullet}_{\R})/\IIm(d) \ar[rr]^{d} \ar[dr]^{a} & & \Omega_{\cl}^{\bullet}(\rho; \h^{\bullet}_{\R}) \ar[dr]^{dR} \\
	H^{\bullet-1}_{\dR}(\rho; \h^{\bullet}_{\R}) \ar[ur] \ar[dr]^{a} & & \hat{h}^{\bullet}(\rho) \ar[ur]^{R} \ar[dr]^{I} & & H^{\bullet}_{\dR}(\rho; \h^{\bullet}_{\R}). \\
	& \hat{h}^{\bullet}_{\fl}(\rho) \ar@{^(->}[ur] \ar[rr]^{I} & & h^{\bullet}(\rho) \ar[ur]^{\ch}
}}
\end{equation}

\noindent We will see in section \ref{UniquenessSec} that, under suitable hypotheses (see \cite[Chapter 5]{BS}), we have a natural isomorphism $\hat{h}^{\bullet}_{\fl}(\rho) \simeq h^{\bullet}(\rho; \R/\Z)$. Moreover, we call $\Omega^{\bullet}_{\ch}(\rho)$ the following sub-group of $\Omega^{\bullet}_{\cl}(\rho)$. A closed relative form $(\omega, \eta)$ belongs to $\Omega^{\bullet}_{\ch}(\rho)$ if and only if the cohomology class $[(\omega, \eta)] \in H^{\bullet}_{\dR}(\rho; \h^{\bullet}_{\R})$ lies in the image of the Chern character $\ch \colon h^{\bullet}(\rho) \to H^{\bullet}_{\dR}(\rho; \h^{\bullet}_{\R})$.
\begin{Lemma} The group $\Omega^{\bullet}_{\ch}(\rho)$ is the image of the curvature functor $R$, thus we have the following exact sequence:
\begin{equation}\label{SeqExCurtaCurv}
	\xymatrix{
	0 \ar[r] & \hat{h}^{\bullet}_{\fl}(\rho) \ar[r] & \hat{h}^{\bullet}(\rho) \ar[r]^(.45){R} & \Omega^{\bullet}_{\ch}(\rho) \ar[r] & 0.
}
\end{equation}
\end{Lemma}
\begin{proof} It immediately follows from diagram \eqref{ComDiagDC} that the image of $R$ is contained in $\Omega^{\bullet}_{\ch}(\rho)$. Given a form $(\omega, \eta) \in \Omega^{\bullet}_{\ch}(\rho)$, let $\alpha \in h^{\bullet}(\rho)$ be a class such that $\ch(\alpha) = [(\omega, \eta)]$. Because of the exact sequence \eqref{ExSeqDC}, the morphism $I$ is surjective, hence there exists $\hat{\alpha} \in \hat{h}^{\bullet}(\rho)$ such that $I(\hat{\alpha}) = \alpha$. It follows from diagram \eqref{ComDiagDC} that $[R(\hat{\alpha})] = [(\omega, \eta)]$, thus there exists $(\alpha, \beta) \in \Omega^{\bullet-1}(\rho)$ such that $R(\hat{\alpha}) = (\omega, \eta) + d(\alpha, \beta)$. Then $R(\hat{\alpha} - a(\alpha, \beta)) = R(\hat{\alpha}) - d(\alpha, \beta) = (\omega, \eta)$.
\end{proof}

\begin{Lemma} The following long exact sequence holds:
\begin{equation}\label{ExSeqFlat}
	\xymatrix{
	\cdots \ar[r] & h^{\bullet}(\rho) \ar[r]^(.38){\ch} & H^{\bullet}_{\dR}(\rho; \h^{\bullet}_{\R}) \ar[r]^(.53){a} & \hat{h}^{\bullet+1}_{\fl}(\rho) \ar[r]^{I} & h^{\bullet+1}(\rho) \ar[r] & \cdots.
}
\end{equation}
\end{Lemma}
\begin{proof} It easily follows from the axioms (A1) and (A3) of definition \ref{RelDiffExt}.
\end{proof}

We use the following notation: if $\rho \colon A \to X$ is a map, we set
\begin{equation}\label{RhoI}
	\rho_{I} := \id_{I} \times \rho \colon I \times A \to I \times X.
\end{equation}
The inclusions $i_{0}, i_{1} \colon X \to I \times X$ and $j_{0}, j_{1} \colon A \to I \times A$ induce the following morphisms between $\rho$ and $\rho_{I}$:
	\[\xymatrix{
	A \ar[r]^{\rho} \ar@{^(->}[d]_{j_{0}} & X \ar@{^(->}[d]^{i_{0}} \\ I \times A \ar[r]^{\rho_{I}} & I \times X
} \qquad\qquad \xymatrix{
	A \ar[r]^{\rho} \ar@{^(->}[d]_{j_{1}} & X \ar@{^(->}[d]^{i_{1}} \\ I \times A \ar[r]^{\rho_{I}} & I \times X.
}\]
We set $\iota_{0} := (i_{0}, j_{0})$ and $\iota_{1} := (i_{1}, j_{1})$. Analogously, the projections $\pi_{X} \colon I \times X \to X$ and $\pi_{A} \colon I \times A \to A$ induce the morphism $(\pi_{X}, \pi_{A}) \colon \rho_{I} \to \rho$. We set $\pi := (\pi_{X}, \pi_{A})$.
\begin{Lemma}[Homotopy formula] If $\hat{\alpha} \in \hat{h}^{\bullet}(\rho_{I})$, we have (using formula \eqref{FiberInt2}):
\begin{equation}\label{HomFormula}
	\iota_{1}^{*}\hat{\alpha} - \iota_{0}^{*}\hat{\alpha} = a\left(\tint_{\rho_{I}/\rho} R(\hat{\alpha}) \right).
\end{equation}
\end{Lemma}
\begin{proof} Since $\iota_{0} \circ \pi \colon \rho_{I} \to \rho_{I}$ is homotopic to the identity of $\rho_{I}$ in the category $\mathcal{M}_{2}$, we have that $I(\hat{\alpha}) = \pi^{*}\iota_{0}^{*}I(\hat{\alpha})$. We set $\hat{\alpha}_{0} := \iota_{0}^{*}\hat{\alpha}$, so that $I(\hat{\alpha}) = \pi^{*}I(\hat{\alpha}_{0})$. It follows that $\hat{\alpha} = \pi^{*}(\hat{\alpha}_{0}) + a(\omega, \eta)$, therefore $\iota_{0}^{*}\hat{\alpha} = \hat{\alpha}_{0} + a(\iota_{0}^{*}(\omega, \eta))$ and $\iota_{1}^{*}\hat{\alpha} = \hat{\alpha}_{0} + a(\iota_{1}^{*}(\omega, \eta))$. Hence:
\begin{align*}
	\iota_{1}^{*}\hat{\alpha} - \iota_{0}^{*}\hat{\alpha} & = a\bigl(\iota_{1}^{*}(\omega, \eta) - \iota_{0}^{*}(\omega, \eta)\bigr) = a \left( \tint_{\partial \rho_{I}/\rho} (\omega, \eta) \right) \\
	& \overset{\eqref{RelativeStokes2}}= a \biggl( d \tint_{\rho_{I}/\rho} (\omega, \eta) + \tint_{\rho_{I}/\rho} d(\omega, \eta) \biggr).
\end{align*}
The term $d\tint_{\rho_{I}/\rho} (\omega, \eta)$ can be cut since, applying $a$ to an exact form, we get $0$. Moreover, $R(\hat{\alpha}) = \pi^{*}R(\hat{\alpha}_{0}) + d(\omega, \eta)$ and $\tint_{\rho_{I}/\rho} \pi^{*}R(\hat{\alpha}_{0}) = 0$, hence we get the result.
\end{proof}

\begin{Corollary} Let $\rho \colon A \to X$ and $\eta \colon B \to Y$ be two objects of $\M_{2}$ and let us consider two morphisms $(f_{0}, g_{0}), (f_{1}, g_{1}) \colon \eta \to \rho$. If $(F, G) \colon \id_{I} \times \eta \to \rho$ is a homotopy between $(f_{0}, g_{0})$ and $(f_{1}, g_{1})$, then, for any $\hat{\alpha} \in \hat{h}^{\bullet}(\rho)$, we have:
\begin{equation}\label{HomFormula2}
	(f_{1}, g_{1})^{*}\hat{\alpha} - (f_{0}, g_{0})^{*}\hat{\alpha} = a\left( \tint_{\eta_{I}/\eta} (F, G)^{*}R(\hat{\alpha}) \right).
\end{equation}
\end{Corollary}
\begin{proof} The result follows replacing $\hat{\alpha}$ by $(F, G)^{*}\hat{\alpha}$ in formula \eqref{HomFormula}.
\end{proof}

\begin{Rmk}\label{FlatTheoryProp} Thanks to the previous corollary, the flat theory is a homotopy-invariant functor. From the exact sequence \eqref{ExSeqFlat}, it is easy to prove that it also satisfies excision and multiplicativity. In fact, both hold for $h^{\bullet}$ and $H^{\bullet}_{\dR}$, since they are cohomology theories, thus it is enough to apply the five lemma.
\end{Rmk}

We briefly recall some basic facts about $S^{1}$-integration. Given a space $A$ and fixing a marked point on $S^{1}$, we get a natural embedding $i_{1} \colon A \to S^{1} \times A$ and a natural projection $\pi_{1} \colon S^{1} \times A \to A$. Since $\pi_{1}$ is a retraction with right inverse $i_{1}$, we have the following split exact sequence:
\begin{equation}\label{SplitExactS1}
	\xymatrix{
	0 \ar[r] & h^{\bullet}(i_{1}) \ar[r]^(.42){\pi^{*}} & h^{\bullet}(S^{1} \times A) \ar[r]^(.6){i_{1}^{*}} \ar@/^1pc/[l]^{\xi} & h^{\bullet}(A) \ar[r] \ar@/^1pc/[l]^(.45){\pi_{1}^{*}} & 0.
}
\end{equation}
Here $\pi$ is the natural morphism from $\emptyset \to S^{1} \times A$ to $i_{1}$ and $\xi(\alpha) = (\pi^{*})^{-1}(\alpha - \pi_{1}^{*}i_{1}^{*}\alpha)$. Moreover, we have the following isomorphism:
\begin{equation}\label{SuspIsoS1}
	s \colon h^{\bullet}(i_{1}) \simeq \tilde{h}^{\bullet}(\Sigma(A_{+})) \simeq \tilde{h}^{\bullet-1}(A_{+}) \simeq h^{\bullet-1}(A).
\end{equation}
Thanks to this picture, we can define the following integration map:
\begin{equation}\label{IntS1Cohom}
\begin{split}
	\int_{S^{1}} \colon & h^{\bullet+1}(S^{1} \times A) \to h^{\bullet}(A) \\
	& \alpha \mapsto s \circ \xi(\alpha).
\end{split}
\end{equation}
We also have the ordinary integration map on differential forms $\int_{S^{1}} \colon \Omega^{\bullet+1}(S^{1} \times A) \to \Omega^{\bullet}(A)$. If we apply it to closed forms, we get a well-defined integration map in de-Rham cohomology, coinciding with \eqref{IntS1Cohom}.

A similar construction holds in the relative case. Given a morphism $\rho \colon A \to X$, we set $S\rho := \id_{S^{1}} \times \rho \colon S^{1} \times A \to S^{1} \times X$. Fixing a marked point on $S^{1}$, we get a natural embedding $i_{1} \colon \rho \to S\rho$ and a natural projection $\pi_{1} \colon S\rho \to \rho$. We define the groups $h^{\bullet}(i_{1})$ as follows: we consider the induced embedding $i'_{1} \colon C(\rho) \to C(S\rho)$ and we set $h^{\bullet}(i_{1}) := h^{\bullet}(i'_{1}) \simeq \tilde{h}^{\bullet}(C(S\rho)/C(\rho))$. Since the induced map $\pi'_{1} \colon C(S\rho) \to C(\rho)$ is a retraction with right inverse $i'_{1}$, we have the following split exact sequence:
	\[\xymatrix{
	0 \ar[r] & h^{\bullet}(i_{1}) \ar[r]^{\pi^{*}} & h^{\bullet}(S\rho) \ar[r]^(.53){i_{1}^{*}} \ar@/^1pc/[l]^{\xi} & h^{\bullet}(\rho) \ar[r] \ar@/^1pc/[l]^(.45){\pi_{1}^{*}} & 0.
}\]
Here $\pi \colon C(S\rho) \to C(S\rho)/C(\rho)$ and $\xi(\alpha) = (\pi^{*})^{-1}(\alpha - \pi_{1}^{*}i_{1}^{*}\alpha)$. Moreover, we have the following isomorphism:
	\[s \colon h^{\bullet}(i_{1}) \simeq \tilde{h}^{\bullet}(C(S\rho)/C(\rho)) \simeq \tilde{h}^{\bullet}(\Sigma(C(\rho)) \simeq \tilde{h}^{\bullet-1}(C(\rho)) \simeq h^{\bullet-1}(\rho).
\]
Thanks to this picture, we can define the following integration map:
\begin{equation}\label{IntS1RelCohom}
\begin{split}
	\int_{S^{1}} \colon & h^{\bullet+1}(S\rho) \to h^{\bullet}(\rho) \\
	& \alpha \mapsto s \circ \xi(\alpha).
\end{split}
\end{equation}
We also have the ordinary integration map on differential forms $\int_{S^{1}} \colon \Omega^{\bullet+1}(S\rho) \to \Omega^{\bullet}(\rho)$ defined by $\int_{S^{1}} (\omega, \eta) := \bigl( \int_{S^{1}} \omega, \int_{S^{1}} \eta \bigr)$. If we apply it to closed relative forms, we get a well-defined integration map in de-Rham cohomology, coinciding with the one constructed as \eqref{IntS1RelCohom}. Of course \eqref{IntS1Cohom} is a particular case of \eqref{IntS1RelCohom}.

\begin{Not} Given a functor $\mathcal{F} \colon \mathcal{M}_{2} \to \mathcal{C}$, for any category $\mathcal{C}$, we define the functor $S\mathcal{F} \colon \mathcal{M}_{2} \to \mathcal{C}$ by $S\mathcal{F}(\rho) := \mathcal{F}(S\rho)$ on objects ($S\rho$ being $\id_{S^{1}} \times \rho$) and $S\mathcal{F}(f, g) := \mathcal{F}(Sf, Sg)$ on morphisms. Moreover, given two morphisms $\rho \colon A \to X$ and $\varphi \colon S^{1} \to S^{1}$, we denote by $\varphi_{\#}\rho \colon S\rho \to S\rho$ the morphism $(\varphi \times \id_{X}, \varphi \times \id_{A})$.
\end{Not}

\begin{Def} A \emph{relative differential extension with integration} of $h^{\bullet}$ is a relative differential extension $(\hat{h}^{\bullet}, I, R, a)$ together with a natural transformation:
	\[\int_{S^{1}} \colon S\hat{h}^{\bullet+1} \to \hat{h}^{\bullet}
\]
such that:
\begin{itemize}
	\item $\int_{S^{1}} \circ (t_{\#}\rho)^{*} = -\int_{S^{1}}$, where $t \colon S^{1} \to S^{1}$ is defined by $t(e^{i\theta}) := e^{-i\theta}$;
	\item $\int_{S^{1}} \circ \,\pi_{1}^{*} = 0$, where $\pi_{1} \colon S\rho \to \rho$ is the projection;
	\item the following diagram is commutative:
	\begin{equation}\label{CommS1}
		\resizebox{0.8\textwidth}{!}{
\xymatrix{
		S\Omega^{\bullet}(\rho; \h^{\bullet}_{\R})/\IIm(d) \ar[r]^(.6){a} \ar[d]^{\int_{S^{1}}} & S\hat{h}^{\bullet+1}(\rho) \ar[r]^{I} \ar[d]^{\int_{S^{1}}} \ar@/^2pc/[rr]^{R} & Sh^{\bullet+1}(\rho) \ar[d]^{\int_{S^{1}}} & S\Omega_{\cl}^{\bullet+1}(\rho; \h^{\bullet}_{\R}) \ar[d]^{\int_{S^{1}}} \\
		\Omega^{\bullet-1}(\rho; \h^{\bullet}_{\R})/\IIm(d) \ar[r]^(.65){a} & \hat{h}^{\bullet}(\rho) \ar[r]^{I} \ar@/_2pc/[rr]_{R} & h^{\bullet}(\rho) & \Omega_{\cl}^{\bullet}(\rho; \h^{\bullet}_{\R}),
	}}
	\end{equation}
	where the first and last vertical arrows are defined by $\int_{S^{1}} (\omega, \eta) := \bigl( \int_{S^{1}} \omega, \int_{S^{1}} \eta \bigr)$ and the third one by \eqref{IntS1RelCohom}.
\end{itemize}
\end{Def}

Finally, we introduce products, thus we suppose that $h^{\bullet}$ is a \emph{multiplicative} cohomology theory.
\begin{Def}\label{MultDiffExt} A \emph{multiplicative relative differential extension} of $h^{\bullet}$ is a relative differential extension $(\hat{h}^{\bullet}, I, R, a)$ such that, for any map $\rho \colon A \to X$, there is a natural right $\hat{h}^{\bullet}(X)$-module structure on $\hat{h}^{\bullet}(\rho)$, in such a way that:
\begin{itemize}
	\item $I(\hat{\alpha} \cdot \hat{\beta}) = I(\hat{\alpha}) \cdot I(\hat{\beta})$, using \eqref{AbsRelMod} on the r.h.s.;
	\item $R(\hat{\alpha} \cdot \hat{\beta}) = R(\hat{\alpha}) \wedge R(\hat{\beta})$, using \eqref{ModWedge} on the r.h.s.;
	\item $\hat{\alpha} \cdot a(\omega) = a(R(\hat{\alpha}) \wedge \omega)$ for every $\hat{\alpha} \in \hat{h}^{\bullet}(\rho)$ and $\omega \in \Omega^{\bullet}(X; \h^{\bullet}_{\R})/\IIm(d)$;
	\item $a(\omega, \eta) \cdot \hat{\alpha} = a((\omega, \eta) \wedge R(\hat{\alpha}))$ for every $(\omega, \eta) \in \Omega^{\bullet}(\rho; \h^{\bullet}_{\R})/\IIm(d)$ and $\hat{\alpha} \in \hat{h}^{\bullet}(X)$.
\end{itemize}
\end{Def}

\SkipTocEntry \subsection{Parallel classes} A class $\hat{\alpha} \in \hat{h}^{\bullet}(\rho)$ is called \emph{parallel} if $\cov(\hat{\alpha}) = 0$ (we recall that $\cov$ is the second component of the curvature). We denote by $\hat{h}^{\bullet}_{\ppar}(\rho)$ the sub-group of $\hat{h}^{\bullet}(\rho)$ formed by parallel classes. Moreover, we use the following notation:
\begin{itemize}
	\item $\Omega^{\bullet}_{0}(\rho)$ is the sub-group of $\Omega^{\bullet}(X)$ containing the forms $\omega$ on $X$ such that $\rho^{*}\omega = 0$;
	\item $\Omega^{\bullet}_{\cl,0}(\rho)$ is the intersection between $\Omega^{\bullet}_{0}(\rho)$ and $\Omega^{\bullet}_{\cl}(X)$;
	\item $\Omega^{\bullet}_{\ch,0}(\rho)$ is the subgroup of $\Omega^{\bullet}_{\cl,0}(\rho)$ containing the forms $\omega$ such that the relative cohomology class $[(\omega, 0)] \in H^{\bullet}_{\dR}(\rho)$ belongs to the image of the Chern character.
\end{itemize}
If $(\omega, 0)$ is the curvature of a parallel class, then $\omega \in \Omega^{\bullet}_{\ch,0}(\rho)$. We get the functor $\hat{h}^{\bullet}_{\ppar} \colon \mathcal{M}_{2}^{\op} \to \mathcal{A}_{\Z}$, together with the following natural transformations of $\mathcal{A}_{\Z}$-valued functors:
\begin{itemize}
	\item $I' \colon \hat{h}^{\bullet}_{\ppar}(\rho) \to h^{\bullet}(\rho)$, which is the restriction of the functor $I$;
	\item $R' \colon \hat{h}^{\bullet}_{\ppar}(\rho) \to \Omega_{\cl, 0}(\rho; \h^{\bullet}_{\R})$, which is the first component of the curvature $R$; 
	\item $a' \colon \Omega^{\bullet-1}_{0}(\rho; \h^{\bullet}_{\R})/\IIm(d) \to \hat{h}^{\bullet}_{\ppar}(\rho)$, defined by $a'(\omega) := a(\omega, 0)$.
\end{itemize}
Parallel classes are well-behaved when $\rho$ is a closed embedding. In this case they satisfy four properties analogous to axioms (A1)--(A4) in definition \ref{RelDiffExt}, as the next theorem shows.

\begin{Theorem}\label{AxiomsParallel} Let $\mathcal{M}'_{2}$ be the full sub-category of $\mathcal{M}_{2}$, whose objects are closed embeddings. The functor $\hat{h}^{\bullet}_{\ppar} \colon {\mathcal{M}'_{2}}^{\op} \to \mathcal{A}_{\Z}$ satisfies the statements (A$'$1)--(A$'$4), obtained from axioms (A1)--(A4) in definition \ref{RelDiffExt}, with the following replacements:
\begin{itemize}
	\item $\hat{h}^{\bullet}$ by $\hat{h}^{\bullet}_{\ppar}$;
	\item $I$, $R$ and $a$ by $I'$, $R'$ and $a'$;
	\item $\Omega^{\bullet}$ and $\Omega^{\bullet}_{\cl}$ by $\Omega^{\bullet}_{0}$ and $\Omega^{\bullet}_{\cl, 0}$.
\end{itemize}
In particular, (A$'$4) is the statement $\rho^{*} \circ \pi^{*} = 0$. Moreover, if the functor $\hat{h}^{\bullet}$ admits $S^{1}$-integration or it is multiplicative, the same holds for $\hat{h}^{\bullet}_{\ppar}$, with the analogous axioms. 
\end{Theorem}
\begin{proof} It is easy to show that (A$'$1), (A$'$2) and (A$'$4) are just a particular case of axioms (A1), (A2) and (A4) (actually, they hold even if $\rho$ is not a closed embedding). We only have to prove that (A$'$3) holds. The fact that, in the sequence obtained from \eqref{ExSeqDC}, the composition of two consecutive morphisms vanishes is again a particular case of the general statement. Let us fix $\alpha \in h^{\bullet}(\rho)$. Because of the exactness of \eqref{ExSeqDC}, there exists a class $\hat{\alpha}' \in \hat{h}^{\bullet}(\rho)$ such that $I(\hat{\alpha}') = \alpha$ and $R(\hat{\alpha}') = (\omega, \eta)$. Since $\rho$ is a closed embedding, we can extend $\eta$ to a form $\tilde{\eta}$ on the whole $X$, thus we set $\hat{\alpha} := \hat{\alpha}' - a(\tilde{\eta}, 0)$. Therefore $I(\hat{\alpha}) = I(\hat{\alpha}') - 0 = \alpha$ and $R(\hat{\alpha}) = (\omega, \eta) - (d\tilde{\eta}, \eta) = (\omega - d\tilde{\eta}, 0)$. It follows that $\hat{\alpha} \in \hat{h}^{\bullet}_{\ppar}(\rho)$ and $I(\hat{\alpha}) = \alpha$, hence $I$ is surjective. Let us fix $\hat{\alpha} \in \hat{h}^{\bullet}_{\ppar}(\rho)$ such that $I(\hat{\alpha}) = 0$. Because of the exactness of \eqref{ExSeqDC}, there exists a form $(\theta, \chi) \in \Omega^{\bullet-1}(\rho; \h^{\bullet}_{\R})$ such that $a(\theta, \chi) = \hat{\alpha}$. Since $\rho$ is a closed embedding, we can extend $\chi$ to a form $\tilde{\chi}$ on the whole $X$, thus $\hat{\alpha} = a(\theta, \chi) - a(d(\tilde{\chi}, 0)) = a((\theta, \chi) - (d\tilde{\chi}, \chi)) = a(\theta - d\tilde{\chi}, 0) = a'(\theta - d\tilde{\chi})$. Finally, the exactness in the second position follows from the one of \eqref{ExSeqDC}, since the first group remains unchanged. The axioms of $S^{1}$-integration and multiplicativity easily restricts to parallel classes (even without assuming that $\rho$ is a closed embedding).
\end{proof}

\begin{Theorem} For any smooth map $\rho \colon A \to X$, we have the following short exact sequence:
\begin{equation}\label{SeqExCurtaCurvPar}
	\xymatrix{
	0 \ar[r] & \hat{h}^{\bullet}_{\fl}(\rho) \ar[r] & \hat{h}^{\bullet}_{\ppar}(\rho) \ar[r]^(.45){R'} & \Omega^{\bullet}_{\ch, 0}(\rho) \ar[r] & 0.
}\end{equation}
Moreover, if $\rho$ is a closed embedding, we have the following short exact sequence, which splits non-canonically:
\begin{equation}\label{SeqExCurtaCovPar}
	\xymatrix{
	0 \ar[r] & \hat{h}^{\bullet}_{\ppar}(\rho) \ar[r] & \hat{h}^{\bullet}(\rho) \ar[r]^(.45){\cov} & \Omega^{\bullet-1}(A) \ar[r] & 0.
}\end{equation}
\end{Theorem}
\begin{proof} The exactness of \eqref{SeqExCurtaCurvPar} immediately follows from \eqref{SeqExCurtaCurv}. The exactness of the first three arrows of \eqref{SeqExCurtaCovPar} is trivial by definition of parallel classes. About the last one, given a form $\eta \in \Omega^{\bullet-1}(A)$, we can extend it to a form $\tilde{\eta}$ on $X$ and we have that $\cov(a(\tilde{\eta}, 0)) = \eta$, hence $\cov$ is surjective. In particular, if we fix an extension map $\Omega^{\bullet-1}(A) \to \Omega^{\bullet-1}(X)$, $\eta \mapsto \tilde{\eta}$, through a suitable open cover of $X$ and a corresponding partition of unity, we get the splitting $\Omega^{\bullet-1}(A) \to \hat{h}^{\bullet}(\rho)$, $\eta \mapsto a(\tilde{\eta}, 0)$.
\end{proof}

\begin{Corollary} Let us fix a manifold $Y$ with no boundary and the trivial cobordism $Y \times I$. We denote by $Y_{0} := Y \times \{0\}$ and $Y_{1} := Y \times \{1\}$ the two components of $\partial(Y \times I) \simeq Y \times \partial I$. We have the following canonical isomorphism:
\begin{equation}\label{IsoPsiI}
\begin{split}
  \Psi \colon & \hat{h}^{\bullet}(Y \times I, Y \times \partial I) \overset{\!\simeq}\longrightarrow \hat{h}^{\bullet}_{\ppar}(Y \times I, Y \times \partial I) \oplus \Omega^{\bullet-1}(Y \times \partial I; \h^{\bullet}_{\R}) \\
	& \hat{\alpha} \mapsto \bigl(\hat{\alpha} - a(t\eta_{1} + (1-t)\eta_{0}, 0), \, \eta_{0} \sqcup \eta_{1}\bigr), 
\end{split}
\end{equation}
where $\eta_{0} \sqcup \eta_{1} = \cov(\hat{\alpha})$.
\end{Corollary}
\begin{proof} The inverse isomorphism is defined by
	\[(\hat{\beta}, \eta_{0} \sqcup \eta_{1}) \mapsto \hat{\beta} + a(t\eta_{1} + (1-t)\eta_{0}, 0).
\]
This is due to the fact that, considering the sequence \eqref{SeqExCurtaCovPar} with $X = Y \times I$ and $A = Y \times \partial I$, the map $\Omega^{\bullet-1}(Y \times \partial I; \h^{\bullet}_{\R}) \to \hat{h}^{\bullet}(Y \times I, Y \times \partial I)$, $\eta_{0} \sqcup \eta_{1} \mapsto a(t\eta_{1} + (1-t)\eta_{0}, 0)$ is a \emph{canonical} splitting.
\end{proof}

The next lemma shows that the groups of parallel classes satisfy excision as the topological ones (see \cite[Theorem 3.8]{BBSS} about Cheeger-Simons characters).
\begin{Lemma}\label{ParExcision} If $i \colon Z \hookrightarrow A$ and $j \colon A \hookrightarrow X$ are embeddings such that the closure of $j(i(Z))$ is contained in the interior of $j(A)$, then the morphism
		\[\xymatrix{
		A \setminus i(Z) \ar[rr]^(0.45){j'} \ar[d]_{\iota'} & & X \setminus j(i(Z)) \ar[d]^{\iota} \\
		A \ar[rr]^{j} & & X
	}\]
induces an isomorphism between $\hat{h}^{\bullet}_{\ppar}(j)$ and $\hat{h}^{\bullet}_{\ppar}(j')$.
\end{Lemma}
\begin{proof} The morphism $(\iota, \iota')$ induces the following morphism of exact sequences of the form \eqref{SeqExCurtaCurvPar}:
	\[\xymatrix{
	0 \ar[r] & \hat{h}^{\bullet}_{\fl}(j) \ar[r] \ar[d] & \hat{h}^{\bullet}_{\ppar}(j) \ar[r]^(.45){R} \ar[d] & \Omega^{\bullet}_{\ch, 0}(j) \ar[r] \ar[d] & 0 \\
	0 \ar[r] & \hat{h}^{\bullet}_{\fl}(j') \ar[r] & \hat{h}_{\ppar}^{\bullet}(j') \ar[r]^(.45){R} & \Omega^{\bullet}_{\ch, 0}(j') \ar[r] & 0.
}\]
The left arrow is an isomorphism by the excision property of $\hat{h}^{\bullet}_{\fl}$ (remark \ref{FlatTheoryProp}). We now prove that the right one is an isomorphism too, hence the result follows from the five lemma. We identify $A$ with $j(A)$ and $Z$ with both $i(Z)$ and $j(i(Z))$. The group of closed forms $\Omega^{\bullet}_{\cl,0}(j)$ contains the forms $\omega \in \Omega^{\bullet}_{\cl}(X)$ such that $\omega\vert_{A} = 0$. Similarly, the group of closed forms $\Omega^{\bullet}_{\cl,0}(j')$ contains the forms $\omega \in \Omega^{\bullet}_{\cl}(X \setminus Z)$ such that $\omega\vert_{A \setminus Z} = 0$. It is clear that the pull-back $(\iota, \iota')^{*} \colon \Omega^{\bullet}_{\cl,0}(j) \to \Omega^{\bullet}_{\cl,0}(j')$ is an isomorphism, inducing the excision isomorphism in de-Rham cohomology. We have to show that $(\iota, \iota')^{*}(\Omega^{\bullet}_{\ch, 0}(j)) = \Omega^{\bullet}_{\ch, 0}(j')$. This is a consequence of the commutativity of the following diagram:
	\[\xymatrix{
	h^{\bullet}(j) \ar[r]^(.4){\ch} \ar[d] & H^{\bullet}_{\dR}(j; h^{\bullet}_{\R}) \ar[d] & \Omega^{\bullet}_{\cl,0}(j) \ar[l]_(.42){\dR} \ar[d] \\
	h^{\bullet}(j') \ar[r]^(.4){\ch} & H^{\bullet}_{\dR}(j'; h^{\bullet}_{\R}) & \Omega^{\bullet}_{\cl,0}(j'). \ar[l]_(.42){\dR}
}\]
The left and central vertical arrows are isomorphisms too, by excision.
\end{proof}

The following lemma will be useful in the construction of the long exact sequence of $\hat{h}^{\bullet}$.
\begin{Lemma}\label{LiftCurvS1} Using the notation of sequence \eqref{SplitExactS1}, for every $\hat{\alpha} \in \hat{h}^{\bullet}(A)$ there exists a unique class $\hat{\beta} \in \hat{h}^{\bullet+1}_{\ppar}(i_{1})$ such that $\int_{S^{1}} \pi^{*}\hat{\beta} = \hat{\alpha}$ and $R'(\hat{\beta}) = \pi_{1}^{*}R(\hat{\alpha}) \wedge dt$.
\end{Lemma}
\begin{proof} We set $\alpha := I(\hat{\alpha})$ and, applying the isomorphism \eqref{SuspIsoS1}, $\beta := s^{-1}(\alpha) \in h^{\bullet+1}(i_{1})$. It follows that $\int_{S^{1}} \pi^{*}\beta = s \circ \xi \circ \pi^{*}(\beta) = s(\beta) = \alpha$, the integral being defined by \eqref{IntS1Cohom}. We choose any parallel differential refinement $\hat{\beta}' \in \hat{h}^{\bullet+1}_{\ppar}(i_{1})$ such that $I(\hat{\beta}') = \beta$; this is possible because of property (A$'$3) of theorem \ref{AxiomsParallel} (in particular, because of the surjectivity of $I'$). From the commutativity of diagram \eqref{CommS1}, we get that
\begin{equation}\label{IntS1BetaHat}
	\int_{S^{1}} \pi^{*}\hat{\beta}' = \hat{\alpha} + a(\chi)
\end{equation}
for a suitable form $\chi \in \Omega^{\bullet-1}(A; \h^{\bullet}_{\R})$. We set $R(\hat{\beta}') = (\omega, 0)$ and $R(\hat{\alpha}) = \bar{\omega}$. It follows from \eqref{IntS1BetaHat} that
\begin{equation}\label{IntS1OmegaChi}
	\int_{S^{1}} \omega = \bar{\omega} + d\chi,
\end{equation}
thus, in de-Rham cohomology, $s([\omega, 0]) = [\bar{\omega}]$. Since also $s([\pi_{1}^{*}\bar{\omega} \wedge dt, 0]) = [\bar{\omega}]$ and $s$ is an isomorphism, we have that $[\omega, 0] = [\pi_{1}^{*}\bar{\omega} \wedge dt, 0]$, thus there exists $\nu \in \Omega_{0}^{\bullet}(i_{1})$ such that
\begin{equation}\label{IntS1OmegaNu}
	\omega = \pi_{1}^{*}\bar{\omega} \wedge dt + d\nu, \qquad i_{1}^{*}\nu = 0.
\end{equation}
Joining \eqref{IntS1OmegaChi} and \eqref{IntS1OmegaNu} we see that $d\int_{S^{1}} \nu = d\chi$, thus
\begin{equation}\label{IntS1NuLambda}
	\int_{S^{1}} \nu = \chi + \lambda, \quad d\lambda = 0.
\end{equation}
We set $\hat{\beta} := \hat{\beta}' - a(\nu, 0) + a (\pi_{1}^{*}\lambda \wedge dt, 0)$. It follows that $\int_{S^{1}} \pi^{*}\hat{\beta} = \hat{\alpha}$ and $R(\hat{\beta}) = (\omega, 0) - (d\nu, 0) =  (\pi_{1}^{*}\bar{\omega} \wedge dt, 0)$, as required. The class $\hat{\beta}$ is unique: if we choose another class $\hat{\beta}_{1}$ satisfying the statement, the difference $\hat{\beta} - \hat{\beta}_{1}$ is a flat class $\hat{u} \in \hat{h}_{\fl}^{n}(i_{1})$ such that $\int_{S^{1}} \pi^{*}\hat{u} = 0$. In the topological and de-Rham theories, the map $\int_{S^{1}} \circ \pi^{*}$ is the isomorphism \eqref{SuspIsoS1}, thus, because of the five lemma applied to the sequence \eqref{ExSeqFlat}, it is an isomorphism also in the flat theory. It follows that $\hat{u} = 0$, hence $\hat{\beta} = \hat{\beta}_{1}$.
\end{proof}

\section{Long exact sequences}\label{LongExSeqSec}

Let us fix a differential cohomology theory $\hat{h}^{\bullet}$ with $S^{1}$-integration (it can be multiplicative or not). We now state three theorems concerning the corresponding long exact sequences. The rest of this section will be devoted to the construction of the Bockstein maps and to the proofs.

\begin{Theorem} Considering the flat theory, we have the following exact sequence for every map $\rho \colon A \to X$:
\begin{equation}\label{LongExactFl}
	\cdots \longrightarrow \hat{h}^{\bullet}_{\fl}(\rho) \longrightarrow \hat{h}^{\bullet}_{\fl}(X) \longrightarrow \hat{h}^{\bullet}_{\fl}(A) \longrightarrow \hat{h}^{\bullet+1}_{\fl}(\rho) \longrightarrow \cdots.
\end{equation}
\end{Theorem}
The morphisms involved are $\rho^{*} \colon \hat{h}^{\bullet}_{\fl}(X) \to \hat{h}^{\bullet}_{\fl}(A)$ and $\pi^{*} \colon \hat{h}^{\bullet}_{\fl}(\rho) \to \hat{h}^{\bullet}_{\fl}(X)$, where $\pi$ is the natural morphism from $\emptyset \to X$ to $\rho$, and the Bockstein map $\Bock \colon \hat{h}^{\bullet}_{\fl}(A) \to \hat{h}^{\bullet+1}_{\fl}(\rho)$, that we are going to construct in following paragraphs.

\begin{Rmk} The sequence \eqref{LongExactFl}, together with remark \ref{FlatTheoryProp}, shows that, if $\hat{h}^{\bullet}$ is a relative differential cohomology theory with $S^{1}$-integration, then the flat theory $\hat{h}^{\bullet}_{\fl}$ is a cohomology theory on $\M_{2}$.
\end{Rmk}

\begin{Theorem} Considering the whole groups $\hat{h}^{\bullet}$, we get long exact sequences of the following form:
\begin{equation}\label{LongExact2}
\begin{tikzcd}
\cdots \arrow[r]& \hat{h}^{\bullet-1}_{\fl}(\rho)\arrow[r]
\arrow[d, phantom, ""{coordinate, name=Z}]
& \hat{h}^{\bullet-1}_{\fl}(X) \arrow[r]& \hat{h}^{\bullet-1}(A)\arrow[dll,rounded corners=8pt,curvarr=Z]& 
\\
&\hat{h}^{\bullet}(\rho) \arrow[r]\arrow[d, phantom, ""{coordinate, name=W}] & \hat{h}^{\bullet}(X)\arrow[r]
& h^{\bullet}(A)\arrow[dll,rounded corners=8pt,curvarr=W]& 
\\
& h^{\bullet+1}(\rho) \arrow[r]& h^{\bullet+1}(X)\arrow[r]& h^{\bullet+1}(A)\arrow[r]&\cdots.
\end{tikzcd}
\end{equation}
\end{Theorem}
The first line is a left-infinite part of \eqref{LongExactFl}, except for the last morphism, which is the composition between $\rho^{*} \colon \hat{h}^{\bullet-1}_{\fl}(X) \to \hat{h}^{\bullet-1}_{\fl}(A)$ and the inclusion $\hat{h}^{\bullet-1}_{\fl}(A) \hookrightarrow \hat{h}^{\bullet-1}(A)$. Similarly, from $h^{\bullet}(A)$ on we just have a part of the long exact sequence of the topological theory. The morphism $\hat{h}^{\bullet}(X) \to h^{\bullet}(A)$ is the composition between $\rho^{*} \colon \hat{h}^{\bullet}(X) \to \hat{h}^{\bullet}(A)$ and $I \colon \hat{h}^{\bullet}(A) \to h^{\bullet}(A)$. We will define the Bockstein map $\Bock \colon \hat{h}^{\bullet-1}(A) \to \hat{h}^{\bullet}(\rho)$ in the following. We remark that \eqref{LongExact2} represents a family of exact sequences, since we are free to decide at which degree we put the group $\hat{h}^{\bullet-1}(A)$ instead of the flat one.

\begin{Theorem} Considering parallel classes, if $\rho$ is a \emph{closed embedding} we get long exact sequences of the following form:
\begin{equation}\label{LongExact1}
\begin{tikzcd}
\cdots \arrow[r]& \hat{h}^{\bullet-1}_{\fl}(\rho)\arrow[r]
\arrow[d, phantom, ""{coordinate, name=Z}]
& \hat{h}^{\bullet-1}_{\fl}(X) \arrow[r]& \hat{h}^{\bullet-1}_{\fl}(A) \arrow[dll,rounded corners=8pt,curvarr=Z]& 
\\
&\hat{h}^{\bullet}_{\ppar}(\rho) \arrow[r]\arrow[d, phantom, ""{coordinate, name=W}] & \hat{h}^{\bullet}(X)\arrow[r]
& \hat{h}^{\bullet}(A)\arrow[dll,rounded corners=8pt,curvarr=W]& 
\\
& h^{\bullet+1}(\rho)  \arrow[r]& h^{\bullet+1}(X)\arrow[r]& h^{\bullet+1}(A)\arrow[r]&\cdots.
\end{tikzcd}
\end{equation}
For a generic map $\rho$, we have to stop at $\hat{h}^{\bullet}(A)$ and cut the third line.
\end{Theorem}
The Bockstein map in the first line is the composition of the one of the flat theory with the embedding $\hat{h}^{\bullet}_{\fl}(\rho) \hookrightarrow \hat{h}^{\bullet}_{\ppar}(\rho)$. The Bockstein map in the second line is the composition of the projection $\hat{h}^{\bullet}(A) \to h^{\bullet}(A)$ with the Bockstein map of the sequence of $h^{\bullet}$. \vspace{5pt}

Historically, the first version of a long exact sequence for differential cohomology appeared in \cite[Theorem 3.2]{BT}, but it corresponds to the sequence \eqref{LongExact4} of the present paper, that we will discuss later on. The sequences \eqref{LongExact2} and \eqref{LongExact1} first appeared in \cite[Corollary 70 p.\ 67 and Theorem 15 p.\ 130]{BB} about Cheeger-Simons characters. Moreover, \eqref{LongExact1} is also introduced in \cite[Theorem 2.7 p.\ 8]{Upmeier} about the Hopkins-Singer model and both \eqref{LongExact2} and \eqref{LongExact1} are discussed in \cite{FR3} for Deligne cohomology and in \cite{FR2} for the Hopkins-Singer model. Here we deduce \eqref{LongExactFl}--\eqref{LongExact1} directly from the axioms of relative differential cohomology.

\SkipTocEntry \subsection{Construction of the Bockstein map}\label{ConstrBockSec}

We define the Bockstein map of \eqref{LongExact2} in six steps (S1)--(S6). We use the following notation:
\begin{equation}\label{BarA}
	\bar{a}(\omega, \eta) := (-1)^{\abs{\omega}} a(\omega, \eta) \qquad\qquad \bar{R}(\hat{\alpha}) := (-1)^{\abs{\hat{\alpha}}} R(\hat{\alpha}).
\end{equation}
\begin{enumerate}[{$\qquad$}({S}1)]
	\item Given $\hat{\alpha} \in \hat{h}^{\bullet-1}(A)$, thanks to lemma \ref{LiftCurvS1} there exists a unique class $\hat{\beta} \in \hat{h}^{\bullet}_{\ppar}(i_{1})$ such that $\int_{S^{1}} \pi^{*}\hat{\beta} = \hat{\alpha}$ and $R(\hat{\beta}) = \pi_{1}^{*}R(\hat{\alpha}) \wedge dt$, where $\pi_{1} \colon S^{1} \times A \to A$ is the natural projection and $i_{1} \colon A \hookrightarrow S^{1} \times A$ is the natural embedding, defined marking a point on $S^{1}$.
	\item Embedding $S^{1}$ in $\C$, we suppose that the marked point is $1$. We have a natural projection $\bar{p} \colon (I, \{0, 1\}) \to (S^{1}, \{1\})$, defined by $t \mapsto e^{2\pi i t}$, inducing the projection $p = \bar{p} \times \id_{A} \colon (I \times A, \{0,1\} \times A) \to (S^{1} \times A, \{1\} \times A)$, that can be thought of as a morphism $p \colon i_{0,1} \to i_{1}$ between the embeddings $i_{0,1} \colon A \sqcup A \hookrightarrow I \times A$ and $i_{1} \colon A \hookrightarrow S^{1} \times A$. We get the class $p^{*}\hat{\beta} \in \hat{h}^{\bullet}_{\ppar}(i_{0,1})$.
	\item We define the following class, using \eqref{BarA}:
		\[\hat{\gamma} = p^{*}\hat{\beta} - \bar{a}(t \cdot p^{*}\pi_{1}^{*}R(\hat{\alpha}), 0) \in \hat{h}^{\bullet}(i_{0,1}).
	\]
In the previous expression, within $\bar{a}(\,\cdot\,)$, we denoted by $t$ the coordinate of $I$ and by $p$ the projection $p \colon I \times A \to S^{1} \times A$, defined in the previous step, without the two subspaces.

Since $R(\hat{\beta}) = (\pi_{1}^{*}R(\hat{\alpha}) \wedge dt, 0)$ and $R \circ \bar{a}(t \cdot p^{*}\pi_{1}^{*}R(\hat{\alpha}), 0) = (p^{*}\pi_{1}^{*}R(\hat{\alpha}) \wedge dt, 0 \sqcup \bar{R}(\hat{\alpha}))$, it follows that
	\begin{equation}\label{CurvGammaHat}
		R(\hat{\gamma}) = (0, 0 \sqcup -\bar{R}(\hat{\alpha})).
	\end{equation}
	For $\epsilon = 0, 1$, we call $j_{\epsilon} \colon A \to A \times I$ the embedding with image $A \times \{\epsilon\}$. Moreover, we call $\pi$ the natural morphism from $\emptyset \to I \times A$ to $i_{0,1}$. It follows from axiom (A4) and formula \eqref{CurvGammaHat} that $j_{0}^{*}\pi^{*}\hat{\gamma} = 0$ and $j_{1}^{*}\pi^{*}\hat{\gamma} = -a(\bar{R}(\hat{\alpha}))$.
	\item Let us consider the following morphism:
	\begin{equation}\label{DiagS4}
		\xymatrix{
		A \sqcup A \ar[rr]^{i_{0, 1}} \ar[d]_{\id} & & A \times I \ar[d]^{\pi} \\
		A \sqcup A \ar[rr]^{\id'} & & A
	}\end{equation}
where $\id'$ acts as the identity on both components of the domain. Let us call $\hat{h}^{\bullet}_{0}(\,\cdot\,)$ the sub-group of $\hat{h}^{\bullet}(\,\cdot\,)$ formed by classes such that the first component of the curvature is vanishing. For any map $\rho \colon B \to Y$, we denote by $\Omega^{\bullet}_{\ch'}(\rho)$ the group of closed forms $\eta \in \Omega^{\bullet}(B)$ such that the class $[(0, \eta)] \in H^{\bullet+1}_{\dR}(\rho)$ belongs to the image of the Chern character. From diagram \eqref{DiagS4} we get the following morphism of exact sequences:
	\begin{equation}\label{DiagS42}
	\xymatrix{
	0 \ar[r] & \hat{h}^{\bullet}_{\fl}(i_{0, 1}) \ar[r] & \hat{h}^{\bullet}_{0}(i_{0, 1}) \ar[r]^(.45){\cov} & \Omega^{\bullet-1}_{\ch'}(i_{0, 1}) \ar[r] & 0 \\
	0 \ar[r] & \hat{h}^{\bullet}_{\fl}(\id') \ar[r] \ar[u]_{(\pi, \id)^{*}} & \hat{h}^{\bullet}_{0}(\id') \ar[r]^(.45){\cov} \ar[u]_{(\pi, \id)^{*}} & \Omega^{\bullet-1}_{\ch'}(\id') \ar[r] \ar@{=}[u] & 0.
	}\end{equation} 
We start proving that the left arrow is an isomorphism. Note that the vertical maps of diagram \eqref{DiagS4} are homotopy equivalences, so they induce isomorphisms in the topological theory $h^{\bullet}$. Thus, using the long exact sequences associated to the horizontal maps and applying the five lemma, we see that $(\pi, \id)^{*} \colon h^{\bullet}(\id') \to h^{\bullet}(i_{0, 1})$ is an isomorphism too. The same holds about the de-Rham theory, hence, applying again the five lemma to the exact sequence \eqref{ExSeqFlat}, we conclude that $(\pi, \id)^{*} \colon \hat{h}^{\bullet}_{\fl}(\id') \to \hat{h}^{\bullet}_{\fl}(i_{0, 1})$ is an isomorphism. The right arrow of diagram \eqref{DiagS42} is an isomorphism too, since it is an equality, therefore, applying again the five lemma, we deduce that also the central arrow is an isomorphism. Thus, we get a unique class $\hat{\delta} \in \hat{h}^{\bullet}_{0}(\id')$ whose pull-back is $\hat{\gamma}$. By construction $R(\hat{\delta}) = (0, 0 \sqcup -\bar{R}(\hat{\alpha}))$, thus, if we pull $\hat{\delta}$ back to $A \sqcup A$, by axiom (A4) it vanishes on the first component of $A \sqcup A$.
	\item Let us consider the following morphism:
	\begin{equation}\label{DiagS5}
	\xymatrix{
		A \sqcup A \ar[rr]^{\id'} \ar[d]_{\rho''} & & A \ar[d]^{\rho} \\
		X \sqcup A \ar[rr]^{\rho'} & & X. \\
	}\end{equation}
The morphism $\rho'$ acts as the identity on the first component of $X \sqcup A$ and as $\rho$ on the second. The morphism $\rho''$ is defined by $\rho''(a_{1} \sqcup a_{2}) := \rho(a_{1}) \sqcup a_{2}$. Let us call $\hat{h}^{\bullet}_{1}(\id')$ the sub-group of $\hat{h}^{\bullet}_{0}(\id')$ formed by classes whose curvature is of the form $[(0, 0 \sqcup \eta)]$. Similarly, let us call $\hat{h}^{\bullet}_{1}(\rho')$ the sub-group of $\hat{h}^{\bullet}_{0}(\rho')$ formed by classes whose curvature is of the form $[(0, 0 \sqcup \eta)]$. In both cases, we call $\Omega^{\bullet}_{\ch'}(A)$ the group of closed forms $\eta$ on $A$ such that the class $[(0, 0 \sqcup \eta)]$ belongs to the image of the Chern character. We get the following morphism of exact sequences:
	\begin{equation}\label{DiagS52}
	\xymatrix{
	0 \ar[r] & \hat{h}^{\bullet}_{\fl}(\id') \ar[r] & \hat{h}^{\bullet}_{1}(\id') \ar[r]^(.47){R} & \Omega^{\bullet-1}_{\ch'}(A) \ar[r] & 0 \\
	0 \ar[r] & \hat{h}^{\bullet}_{\fl}(\rho') \ar[r] \ar[u]_{(\rho, \rho'')^{*}} & \hat{h}^{\bullet}_{1}(\rho') \ar[r]^(.45){R} \ar[u]_{(\rho, \rho'')^{*}} & \Omega^{\bullet-1}_{\ch'}(A) \ar[r] \ar@{=}[u] & 0.
	}\end{equation}
The left arrow is an isomorphism. In fact, let us consider the mapping cones $C(\id')$ and $C(\rho'')$. The embeddings $CA \hookrightarrow C(\id')$ ($A$ being the first component of $A \sqcup A$) and $CX \hookrightarrow C(\rho')$ are cofibrations, because their images are a deformation retract of a neighbourhood. Thus, collapsing $CA$ and $CX$ to a point, we see that both $C(\id')$ and $C(\rho')$ are homotopically equivalent to the suspension $\Sigma(A_{+})$, i.e., to the double cone of $A$ with the two vertices identified. We get the following commutative diagram:
	\begin{equation}\label{DiagS53}
	\xymatrix{
	C(\id') \ar[r]^(.42){\simeq} \ar[d]_{(\rho, \rho'')} & C(\id')/CA\ar[d]^{\approx}\\
	C(\rho') \ar[r]^(.42){\simeq} & C(\rho')/CX
}\end{equation}
where `$\simeq$' denotes a homotopy equivalence and `$\approx$' an homeomorphism. This implies that $(\rho, \rho'')^{*} \colon \tilde{h}^{\bullet}(C(\rho')) \to \tilde{h}^{\bullet}(C(\id'))$ is an isomorphism, being the composition of three isomorphisms, therefore $(\rho, \rho'')^{*} \colon h^{\bullet}(\rho') \to h^{\bullet}(\id')$ is an isomorphism too. The same holds for the de-Rham cohomology, thus, applying the five lemma to the exact sequence \eqref{ExSeqFlat}, we see that the left arrow of diagram \eqref{DiagS52} is an isomorphism. The right arrow of \eqref{DiagS52} is an isomorphism too, since it is an equality, hence, again because of the five lemma, the central arrow of diagram \eqref{DiagS52} is an isomorphism as well. Therefore, we get a unique class $\hat{\varepsilon} \in \hat{h}^{\bullet}_{1}(\rho')$, whose pull-back is $\hat{\delta}$. By construction $R(\hat{\varepsilon}) = (0, 0 \sqcup -\bar{R}(\hat{\alpha}))$.
	\item Finally, let us consider the following morphism:
		\[\xymatrix{
		A \ar[rr]^{\rho} \ar[d]_{\id} & & X \ar[d]^{\id} \\
		X \sqcup A \ar[rr]^{\rho'} & & X. \\
	}\]
	The pull-back of $\hat{\varepsilon}$ is a class $\hat{\mu} \in \hat{h}^{\bullet}(\rho)$ and we set $\Bock^{\bullet-1}(\hat{\alpha}) := (-1)^{\abs{\hat{\alpha}}}\hat{\mu}$. It follows that $R(\hat{\mu}) = (0, -R(\hat{\alpha}))$.
\end{enumerate}
This completes the construction of the sequence \eqref{LongExact2}. By construction, we have that:
\begin{equation}\label{CurvBockstein}
	R \circ \Bock(\hat{\alpha}) = (0, -R(\hat{\alpha})).
\end{equation}
The next lemma shows the behaviour of the Bockstein map on topologically trivial classes.
\begin{Lemma} The following formula holds:
\begin{equation}\label{ABockstein}
	\Bock \circ a(\eta) = a(0, \eta).
\end{equation}
\end{Lemma}
\begin{proof} Let us set $\hat{\alpha} = a(\eta)$ in step (S1). It follows that $\hat{\beta} = a(\pi_{1}^{*}\eta \wedge dt)$. In step (S3), setting $p_{1} := \pi_{1} \circ p$, we get:
\begin{align*}
	\hat{\gamma} &= a(p_{1}^{*}\eta \wedge dt, 0) - \bar{a}(t \cdot p_{1}^{*}d\eta, 0) = (-1)^{\abs{\eta}} a(dt \wedge p_{1}^{*}\eta + t \cdot p_{1}^{*}d\eta, 0) \\
	& = (-1)^{\abs{\eta}} a(d(t \cdot p_{1}^{*}\eta), 0) \overset{(\star)}= (-1)^{\abs{\eta}+1} a(0, 0 \sqcup \eta) = \bar{a}(0, 0 \sqcup \eta),
\end{align*}
the equality $(\star)$ being due to the fact that $0 = a \circ d(t \cdot p_{1}^{*}\eta, 0) = a(d(t \cdot p_{1}^{*}\eta), 0 \sqcup \eta) = a(d(t \cdot p_{1}^{*}\eta), 0) + a(0, 0 \sqcup \eta)$. Then, in step (S4), $\hat{\delta} = \bar{a}(0, 0 \sqcup \eta)$ and, in step (S5), $\hat{\varepsilon} = \bar{a}(0, 0 \sqcup \eta)$. Finally, in step (S6), we get $\hat{\mu} = (-1)^{\abs{\eta} + 1}\bar{a}(0, \eta) = a(0, \eta)$.
\end{proof}

\begin{Rmk} Formulas \eqref{CurvBockstein} and \eqref{ABockstein} are coherent with the functoriality of the exact sequence with respect to \eqref{ShortSeqForms}. In fact, the following diagram commutes:
	\[\xymatrix{
	\Omega^{\bullet-1}(A) \ar[r]^(.53){i} \ar[d]_(.45){a} & \Omega^{\bullet}(\rho) \ar[r]^(.48){\pi} \ar[d]^(.45){a} & \Omega^{\bullet}(X) \ar[d]^(.45){a} \\
	\hat{h}^{\bullet}(A) \ar[r]^(.45){\Bock} \ar[d]_(.46){-R} & \hat{h}^{\bullet+1}(\rho) \ar[r]^(.49){\pi^{*}} \ar[d]^(.46){R} & \hat{h}^{\bullet+1}(X) \ar[d]^(.46){R} \\
	\Omega^{\bullet}(A) \ar[r]^(.47){i} & \Omega^{\bullet+1}(\rho) \ar[r]^{\pi} & \Omega^{\bullet+1}(X).
}\]
Moreover, let us suppose that, in formula \eqref{ABockstein}, $d\eta = 0$. Then $a(\eta)$ only depends on the de-Rham cohomology class $[\eta]$, hence we can write $a[\eta]$. The Bockstein map in the (topological) exact sequence of de-Rham cohomology is defined by $\Bock_{\dR}[\eta] = [0, \eta]$, coherently with \eqref{ShortSeqForms}, thus formula \eqref{ABockstein} becomes $\Bock \circ a[\eta] = a \circ \Bock_{\dR}[\eta]$, coherently with the functoriality of the Bockstein map.
\end{Rmk}

If we consider the composition $\hat{h}^{\bullet-1}(X) \to \hat{h}^{\bullet-1}(A) \to \hat{h}^{\bullet}(\rho)$, in general it does not vanish. In fact, in \eqref{LongExact2} only the flat group $\hat{h}^{\bullet-1}_{\fl}(X)$ appears in this segment of the sequence. The next lemma shows the behaviour of the composition.
\begin{Lemma} The following formula holds:
\begin{equation}\label{RhoBockstein}
	\Bock \circ \rho^{*}(\hat{\beta}) = -a(R(\hat{\beta}), 0).
\end{equation}
\end{Lemma}
\begin{proof} Let us consider the following morphism $\rho' := (\id_{X}, \rho) \colon \rho \to \id_{X}$:
	\[\xymatrix{
	A \ar[rr]^{\rho} \ar[d]_{\rho} & & X \ar[d]^{\id_{X}} \\
	X \ar[rr]^{\id_{X}} & & X.
}\]
We get the following diagram:
	\[\xymatrix{
	\hat{h}^{\bullet}(X) \ar[r]^{\rho^{*}} & \hat{h}^{\bullet}(A) \ar[r]^(.45){\Bock} & \hat{h}^{\bullet+1}(\rho) \\
	\hat{h}^{\bullet}(X) \ar[r]^{\id} \ar[u]^{\id} & \hat{h}^{\bullet}(X) \ar[r]^(.42){\Bock'} \ar[u]_{\rho^{*}} & \hat{h}^{\bullet+1}(\id_{X}) \ar[u]_{{\rho'}^{*}}.
}\]
It follows that $\Bock \circ \rho^{*}(\hat{\beta}) = {\rho'}^{*} \circ \Bock'(\hat{\beta})$. Since $h^{\bullet}(\id_{X}) = 0$, because of the sequence \eqref{ExSeqDC} we have that $\hat{h}^{\bullet}(\id_{X}) \simeq \Omega^{\bullet-1}(\rho; \h^{\bullet}_{\R})/\IIm(d)$, hence every element of $\hat{h}^{\bullet}(\id_{X})$ is of the form $a(\omega, \eta)$. Moreover, $(\omega, \eta) - d(\eta, 0) = (\omega, \eta) - (d\eta, \eta) = (\omega - d\eta, 0)$, thus every element of $\hat{h}^{\bullet}(\id_{X})$ is of the form $a(\omega, 0)$, therefore:
\begin{align*}
	& R \circ \Bock'(\hat{\beta}) \,=\, R \circ a(\omega, 0) = (d\omega, \omega) \\
	& R \circ \Bock'(\hat{\beta}) \overset{\eqref{CurvBockstein}}= (0, -R(\hat{\beta})).
\end{align*}
Comparing the second components we get $\omega = -R(\hat{\beta})$, hence $\Bock'(\hat{\beta}) = a(-R(\hat{\beta}), 0)$. It follows that $\Bock \circ \rho^{*}(\hat{\beta}) = {\rho'}^{*} \circ \Bock'(\hat{\beta}) = a(-R(\hat{\beta}), 0)$.
\end{proof}

\begin{Rmk} Let us suppose that, in formula \eqref{RhoBockstein}, $\hat{\beta} = a(\theta)$. Then we get $\Bock \circ a(\rho^{*}\theta) = -a(d\theta, 0)$. Because of formula \eqref{ABockstein} we have $\Bock \circ a(\rho^{*}\theta) = a(0, \rho^{*}\theta)$. The two results are coherent. In fact, $(d\theta, 0) + (0, \rho^{*}\theta) = (d\theta, \rho^{*}\theta) = d(\theta, 0)$, hence, since $a$ vanishes on exact forms, we have that $a(d\theta, 0) + a(0, \rho^{*}\theta) = 0$.
\end{Rmk}

The Bockstein map of \eqref{LongExact2} has been defined. The one of \eqref{LongExactFl} coincides with the one of \eqref{LongExact2}, applied to flat classes; it follows from formula \eqref{CurvBockstein} that the image of a flat class is flat. Finally, in the comments after the sequence \eqref{LongExact1}, we have already shown how to define the corresponding Bockstein maps. It remains to prove the exactness of each sequence.

\SkipTocEntry \subsection{Exactness}

We start from \eqref{LongExactFl}.

\emph{Exactness in $\hat{h}^{n}_{\fl}(X)$.} The fact that $\rho^{*} \circ \pi^{*} = 0$ is an easy consequence of axiom (A4) in definition \ref{RelDiffExt}. Let us consider a class $\hat{\alpha} \in \Ker(\rho^{*})$. We set $\alpha = I(\hat{\alpha})$. Since $\rho^{*}\alpha = 0$, because of the exactness of the topological sequence, there exists a class $\beta \in h^{\bullet}(\rho)$ such that $\pi^{*}\beta = \alpha$. Let $\hat{\beta}' \in \hat{h}^{\bullet}(\rho)$ be any differential class such that $I(\hat{\beta}') = \beta$. It follows that $\pi^{*}\hat{\beta}' = \hat{\alpha} + a(\theta)$, being $\theta \in \Omega^{\bullet-1}(X; \h^{\bullet}_{\R})$, thus $R(\pi^{*}\hat{\beta}') = d\theta$, therefore there exists a \emph{closed} form $\eta \in \Omega^{\bullet-1}(A; \h^{\bullet}_{\R})$ such that $R(\hat{\beta}') = (d\theta, \rho^{*}\theta + \eta)$. Let us prove that the de-Rham class $[\eta]$ belongs to the image of the Chern character. In fact, we have that $\rho^{*}\pi^{*}\hat{\beta}' = \rho^{*}\hat{\alpha} + \rho^{*}a(\theta) = a(\rho^{*}\theta)$ and, by axiom (A4), $\rho^{*}\pi^{*}\hat{\beta}' = a(\cov(\hat{\beta}')) = a(\rho^{*}\theta + \eta) = a(\rho^{*}\theta) + a(\eta)$. It follows that $a(\eta) = 0$, hence $\eta \in \Omega^{\bullet-1}_{\ch}(X; \h^{\bullet}_{\R})$. This implies that there exists $\hat{\gamma} \in \hat{h}^{n-1}(A)$ such that $R(\hat{\gamma}) = \eta$, thus we set $\hat{\beta} := \hat{\beta}' - a(\theta, 0) + \Bock(\hat{\gamma})$. Because of the following remark \ref{BockPi0Both}, we have that $\pi^{*} \circ \Bock = 0$, therefore $\pi^{*}\hat{\beta} = (\hat{\alpha} + a(\theta)) - a(\theta) - 0 = \hat{\alpha}$ and $R(\hat{\beta}) = (d\theta, \rho^{*}\theta + \eta) - (d\theta, \rho^{*}\theta) - (0, \eta) = (0, 0)$.

\emph{Exactness in $\hat{h}^{n}_{\fl}(A)$.} If $\hat{\beta} \in \hat{h}^{\bullet}_{\fl}(X)$, by formula \eqref{RhoBockstein} we have that $\Bock \circ \rho^{*}(\hat{\beta}) = 0$, thus $\Bock \circ \rho^{*} = 0$. Let us consider a flat class $\hat{\alpha} \in \Ker(\Bock)$. Setting $\alpha := I(\hat{\alpha})$, we have that $\Bock(\alpha) = 0$, thus there exists $\beta \in h^{\bullet}(X)$ such that $\alpha = \rho^{*}\beta$. If $\hat{\beta}' \in \hat{h}^{\bullet}(X)$ is any differential refinement of $\beta$, there exists $\theta \in \Omega^{\bullet-1}(X; \h^{\bullet}_{\R})$ such that $\rho^{*}\hat{\beta}' = \hat{\alpha} + a(\theta)$. Applying formula \eqref{ABockstein} we get $\Bock \circ \rho^{*}(\hat{\beta}') = \Bock \circ a(\theta) = a(0, \theta)$ and applying formula \eqref{RhoBockstein} we get $\Bock \circ \rho^{*}(\hat{\beta}') = -a(R(\hat{\beta}', 0))$, thus $a(R(\hat{\beta}'), \theta) = 0$. It follows that $(R(\hat{\beta}'), \theta)$ represents a class belonging to the image of the Chern character, hence there exists a class $\hat{\gamma} \in \hat{h}^{\bullet}(\rho)$ such that $R(\hat{\gamma}) = (R(\hat{\beta}'), \theta)$. We set $\hat{\beta} := \hat{\beta}' - \pi^{*}\hat{\gamma}$. We get that $R(\hat{\beta}) = R(\hat{\beta}') - R(\hat{\beta}') = 0$ and $\rho^{*}\hat{\beta} = \hat{\alpha} + a(\theta) - a(\theta) = \hat{\alpha}$.

\emph{Exactness in $\hat{h}^{n}_{\fl}(\rho)$.} It follows from the construction of the Bockstein map that $\pi^{*} \circ \Bock = 0$. In fact, if $\hat{\mu} = \Bock(\hat{\alpha})$, by the step (S6) we have that $\hat{\mu} = (\id, \id)^{*}\hat{\varepsilon}$. The pull-back $\pi^{*}\hat{\mu}$ coincides with the pull-back of $\hat{\varepsilon}$ via the following composition:
\[\resizebox{!}{5pt}{
\xymatrix{
		\emptyset \ar[rr] \ar[d] & & X \ar[d]^{\id} \\
		\emptyset \ar[rr] \ar[d] & & X \sqcup A \ar[d]^{\id \sqcup \rho} \\
		A \ar[rr]^{\rho} \ar[d]_{\id} & & X \ar[d]^{\id} \\
		X \sqcup A \ar[rr] \ar[rr]^{\id \sqcup \rho} & & X.
}}\]
The last map provides the pull-back from $\hat{\varepsilon}$ to $\hat{\mu}$ and the composition of the first two coincides with $\pi$. The composition of the last two morphism coincides with the following:
	\[\xymatrix{
		\emptyset \ar[rr] \ar[d] & & X \sqcup A \ar[d]^{\id \sqcup \rho} \\
		X \sqcup A \ar[rr] \ar[rr]^{\id \sqcup \rho} & & X.
}\]
By axiom (A4), the pull-back of $\hat{\varepsilon}$ is equal to $a(\cov(\hat{\varepsilon}))$. Since we start from a flat class $\hat{\alpha}$, we have that $\cov(\hat{\varepsilon}) = 0$, thus the pull-back vanishes.
\begin{Rmk}\label{BockPi0Both} Even if $\hat{\alpha}$ is not flat, since $\hat{\varepsilon} \in \hat{h}^{\bullet}_{1}(\id \sqcup \rho)$ by construction, it follows that $\cov(\hat{\varepsilon})$ vanishes on $X$, thus the pull-back to $\emptyset \to X$ is $0$. This proves that $\pi^{*} \circ \Bock = 0$ in \eqref{LongExactFl} and in \eqref{LongExact2}.
\end{Rmk}
Let us consider a flat class $\hat{\mu} \in \Ker(\pi^{*})$. It is enough to prove that there exists a flat class $\hat{\varepsilon} \in \hat{h}^{\bullet}_{\fl}(\id \sqcup \rho)$ such that $(\id, \id)^{*}\hat{\varepsilon} = \hat{\mu}$. In fact, dealing with flat classes, all of the steps (S1)-(S5) consist in the application of an isomorphism, thus, starting from $\hat{\varepsilon}$, we get a class $\hat{\alpha} \in \hat{h}^{\bullet-1}_{\fl}(A)$ such that $\Bock(\hat{\alpha}) = \hat{\mu}$. Applying the steps analogous to (S1)-(S5) to the topological class $\alpha$, we get the Bockstein map of the topological exact sequence, hence, in particular, we get a class $\varepsilon \in h^{\bullet}(\id \sqcup \rho)$ such that $(\id, \id)^{*}\varepsilon = \mu := I(\hat{\mu})$. Let $\hat{\varepsilon}'' \in \hat{h}^{\bullet}(\id \sqcup \rho)$ be any differential class such that $I(\hat{\varepsilon}'') = \varepsilon$. It follows that there exists a relative form $(\theta, \eta) \in \Omega^{\bullet-1}(\rho)$ such that $(\id, \id)^{*}\hat{\varepsilon}'' = \hat{\mu} + a(\theta, \eta)$. We set $\hat{\varepsilon}' := \hat{\varepsilon}'' - a(\theta, 0 \sqcup \eta)$, so that $(\id, \id)^{*}\hat{\varepsilon}' = (\hat{\mu} + a(\theta, \eta)) - a(\theta, \eta) = \hat{\mu}$. Now we have to reach a flat class with the same pull-back of $\hat{\varepsilon}'$. We have that $(\id, \id)^{*}R(\hat{\varepsilon}') = R(\hat{\mu}) = 0$, thus there exists a form $\chi \in \Omega^{\bullet-1}(X)$ such that $R(\hat{\varepsilon}') = (0, \chi \sqcup 0)$. Let us show that the de-Rham class $[\chi]$ belongs to the image of the Chern character. We consider the pull-back of $\hat{\varepsilon}'$ via the following composition:
	\[\resizebox{!}{6.5pt}{
	\xymatrix{
		\emptyset \ar[rr] \ar[d] & & X \sqcup A \ar[d]^{\id \sqcup \rho} \\
		\emptyset \ar[rr] \ar[d] & & X \ar[d]^{\id} \\
		A \ar[rr]^{\rho} \ar[d]_{\id} & & X \ar[d]^{\id} \\
		X \sqcup A \ar[rr] \ar[rr]^{\id \sqcup \rho} & & X.
}}\]

\vspace{3pt}
\noindent The pull-back of $\hat{\varepsilon}'$ from the last to the third line is $\hat{\mu}$ and the pull-back to the second line is $\pi^{*}(\hat{\mu})$, that vanishes by hypothesis. Thus, the overall pull-back is $0$. On the other side, the pull-back from the forth to the first line, by the axiom (A4), is $a(\cov(\hat{\varepsilon}')) = a(\chi) \sqcup 0$, thus $a(\chi) = 0$. This shows that $\chi \in \Omega^{\bullet-1}_{\ch}(X)$, thus there exists a class $\hat{\gamma} \in \hat{h}^{\bullet-1}(X)$ such that $R(\hat{\gamma}) = \chi$. Considering the sequence \eqref{LongExact2} associated to the last line of the previous diagram, we get the class $\Bock(\hat{\gamma} \sqcup 0) \in \hat{h}^{\bullet}(\id \sqcup \rho)$ and we set $\hat{\varepsilon} := \hat{\varepsilon}' + \Bock(\hat{\gamma} \sqcup 0)$. It follows that $R(\hat{\varepsilon}) = (0, \chi \sqcup 0) + (0, -R(\hat{\gamma} \sqcup 0)) = 0$. Moreover, because of the naturality of the Bockstein map, $(\id, \id)^{*}\Bock(\hat{\gamma} \sqcup 0) = \Bock(\id^{*}(\hat{\gamma} \sqcup 0)) = \Bock(0) = 0$, thus $(\id, \id)^{*}\hat{\varepsilon} = \hat{\mu} - 0 = \hat{\mu}$. \\

About \eqref{LongExact2}, we must prove the exactness from $\hat{h}^{\bullet-1}_{\fl}(X)$ to $h^{\bullet}(A)$. Actually, the exactness in $\hat{h}^{\bullet-1}_{\fl}(X)$ easily follows from the embedding $\hat{h}^{\bullet-1}_{\fl}(A) \hookrightarrow \hat{h}^{\bullet-1}(A)$ and the exactness of \eqref{LongExactFl}. Similarly, the exactness in $h^{\bullet}(A)$ easily follows from the surjectivity of $I \colon \hat{h}^{\bullet}(X) \to h^{\bullet}(X)$ and the exactness of the sequence associated to $h^{\bullet}$. Thus, there are three meaningful positions left.

\emph{Exactness in $\hat{h}^{n-1}(A)$.} The composition $\Bock \circ \rho^{*}$, starting from $\hat{h}^{\bullet}_{\fl}(X)$, coincides with the one of the flat sequence, hence it vanishes. Given $\hat{\alpha} \in \hat{h}^{n-1}(A)$, if $\Bock^{n-1}(\hat{\alpha}) = 0$ then $R(\hat{\alpha}) = 0$, because of formula \eqref{CurvBockstein}. Therefore the kernel of $\Bock^{n-1}$ is contained in the flat part $\hat{h}_{\fl}^{n-1}(A)$, hence the exactness follows from the one of \eqref{LongExactFl}.

\emph{Exactness in $\hat{h}^{n}(\rho)$.} We have already proven that $\pi^{*} \circ \Bock = 0$ in remark \ref{BockPi0Both}. Let us consider $\hat{\mu} \in \hat{h}^{n}(\rho)$ such that $\pi^{*}\hat{\mu} = 0$. It follows that $R(\hat{\mu}) = (0, \eta)$. Moreover, $0 = \rho^{*}\pi^{*}\hat{\mu} = a(\eta)$, thus $\eta$ represents a class belonging to the image the Chern character. It follows that there exists a class $\hat{\alpha} \in \hat{h}^{n-1}(A)$ such that $R(\hat{\alpha}) = -\eta$. Then, because of formula \eqref{CurvBockstein}, $\Bock(\hat{\alpha}) = \hat{\mu} + \hat{\mu}'$, with $\hat{\mu}' \in \hat{h}^{n}_{\fl}(\rho)$. Since $0 = \pi^{*}\Bock(\hat{\alpha}) = \pi^{*}\hat{\mu}'$, because of the exactness of \eqref{LongExactFl} there exists a class $\hat{\alpha}' \in \hat{h}^{n-1}_{\fl}(A)$ such that $\Bock(\hat{\alpha}') = \hat{\mu}'$. It follows that $\Bock(\hat{\alpha} - \hat{\alpha}') = \hat{\mu}$.

\emph{Exactness in $\hat{h}^{n}(X)$.} The pull-back to $A$ of a class in $\hat{h}^{n}(\rho)$ is topologically trivial because of the long exact sequence of $h^{\bullet}$ (or because of axiom (A4)). Let us fix a class $\hat{\nu} \in \hat{h}^{n}(X)$, such that $\rho^{*}I(\hat{\nu}) = 0$. It follows that $\rho^{*}\hat{\nu} = a(\eta)$ for a suitable $\eta \in \Omega^{\bullet-1}(A)$. We set $\omega := R(\hat{\nu})$. Then $\rho^{*}\omega = d\eta$, that is equivalent to $d(\omega, \eta) = 0$. Moreover:
\begin{align*}
	& \Bock \circ \rho^{*}(\hat{\nu}) \overset{\eqref{RhoBockstein}}= -a(\omega, 0) \\
	& \Bock \circ \rho^{*}(\hat{\nu}) = \Bock \circ a(\eta) \overset{\eqref{ABockstein}}= a(0, \eta).
\end{align*}
It follows that $a(\omega, \eta) = 0$, thus we can fix a class $\hat{\alpha}' \in \hat{h}^{n}(\rho)$ such that $R(\hat{\alpha}') = (\omega, \eta)$. It follows that $\pi^{*}(\hat{\alpha}') = \hat{\nu} + \hat{\nu}'$, with $\hat{\nu}'$ flat. Then:
\begin{align*}
	& \rho^{*}\pi^{*}(\hat{\alpha}') = a(\eta) + \rho^{*}\hat{\nu}' \\
	& \rho^{*}\pi^{*}(\hat{\alpha}') \overset{\textnormal{(A4)}}= a \circ \cov(\hat{\alpha}') = a(\eta).
\end{align*}
Thus $\rho^{*}\hat{\nu}' = 0$ and, by the exactness of the flat sequence, we can find $\hat{\alpha}''$ such that $\pi^{*}\hat{\alpha}'' = \hat{\nu}'$. Setting $\hat{\alpha} := \hat{\alpha}' - \hat{\alpha}''$, we get that $\pi^{*}\hat{\alpha} = \hat{\nu}$. \\

About \eqref{LongExact1}, we must prove the exactness from $\hat{h}^{\bullet-1}_{\fl}(A)$ to $h^{\bullet+1}(\rho)$. Actually, the exactness in $\hat{h}^{\bullet-1}_{\fl}(A)$ easily follows from the embedding $\hat{h}^{\bullet}_{\fl}(\rho) \hookrightarrow \hat{h}^{\bullet}_{\ppar}(\rho)$ and the exactness of \eqref{LongExactFl}. Similarly, the exactness in $h^{\bullet+1}(\rho)$ easily follows from the surjective map $I \colon \hat{h}^{\bullet}(A) \to h^{\bullet}(A)$ and the exactness of the sequence associated to $h^{\bullet}$. Thus, there are three meaningful positions left.

\emph{Exactness in $\hat{h}^{\bullet}_{\ppar}(\rho)$.} If a class belongs to the image of the Bockstein map, it follows from the exact sequence of the flat theory that its pull-back to $X$ vanishes. Vice-versa, let us consider $\hat{\mu} \in \hat{h}^{\bullet}(\rho)_{\ppar}$ such that $\pi^{*}\hat{\mu} = 0$. It follows that $R(\hat{\mu}) = (0, 0)$, thus the class is flat. By the sequence of the flat theory, we can find a pre-image via the Bockstein morphism.

\emph{Exactness in $\hat{h}^{n}(X)$.} The pull-back to $A$ of a class in $\hat{h}^{\bullet}_{\ppar}(\rho)$ vanishes because of the axiom (A4). Vice-versa, let us fix a class $\hat{\nu} \in \hat{h}^{\bullet}(X)$ such that $\rho^{*}\hat{\nu} = 0$. Since, in particular, $\rho^{*}I(\hat{\nu}) = 0$, by the exactness of \eqref{LongExact2} there exists $\hat{\alpha}' \in \hat{h}^{\bullet}(\rho)$ such that $\pi^{*}\hat{\alpha}' = \hat{\nu}$. We set $(\omega, \eta) := R(\hat{\alpha}')$, thus $\omega = R(\hat{\nu})$. We have that $0 = \rho^{*}\pi^{*}\hat{\alpha}' = a(\eta)$, thus there exists a class $\hat{\beta} \in h^{\bullet-1}(A)$ such that $R(\hat{\beta}) = \eta$. We set $\hat{\alpha} := \hat{\alpha}' + \Bock(\hat{\beta})$. Then, by formula \eqref{CurvBockstein}, $R(\hat{\alpha}) = (\omega, \eta) + (0, -\eta) = (\omega, 0)$, thus $\hat{\alpha}$ is parallel, and, by the exactness of \eqref{LongExact2}, $\pi^{*}\hat{\alpha} = \hat{\nu}$.

\emph{Exactness in $\hat{h}^{\bullet}(A)$.} If a class $\hat{\alpha} \in \hat{h}^{\bullet}(A)$ is the pull-back of a class in $X$, it follows from the exact sequence of $h^{\bullet}$ that it lies in the kernel of the (topological) Bockstein map. Vice-versa, if it belongs to the kernel, by the exact sequence of $h^{\bullet}$ we can find a class $\hat{\beta} \in \hat{h}^{n}(X)$ such that $\rho^{*}I(\hat{\beta}) = I(\hat{\alpha})$, hence $\rho^{*}\hat{\beta} = \hat{\alpha} + a(\theta)$ for a suitable form $\theta \in h^{\bullet-1}(A)$. If $\rho$ is a closed embedding, there exists a form $\xi \in h^{\bullet-1}(X)$ such that $\theta = \rho^{*}\xi$, thus $\rho^{*}(\hat{\beta} - a(\xi)) = \hat{\alpha}$.

\section{Existence and uniqueness}\label{ExUniqSec}

We are going to verify that any cohomology theory admits a relative differential extension, that is unique under the same hypotheses stated in \cite{BS} about the absolute case.

\SkipTocEntry \subsection{Existence}\label{ExistenceSec}

Given a cohomology theory $h^{\bullet}$, there exists a relative differential extension with $S^{1}$-integration, which is multiplicative if $h^{\bullet}$ is. This can be shown using the Hopkins-Singer model \cite{FR2}. We briefly recall the construction and verify that it satisfies the axioms.
\begin{Def}\label{DiffFunct} If $X$ is a smooth manifold, $Y$ a topological space, $V^{\bullet}$ a graded real vector space and $\kappa_{n} \in C^{n}(Y; V^{\bullet})$ a real singular cocycle, a \emph{differential function} from $X$ to $(Y, \kappa_{n})$ is a triple $(f, h, \omega)$ such that:
\begin{itemize}
	\item $f \colon X \to Y$ is a continuous function;
	\item $h \in C^{n-1}_{sm}(X; V^{\bullet})$ (`sm' means smooth);
	\item $\omega \in \Omega^{n}_{\cl}(X; V^{\bullet})$
\end{itemize}
satisfying the following condition, where $\chi \colon \Omega^{\bullet}(X; V^{\bullet}) \to C^{\bullet}(X; V^{\bullet})$ is the natural morphism:
\begin{equation}\label{DiffFunctEq}
	\delta^{n-1}h = \chi^{n}(\omega) - f^{*}\kappa_{n}.
\end{equation}
Moreover, a \emph{homotopy between two differential functions} $(f_{0}, h_{0}, \omega)$ and $(f_{1}, h_{1}, \omega)$ is a differential function $(F, H, \pi^{*}\omega) \colon X \times I \to (Y, \kappa_{n})$, such that $F$ is a homotopy between $f_{0}$ and $f_{1}$, $H\vert_{X \times \{i\}} = h_{i}$ for $i = 0,1$, and $\pi \colon X \times I \to X$ is the natural projection.
\end{Def}
We represent a fixed cohomology theory $h^{\bullet}$ via an $\Omega$-spectrum $(E_{n}, e_{n}, \varepsilon_{n})$, where $e_{n}$ is the marked point of $E_{n}$ and $\varepsilon_{n} \colon (\Sigma E_{n}, \Sigma e_{n}) \to (E_{n+1}, e_{n+1})$ is the structure map, whose adjoint $\tilde{\varepsilon}_{n} \colon E_{n} \to \Omega_{e_{n+1}} E_{n+1}$ is a homeomorphism. We also fix real singular cocycles $\iota_{n} \in C^{n}(E_{n}, e_{n}, \mathfrak{h}^{\bullet}_{\mathbb{R}})$ representing the Chern character of $h^{\bullet}$, such that $\iota_{n-1} = \int_{S^{1}} \varepsilon_{n}^{*}\iota_{n}$ \cite{Upmeier}.

\begin{Def}\label{StrongTTr} Given a differential function $(f, h, \omega) \colon X \to (E_{n}, \iota_{n})$, a \emph{strong topological trivialization} of $(f, h, \omega)$ is a homotopy $(F, H, \pi^{*}\omega) \colon X \times I \to (E_{n}, \iota_{n})$ between $(f, h, \omega)$ and a function of the form $(c_{e_{n}}, \chi(\eta), d\eta)$, where $c_{e_{n}}$ is the constant function with value $e_{n}$ and $\eta \in \Omega^{n-1}(X; \h^{\bullet}_{\R})$.
\end{Def}

Let us consider a smooth function between manifolds $\rho \colon A \to X$. We define the cylinder $\Cyl(\rho) := X \sqcup (A \times I) / \sim$, where $(a, 0) \sim \rho(a)$. We consider the following natural maps:
\begin{itemize}
 \item $\iota_{\Cyl(\rho)} \colon \Cyl(\rho) \to X \times I$, $x \mapsto (x, 0)$, $[(a, t)] \mapsto (\rho(a), t)$;
 \item $\iota_{\Cyl(A)} \colon \Cyl(A) \to \Cyl(\rho)$, $(a, t) \mapsto [(a, t)]$;
 \item $\iota_{A} \colon A \to \Cyl(\rho)$, $a \mapsto [(a, 0)]$;
 \item $\iota'_{A} \colon A \to \Cyl(\rho)$, $a \mapsto (a, 1)$;
 \item $\pi_{A} \colon I \times A \to A$, $(t, a) \mapsto a$.
\end{itemize}
In general $\Cyl(\rho)$ is not a manifold, nevertheless we will deal with differential functions $(f, h, \omega) \colon \Cyl(\rho) \to (E_{n}, \iota_{n})$, defined in the following way:
\begin{itemize}
  \item $f \colon \Cyl(\rho) \to E_{n}$ is a continuous function.
  \item $\omega \in \Omega^{n}_{cl}(X; \mathfrak{h}^{\bullet}_{\mathbb{R}})$, and it defines a smooth cocycle $\chi^{n}(\omega)$ on $\Cyl(\rho)$ as follows. Let us consider the pull-back $\pi_{X}^{*}\omega$ on $X \times I$. A simplex $\sigma \colon \Delta^{n} \to \Cyl(\rho)$ is defined to be smooth if and only if the composition $\iota_{\Cyl(\rho)} \circ \sigma \colon \Delta^{n} \rightarrow X \times I$ is. The smooth cochain $\chi^{n}(\omega)$ on $\Cyl(\rho)$ is defined by $\chi^{n}(\omega)(\sigma) := \chi^{n}(\pi_{X}^{*}\omega)(\iota_{\Cyl(\rho)} \circ \sigma)$.
  \item $h \in C^{n-1}_{sm}(\Cyl(\rho); \mathfrak{h}^{\bullet}_{\mathbb{R}})$ and it satisfies $\delta^{n-1}h = \chi^{n}(\omega) - f^{*}\iota_{n}$.
\end{itemize}
\begin{Def}\label{HHatHS} The group $\hat{h}^{n}(\rho)$ contains the equivalence classes $[(f, h, \omega, \eta)]$, where:
\begin{itemize}
  \item $(f, h, \omega) \colon \Cyl(\rho) \to (E_{n}, \iota_{n})$ is a differential function such that $\iota_{\Cyl(A)}^{*}(f, h, \omega)$ is a strong topological trivialization of $\iota_{A}^{*}(f, h, \omega)$ verifying the relation $(\iota'_{A})^{*}(f, h, \omega) = (c_{e_{n}}, \chi(\eta), d\eta)$;
  \item $(f, h, \omega, \eta)$ is equivalent to $(g, k, \omega, \eta)$ if the differential functions $(f, h, \omega)$ and $(g, k, \omega)$ are homotopic relatively to the upper base of the cylinder. This means that a homotopy $(F, H, \pi^{*}\omega) \colon \Cyl(\rho_{I}) \to (E_{n}, \iota_{n})$ between the two functions is required to satisfy $(\iota'_{A})_{I}^{*}(F, H, \pi^{*}\omega) = \pi_{A}^{*}(\iota'_{A})^{*}(f, h, \omega)$, where $(\iota'_{A})_{I}$ is defined as in formula \eqref{RhoI}.\footnote{Since $I$ is (locally) compact, $I \times \Cyl(\rho)$ is canonically homeomorphic to $\Cyl(\rho_{I})$. We will apply this homeomorphism when necessary, without stating it explicitly.}
\end{itemize}
\end{Def}
We set:
\begin{equation}\label{IRDef}
 I[(f, h, \omega)] := [f] \qquad\qquad R[(f, h, \omega, \eta)] := (\omega, \eta),
\end{equation}
where $[f] \in [(\Cyl(\rho), A \times \{1\}), (E_{n}, e_{n})] \simeq h^{n}(\rho)$. Moreover, we define the map $a$ in the following way. Given a form $\omega \in \Omega^{n}(X; \h^{\bullet}_{\R})$, we set $\tilde{\omega} := \pi_{A}^{*}\rho^{*}\omega \in \Omega^{n}(\Cyl(A); \h^{\bullet}_{\R})$ and, given a form $\eta \in \Omega^{n-1}(A; \h^{\bullet}_{\R})$, we set $\tilde{\eta} := \pi_{A}^{*}\eta \in \Omega^{n-1}(\Cyl(A); \h^{\bullet}_{\R})$. We define the smooth singular cochain $\chi^{n}(\omega, \eta) \in C^{n}_{sm}(\rho; \h^{\bullet}_{\R})$ as follows. We fix a real number $\varepsilon \in (0, 1)$ and we take a smooth non-decreasing function $\theta \colon I \to I$ such that $\theta(t) = 0$ for $t \leq \varepsilon$ and $\theta(1) = 1$. We fix the open cover $\{U,W\}$ of $\Cyl(\rho)$ defined by $U := A \times (\frac{\varepsilon}{3},1]$ and $W := A\times [0,\frac{\varepsilon}{2}) \sqcup_{\rho} X$. For each smooth chain $\sigma \colon \Delta^{n} \to \Cyl(\rho)$, we take the iterated barycentric subdivision, so that the image of each sub-chain is contained in $U$ or in $W$; then, for each small chain $\sigma'$, we set
  \[\chi^{n}(\omega, \eta)(\sigma') = \left\{ \begin{array}{ll} \chi^{n} \bigl(\tilde{\omega}-d(\theta\tilde{\eta})\bigr)(\sigma') & \mbox{if } \sigma' \subset U \\
  \chi^{n}(\omega)(\pi_{X}\circ\sigma') & \mbox{if } \sigma' \subset W, \end{array} \right.
\]
where $\pi_{X} \colon W \to X$ is the natural projection defined by $[a,t] \mapsto \rho(a)$ and $[x]\mapsto x$. Note that the morphism is well defined for $\sigma' \subset U \cap W$, since $\theta(t) = 0$ for $t \leq \varepsilon$. Finally, we define
\begin{equation}\label{ADef}
\begin{split}
	a \colon & \Omega^{n-1}(\rho; \h^{\bullet}_{\R})/\IIm(d)\to \hat{h}^{n}(\rho) \\
	& [(\omega,\eta)] \mapsto [(c_{e_n}, \chi^{n-1}(\omega, \eta), d\omega, \rho^{*}\omega - d\eta)].
\end{split}
\end{equation}
The cochain $\chi^{n-1}(\omega, \eta)$ depends on the choice of the function $\theta$, but the equivalence class $[(c_{e_n}, \chi^{n-1}(\omega, \eta), d\omega, \rho^{*}\omega - d\eta)]$ does not, since two different functions $\theta$ lead to homotopic representatives.

Given two maps $\rho \colon A \to X$ and $\eta \colon B \to Y$ and a morphism $(\varphi, \psi) \colon \rho \to \eta$, there is a natural induced map $(\varphi, \psi) \colon \Cyl(\rho) \to \Cyl(\eta)$, $x \mapsto \varphi(x)$, $[(a, t)] \mapsto [(\psi(a), t)]$. We define the pull-back $(\varphi, \psi)^{*}[(f, h, \omega, \eta)] := [(f \circ \varphi, \varphi^{*}h, \varphi^{*}\omega, \psi^{*}\eta)]$.

Let us verify that axioms (A1)--(A4) of definition \ref{RelDiffExt} hold. The first one is a direct consequence of the definitions \eqref{ADef} and \eqref{IRDef}. About axiom (A4), we have that
\[\rho^{*}\pi^{*}[(f, h, \omega, \eta)] = \rho^{*}[(f\vert_{X}, h\vert_{X}, \omega)] = \iota_{A}^{*}[(f, h, \omega)] \overset{(*)}= a(\eta) = a \circ \cov([(f, h, \omega, \eta)]),
\]
the equality $(*)$ being due to the fact that, by definition, $\iota_{\Cyl(A)}^{*}(f, h, \omega)$ is a homotopy between $\iota_{A}^{*}(f, h, \omega)$ and $a(\eta) = (c_{e_{n}}, \chi(\eta), d\eta)$. In order to verify (A2), we observe that a representative $(f, h, \omega, \eta)$ of an element of $\hat{h}^{n}(\rho)$, as defined in \ref{HHatHS}, can be described in the following equivalent way:
\begin{itemize}
  \item $f \colon (\Cyl(\rho), A \times \{1\}) \to (E_{n},e_{n})$ is a map of pairs;
  \item $h \in C^{n-1}(\Cyl(\rho); \h^{\bullet}_{\R})$;
  \item $(\omega,\eta) \in \Omega^{n}_{\cl}(\rho; \h^{\bullet}_{\R})$;
  \item $\delta^{n-1}(h,0) = (\chi^{n}(\omega),\chi^{n-1}(\eta)) - (f^{*}\iota_{n},0)$, the boundary $\delta$ being the one of the mapping cone complex $C^{\bullet}(\Cyl(\rho)) \oplus C^{\bullet-1}(A)$ associated to the embedding of the upper base $\iota'_{A} \colon A \to \Cyl(\rho)$.
\end{itemize} 
The condition $\delta^{n-1}(h,0) = (\chi^{n}(\omega), \chi^{n-1}(\eta)) - (f^{*}\iota_{n}, 0)$ implies that $[(\omega,\eta)]_{\dR} = [f^{*}\iota_{n}] = \ch[f]$, where $f$ is a map of pairs. Finally, in order to show that axiom (A3) holds, the proof is similar to \cite[Theorems 2.4 and 2.5]{Upmeier}, applied to the pair $(\Cyl(\rho), A \times \{1\})$.

This differential extension has a natural $S^{1}$-integration, defined integrating each component of the differential function $(f, h, \omega, \eta)$ as shown in \cite[Chapter 3]{Upmeier}. Moreover, if $h^{\bullet}$ is multiplicative, then $\hat{h}^{\bullet}$ is multiplicative too, the product being defined as in \cite[Chapter 4]{Upmeier}.

\SkipTocEntry \subsection{Uniqueness for smooth cofibrations}\label{UniquenessSec} We are going to show that the uniqueness result of \cite{BS} holds even in the relative case, under the same hypotheses. Our aim consists in extending to the relative case the construction of the morphism between any two differential extensions of $h^{\bullet}$, described in \cite[Section 3]{BS}. It will easily follow that it induces a morphism between the corresponding long exact sequences of the form \eqref{LongExact2}, therefore it is an isomorphism because of the five lemma.

We use the results of \cite[Section 2]{BS} about approximation of spectra through manifolds, assuming the same hypotheses therein. Let  $h^{\bullet}$ be a cohomology theory and consider a spectrum $\{E_{n},e_{n}, \varepsilon_{n}\}_{n\in \mathbb{Z}}$ representing it. For a fixed $n \in \Z$, we take a sequence of pointed manifolds $\{\mathcal{E}_{i}, a_{i}\}_{i \in \Z}$ and two sequences of maps 
\begin{align*}
  x_{i} \colon (\mathcal{E}_{i}, a_{i}) \to (E_{n}, e_{n}) \qquad\qquad k_{i} \colon (\mathcal{E}_{i}, a_{i}) \to (\mathcal{E}_{i+1}, a_{i+1})
\end{align*}
such that $x_{i} = x_{i+1}\circ k_{i}$. Moreover, for a given class $u \in h^{\bullet}(E_{n}, e_{n})$, we fix a family of closed forms $\omega_{i} \in \Omega^{\bullet}(\mathcal{E}_{i}, a_{i}; \h^{\bullet}_{\R})$ such that $\omega_{i} = k_{i}^{*}\omega_{i+1}$ and, in the reduced cohomology with marked point $a_{i}$, we have $[\omega_{i}]_{\dR} = \ch(x_{i}^{*}u)$. Finally, we fix a family of differential classes $\hat{u}_{i} \in \hat{h}^{\bullet}_{\ppar}(\mathcal{E}_{i}, a_{i})$ such that $I'(\hat{u}_{i}) = x^{*}_{i}(u)$ and $R'(\hat{u}_{i}) = \omega_{i}$.\footnote{The notation $\hat{h}^{\bullet}_{\ppar}(\mathcal{E}_{i}, a_{i})$ refers to the cohomology of the embedding $\{a_{i}\} \hookrightarrow E_{i}$.}
 
Let $(\hat{h}^{\bullet}, I, R, a)$ and $(\hat{h}'^{\bullet}, I', R', a')$ be two differential extensions of $h^{\bullet}$. The morphism from $\hat{h}^{\bullet}$ to $\hat{h}'^{\bullet}$ we are looking for is much easier to construct in the case of a closed embedding $\rho \colon A \hookrightarrow X$, since $\rho$ is a smooth cofibration. In this case, for any $\hat{v} \in \hat{h}^{n}(\rho)$, there exists a morphism in the category $\CC_{2}$
\begin{equation}\label{RelativeSpectrum}
	\xymatrix{ A \ar@{^(->}[r]^{\rho} \ar[d] & X \ar[d]^{f} \\ e_{n} \ar@{^(->}[r]^{i_{e_{n}}} & E_{n}
}
\end{equation}
such that $I(\hat{v})$ is represented by the homotopy class $[f] \in [\rho, i_{e_{n}}]$. It follows that $I(\hat{v}) = f^{*}(u)$, where $u \in h^{n}(i_{e_{n}})$ is the tautological class represented by the identity map of $(E_{n}, e_{n})$. By the approximation lemmas \cite[Prop.\ 2.1 and 2.3]{BS}, there exist a based manifold $(\mathcal{E}_{i}, a_{i})$ and a map
\[\xymatrix{ A \ar[r]^{\rho} \ar[d] & X \ar[d]^{f_{i}} \\ a_{i} \ar@{^(->}[r]^{i_{a_{i}}} & \mathcal{E}_{i}
}\]
such that $f = x_{i} \circ f_{i}$. Note that $I(\hat{v}) = f^{*}(u) = f_{i}^{*}x_{i}^{*}(u) = f_{i}^{*}(I(\hat{u}_{i})) = I(f_{i}^{*}\hat{u}_{i})$, hence we have 
  \[\hat{v} = f_{i}^{*}\hat{u} + a(\zeta,\nu)
\]
for a unique $(\zeta,\nu) \in \Omega^{n-1}(\rho; \h^{\bullet}_{\R})/\IIm(\ch)$. Repeating the same construction for $\hat{h}'^{\bullet}$, we define:
  \[\begin{split}
  \Phi \colon & \hat{h}^{\bullet}(\rho) \to \hat{h}'^{\bullet}(\rho) \\
  & f_{i}^{*}\hat{u} + a(\zeta,\nu) \mapsto f^{*}_{i}\hat{u}_{i}' + a'(\zeta,\nu).
\end{split}\]
In order to show that $\Phi$ is well defined, we must verify that it is independent of the choice of the functions $f_{i}$. Let us fix $\hat{v}\in \hat{h}^{n}(\rho)$. As in the absolute case, we may reduce the problem to the case of two homotopic functions $f_{i}, \tilde{f}_{i} \colon \rho \to (\mathcal{E}_{i},a_{i})$ such that $I(\hat{v}) = f_{i}^{*}x_{i}^{*}(u) = \tilde{f}_{i}^{*}x_{i}^{*}(u)$; we consider a homotopy $F \colon \rho_{I} \to (\mathcal{E}_{i}, a_{i})$. To each $f_{j}$ and $\tilde{f}_{j}$ we associate as above the forms $(\zeta,\nu)$ and $(\tilde{\zeta},\tilde{\nu})$ and the morphisms $\Phi$ and $\tilde{\Phi}$ respectively. We define $(\alpha,\beta) = \int_{\rho_{I}/\rho} F^{*}(\omega_{i},0)$. We note that $F^{*}(\omega_{i},0) = R(F^{*}\hat{u}_{i})$, so by the homotopy formula we obtain
  \[f_{i}^{*} \hat{u}_{i} - \tilde{f}_{i}^{*} \hat{u}_{i} = a(\alpha,\beta) \qquad\qquad f_{i}^{*} \hat{u'}_{i} - \tilde{f}_{i}^{*} \hat{u'}_{i} = a'(\alpha,\beta).
\]
Since $\hat{v} = f_{i}^{*}\hat{u}' + a'(\zeta,\nu) = \tilde{f}_{i}^{*}\hat{u}' + a'(\tilde{\zeta}, \tilde{\nu})$, the homotopy formula also implies that $a'(\tilde{\zeta},\tilde{\nu}) = a'(\zeta,\nu)+a'(\alpha,\beta)$, thus:
  \[\Phi(\hat{v}) = f^{*}_{i}\hat{u}'_{i} + a'(\zeta, \nu) = f^{*}_{i}\hat{u}'_{i} + a'(\tilde{\zeta},\tilde{\nu}) - a'(\alpha,\beta) = a'(\tilde{\zeta},\tilde{\nu}) + \tilde{f}^{*}_{i}\hat{u}'_{i} = \tilde{\Phi}(\hat{v}).
\] 
In order to show that $\Phi$ induces a morphism of long exact sequences it is enough to show that $\Phi$ commutes with the Bockstein map. This easily follows from the naturality of $\Phi$, since it commutes with each of the steps (S1)--(S6). From the five lemma, it follows that it is an isomorphism.

This implies in particular that, at least on closed embeddings, any differential extension of $h^{\bullet}$ is naturally isomorphic to the Hopkins-Singer model, summarized in section \ref{ExistenceSec}. Since in such a model the flat theory is isomorphic to $h^{\bullet}(\rho; \R/\Z)$, the latter being defined via the Moore spectrum (see \cite[Chapter 5]{BS}), it follows that the same isomorphism holds about $h^{\bullet}_{\fl}$.\footnote{In \cite{BS} the authors gave an independent proof of this result, since the details of the integration map in the Hopkins-Singer model had not been worked out yet. Now we do not have this problem any more (see \cite{Upmeier}).}

\SkipTocEntry \subsection{Uniqueness for generic maps (sketch)}

For a generic smooth map $\rho \colon A \to X$, the construction of $\Phi$ is less trivial, since we cannot represent any element of $h^{\bullet}(\rho)$ via a morphism of the form \eqref{RelativeSpectrum}. We just sketch such a construction, deferring the details to a future paper. In order to use the language of spectra in this general context, we must use cones or cylinders. In particular, we can replace \eqref{RelativeSpectrum} by a map of the following form:
\begin{equation}\label{RelativeSpectrum2}
	\xymatrix{ A \ar@{^(->}[r]^(.4){i_{1}} \ar[d] & \Cyl(\rho) \ar[d]^{f} \\ e_{n} \ar@{^(->}[r]^{i_{e_{n}}} & E_{n},
}
\end{equation}
where $i_{1}$ is the embedding of $A$ in the upper base of the cylinder. Now we should apply the same idea used above, but of course we must deal with the fact that $\Cyl(\rho)$ is not a manifold in general. We can extend the groups $\hat{h}^{\bullet}$ to cylinders following \cite{AH2}. We start defining smoothness.
\begin{Def} Given a smooth map $\rho \colon A \to X$, we consider the map $\iota_{\Cyl(\rho)} \colon \Cyl(\rho) \to X \times I$, $x \mapsto (x, 0)$ and $[(a, t)] \mapsto (\rho(a), t)$.
\begin{itemize}
	\item Given a manifold $Y$ and a function $f \colon Y \to \Cyl(\rho)$, the function $f$ is \emph{smooth} if and only if $\iota_{\Cyl(\rho)} \circ f \colon Y \to X \times I$ is smooth.
	\item Given a cylinder $\Cyl(\eta)$ (in particular, it can be a manifold) and a function $g \colon \Cyl(\rho) \to \Cyl(\eta)$, the function $g$ is \emph{smooth} if and only if, for any manifold $Y$ and any smooth map $f \colon Y \to \Cyl(\rho)$, the composition $g \circ f \colon Y \to \Cyl(\eta)$ is smooth.
	\item A singular chain $\sigma \colon \Delta^{\bullet} \to \Cyl(\rho)$ is \emph{smooth} if and only if it is smooth as a function from the manifold $\Delta^{\bullet}$ to $\Cyl(\rho)$. We call $C_{\bullet}^{sm}(\Cyl(\rho))$ the corresponding group. A \emph{smooth real singular cochain} is an element of $C^{\bullet}_{sm}(\Cyl(\rho)) := \Hom(C_{\bullet}^{sm}(\Cyl(\rho)), \R)$.
	\item A \emph{differential form} on $\Cyl(\rho)$ is a smooth singular cochain $\varphi \in C^{\bullet}_{sm}(\Cyl(\rho))$ such that, for any manifold $Y$ and any smooth map $f \colon Y \to \Cyl(\rho)$, the pull-back $f^{*}\varphi$ is (the image of) a differential form on $Y$. We call $\Omega^{\bullet}(\Cyl(\rho))$ the corresponding group.
\end{itemize}
\end{Def}
Now we can show how to extend to cylinders the group $\hat{h}^{\bullet}$. Let us consider a smooth map between cylinders $\varphi \colon \Cyl(\nu) \to \Cyl(\xi)$. We define the category $\CC_{\varphi}$ in the following way. An object is a diagram of the form
\begin{equation}\label{MorphismDiagramCyl}
	\xymatrix{
	\Cyl(\nu) \ar[r]^{\varphi} & \Cyl(\xi) \\
	A \ar[r]^{\rho} \ar[u]^(.45){v} & X, \ar[u]_(.45){u}
}\end{equation}
where $\rho$ is a smooth map between manifolds and $u$ and $v$ are smooth too. We denote such an object by $(u, v) \colon \rho \to \varphi$. A morphism from the object $(u', v') \colon \eta \to \varphi$ to the object $(u, v) \colon \rho \to \varphi$ is a diagram of the form \eqref{MorphismDiagram}, such that $u \circ f = u'$ and $v \circ g = v'$.
\begin{Def} An element of the group $\hat{h}^{\bullet}(\varphi)$, with curvature $(\omega, \eta) \in \Omega^{\bullet}(\varphi)$, is a functor $\hat{\alpha} \colon \CC_{\varphi} \to \A_{\Z}$ with the following properties. We set $(u, v)^{!}\hat{\alpha} := \hat{\alpha}((u, v) \colon \rho \to \varphi)$.
\begin{itemize}
	\item Given $(u, v) \colon \rho \to \varphi$, we have that $(u, v)^{!}\hat{\alpha} \in \hat{h}^{\bullet}(\rho)$.
	\item Given a morphism $(f, g) \colon \eta \to \rho$, we have that $(f, g)^{*}(u, v)^{!}\hat{\alpha} = (u \circ f, v \circ g)^{!}\hat{\alpha}$.
	\item Given $(u, v), (u', v') \colon \rho \to \Cyl(\varphi)$ and a smooth homotopy $(U, V) \colon \rho_{I} \to \varphi$ between $(u, v)$ and $(u', v')$, we have that $(u, v)^{!}\hat{\alpha} - (u', v')^{!}\hat{\alpha} = a \bigl( \int_{I}(U, V)^{*}(\omega, \eta) \bigr)$.
\end{itemize}
We set $R(\hat{\alpha}) := (\omega, \eta)$. Moreover, given $(\omega', \eta') \in \Omega^{\bullet-1}(\varphi)$, we define $a(\omega', \eta')$ as the class such that $(u, v)^{!}(a(\omega', \eta')) = a((u, v)^{*}(\omega', \eta'))$.
\end{Def}
Given a morphism from $\psi \colon \Cyl(\nu') \to \Cyl(\xi')$ to $\varphi \colon \Cyl(\nu) \to \Cyl(\xi)$, i.e.\ a pair of smooth functions $h \colon \Cyl(\xi') \to \Cyl(\xi)$ and $k \colon \Cyl(\nu') \to \Cyl(\nu)$ such that $\varphi \circ k = h \circ \psi$, we define the pull-back $(h, k)^{*} \colon \hat{h}^{\bullet}(\varphi) \to \hat{h}^{\bullet}(\psi)$ by $(h, k)^{*}(f, g)^{!}(\hat{\alpha}) := (f \circ h, g \circ k)^{!}(\hat{\alpha})$. In this way we get the functor $\hat{h}^{\bullet}$ from the category of smooth maps between cylinders to $\A_{\Z}$, that can be easily endowed with the corresponding $S^{1}$-integration.

Given a smooth map $\rho \colon A \to X$ between manifolds, we call $\hat{h}^{\bullet}_{0}(i_{1})$ the subgroup of $\hat{h}^{\bullet}(i_{1})$ formed by the classes whose curvature $(\omega, \eta) \in \Omega^{\bullet}(i_{1})$ verifies $\iota_{\Cyl(A)}^{*}\omega = \pi_{A}^{*}d\eta$. In particular, we call $\hat{h}^{\bullet}_{\ppar, 0}(i_{1})$ the subgroup of $\hat{h}^{\bullet}_{0}(i_{1})$ formed by the classes whose curvature $(\omega, 0) \in \Omega^{\bullet}(i_{1})$ verifies $\iota_{\Cyl(A)}^{*}\omega = 0$. We have the natural isomorphisms:
\begin{equation}\label{EmdHatCyl}
	\hat{h}^{\bullet}(\rho) \overset{\!\simeq}\longrightarrow \hat{h}^{\bullet}_{0}(i_{1}) \qquad\qquad \hat{h}^{\bullet}_{\ppar}(\rho) \overset{\!\simeq}\longrightarrow \hat{h}_{\ppar, 0}^{\bullet}(i_{1}),
\end{equation}
induced by pull-back from the following diagram:
\begin{equation}\label{MapRhoCyl}
	\xymatrix{ A \ar@{^(->}[r]^(.4){i_{1}} \ar@{=}[d] & \Cyl(\rho) \ar[d]^{\pi} \\ A \ar[r]^{\rho} & X.
}
\end{equation}
In order to prove that they are isomorphisms, we construct the long exact sequence \eqref{LongExact1} (truncated at $\hat{h}^{\bullet}(A)$ in general), replacing $\hat{h}^{\bullet}_{\ppar}(\rho)$ by $\hat{h}_{\ppar, 0}^{\bullet}(i_{1})$, and we apply the five lemma. The morphism $\hat{h}_{\ppar, 0}^{\bullet}(i_{1}) \to \hat{h}^{\bullet}(X)$ is the pull-back via the inclusion of $X$ in $\Cyl(\rho)$, while the Bockstein map is defined by $(u, v)^{!}(\Bock(\alpha)) := \Bock(v^{*}(\hat{\alpha}))$. The sequence is exact because the kernel of $\hat{h}_{\ppar, 0}^{\bullet}(i_{1}) \to \hat{h}^{\bullet}(X)$ is made by flat classes vanishing on $X$, hence, by the suspension isomorphism in the flat theory, the Bockstein map is an isomorphism from $\hat{h}^{\bullet}_{\fl}(A)$ to such a kernel. Now the five lemma implies that the second map of \eqref{EmdHatCyl} is an isomorphism. Now we consider the sequence \eqref{SeqExCurtaCurvPar} and the corresponding sequence
	\[\xymatrix{
	0 \ar[r] & \hat{h}_{\ppar, 0}^{\bullet}(i_{1}) \ar[r] & \hat{h}^{\bullet}_{0}(i_{1}) \ar[r]^(.45){R'} & \Omega^{\bullet}_{\ch, 0}(\rho) \ar[r] & 0.
}\]
Applying the five lemma again we get the first isomorphism of \eqref{EmdHatCyl}.

Now we construct the morphism $\Phi$ as above. For any $\hat{v} \in \hat{h}^{n}(\rho)$, we embed it in $\hat{h}^{n}(i_{1})$ via \eqref{EmdHatCyl}, hence there exists a morphism of the form \eqref{RelativeSpectrum2}, such that $I(\hat{v})$ is represented by the homotopy class $[f] \in [i_{1}, i_{e_{n}}]$. It follows that $I(\hat{v}) = f^{*}(u)$, where $u \in h^{n}(i_{e_{n}})$ is the tautological class represented by the identity map of $(E_{n}, e_{n})$. By the approximation lemmas, there exist a based manifold $(\mathcal{E}_{i}, a_{i})$ and a map
\[\xymatrix{ A \ar@{^(->}[r]^(.4){i_{1}} \ar[d] & \Cyl(\rho) \ar[d]^{f_{i}} \\ a_{i} \ar@{^(->}[r]^{i_{a_{i}}} & \mathcal{E}_{i}
}\]
such that $f = x_{i} \circ f_{i}$. Note that $I(\hat{v}) = f^{*}(u) = f_{i}^{*}x_{i}^{*}(u) = f_{i}^{*}(I(\hat{u}_{i})) = I(f_{i}^{*}\hat{u}_{i})$, hence we have 
  \[\hat{v} = f_{i}^{*}\hat{u} + a(\zeta,\nu)
\]
for a unique $(\zeta,\nu) \in \Omega^{n-1}(i_{1}; \h^{\bullet}_{\R})/\IIm(\ch)$. Repeating the same construction for $\hat{h}'^{\bullet}$, we define:
  \[\begin{split}
  \Phi \colon & \hat{h}^{\bullet}(\rho) \to \hat{h}'^{\bullet}(\rho) \\
  & f_{i}^{*}\hat{u} + a(\zeta,\nu) \mapsto f^{*}_{i}\hat{u}_{i}' + a'(\zeta,\nu).
\end{split}\]
Applying the five lemma to the exact sequence \eqref{LongExact2} we get the uniqueness of the relative extension. In particular, the isomorphism with the Hopkins-Singer model provides the isomorphism $h^{\bullet}_{\fl}(\rho) \simeq h^{\bullet}(\rho; \R/\Z)$.

\section{Complements on differential extensions}\label{ComplementsSec}

We show some other interesting versions of the notion of differential refinement.

\SkipTocEntry \subsection{Differential cohomology with compact support}\label{DiffCohCpt}

We are going to define the compactly-supported version of differential cohomology. This has been done about ordinary differential cohomology in \cite{BBSS}, using the language of Cheeger-Simons characters. Here we generalize the construction to any cohomology theory, within the axiomatic setting. Given a smooth manifold $X$, we denote by $\K_{X}$ the directed set formed by the compact subsets of $X$, the partial ordering being given by set inclusion. We think of $\K_{X}$ as a category, whose objects are the compact subsets of $X$ and such that the set of morphisms from $K$ to $H$ contains one element if $K \subset H$ and is empty otherwise. There is a natural functor $\Cc_{X} \colon \K_{X} \to \M_{2}^{\op}$, assigning to an object $K$ the open embedding $i_{K} \colon X \setminus K \hookrightarrow X$ and to a morphism $K \subset H$ the natural morphism $i_{KH} \colon i_{H} \to i_{K}$ defined by the following diagram:
\begin{equation}\label{DiagCptX}
\xymatrix{
  X \setminus H \ar@{^(->}[d] \ar@{^(->}[r]^(.6){i_{H}} & X \ar@{=}[d] \\
  X \setminus K \ar@{^(->}[r]^(.6){i_{K}} & X.
}
\end{equation}
Given a cohomology theory $h^{\bullet}$ and a differential extension $\hat{h}^{\bullet} \colon \M_{2}^{\op} \to \A_{\Z}$, the corresponding compactly-supported differential extension $\hat{h}^{\bullet}_{\cpt}(X)$ is the colimit of the composition functor $\hat{h}^{\bullet}_{\ppar} \circ \Cc_{X} \colon \K_{X} \to \A_{\Z}$:
\begin{equation}
  \hat{h}^{\bullet}_{\cpt}(X) := \colim \bigl( \, \hat{h}_{\ppar}^{\bullet} \circ \Cc_{X} \colon \K_{X} \to \A_{\Z} \, \bigr).
\end{equation}
Since $\hat{h}_{\ppar}^{\bullet}$ and $\Cc_{X}$ are both contravariant, the composition is covariant. Concretely, an element $\hat{\alpha}_{\cpt} \in \hat{h}^{\bullet}_{\cpt}(X)$ is an equivalence class $\hat{\alpha}_{\cpt} = [\hat{\alpha}]$, represented by a parallel class $\hat{\alpha} \in \hat{h}_{\ppar}^{\bullet}(X, X \setminus K)$, $K$ being a compact subset of $X$. The colimit is taken over the groups $\hat{h}_{\ppar}^{\bullet}(X, X \setminus K)$, where, if $K \subset H$, the corresponding morphism in the direct system is the pull-back $i_{KH}^{*} \colon \hat{h}_{\ppar}^{\bullet}(X, X \setminus K) \to \hat{h}_{\ppar}^{\bullet}(X, X \setminus H)$. If $X$ itself is compact, we get a canonical isomorphism $\hat{h}^{\bullet}(X) \simeq \hat{h}^{\bullet}_{\cpt}(X)$, identifying $\hat{\alpha} \in \hat{h}^{\bullet}(X) = \hat{h}^{\bullet}_{\ppar}(X, X \setminus X)$ with the class represented by $\hat{\alpha}$ itself for $K = X$.

We have defined the group associated to a manifold $X$. We can extend this definition to the category $\M'$ whose objects are smooth manifolds (the same of the category $\M$) and whose morphisms are \emph{open embeddings}. In fact, let us fix an open embedding $\iota \colon Y \hookrightarrow X$. For any compact subset $K \subset Y$, from the embedding of pairs $\iota_{K} \colon (Y, Y \setminus K) \hookrightarrow (X, X \setminus \iota(K))$, we get the induced morphism $\iota_{K}^{*} \colon \hat{h}_{\ppar}^{\bullet}(X, X \setminus \iota(K)) \to \hat{h}_{\ppar}^{\bullet}(Y, Y \setminus K)$. By the excision property of parallel classes \ref{ParExcision}, it follows that $\iota_{K}^{*}$ is an isomorphism. If $K \subset H$, the following diagram commutes:
\[\xymatrix{
  \hat{h}_{\ppar}^{\bullet}(Y, Y \setminus K) \ar[rr]^{(\iota_{K}^{*})^{-1}} \ar[d]^{i_{KH}^{*}} & & \hat{h}_{\ppar}^{\bullet}(X, X \setminus \iota(K)) \ar[d]^{i_{KH}^{*}} \\
  \hat{h}_{\ppar}^{\bullet}(Y, Y \setminus H) \ar[rr]^{(\iota_{H}^{*})^{-1}} & & \hat{h}_{\ppar}^{\bullet}(X, X \setminus \iota(H))
}\]
therefore we get an induced morphism between the colimits, i.e., $\iota_{*} \colon \hat{h}^{\bullet}_{\cpt}(Y) \to \hat{h}^{\bullet}_{\cpt}(X)$.

We can define as above the following functors:
\begin{align*}
  & h^{\bullet}_{\cpt}(X) := \colim \bigl( \, h^{\bullet} \circ \Cc_{X} \colon \K_{X} \to \A_{\Z} \, \bigr) & & \Omega^{\bullet}_{\cpt}(X) := \colim \bigl( \, \Omega^{\bullet} \circ \Cc_{X} \colon \K_{X} \to \A_{\Z} \, \bigr) \\
  & \Omega^{\bullet}_{\cl, \cpt}(X) := \colim \bigl( \, \Omega^{\bullet}_{\cl} \circ \Cc_{X} \colon \K_{X} \to \A_{\Z} \, \bigr) & & \Omega^{\bullet}_{\ex, \cpt}(X) := \colim \bigl( \, \Omega^{\bullet}_{\ex} \circ \Cc_{X} \colon \K_{X} \to \A_{\Z} \, \bigr) \\
  & \Omega^{\bullet}_{\ch, \cpt}(X) := \colim \bigl( \, \Omega^{\bullet}_{\ch} \circ \Cc_{X} \colon \K_{X} \to \A_{\Z} \, \bigr).
\end{align*}
A relative form $\omega \in \Omega^{\bullet}(X, X \setminus K)$ is defined as a form $\omega \in \Omega^{\bullet}(X)$ whose restriction to $X \setminus K$ vanishes. Since the pull-back $i_{KH}^{*} \colon \Omega^{\bullet}(X, X \setminus K) \to \Omega^{\bullet}(X, X \setminus H)$ is injective, an element of $\Omega^{\bullet}_{\cpt}(X)$ is a form $\omega$ on $X$ whose support is compact, i.e., such that $\omega$ vanishes on the complement of a compact subset of $X$. The same holds for closed forms. In the case of exact forms, an element of $\Omega^{\bullet}_{\ex, \cpt}(X)$ is a form $\omega$ on $X$ such that there exists a form \emph{with compact support} $\eta \in \Omega^{\bullet-1}_{\cpt}(X)$ satisfying the identity $\omega = d\eta$. Finally, an element of $\Omega^{\bullet}_{\ch, \cpt}(X)$ is a form $\omega$ on $X$ whose support $K$ is compact and such that the \emph{relative} de-Rham class $[\omega] \in H^{\bullet}_{\dR}(X, X \setminus K)$ belongs to the image of the Chern character. We also recall that colimit on abelian groups is an exact functor, hence the colimit of a quotient is the quocient of the colimits. For example, $H^{\bullet}_{\dR, \cpt}(X) = \Omega^{\bullet}_{\cl, \cpt}(X)/\Omega^{\bullet}_{\ex, \cpt}(X)$.

We easily get the following natural transformations of $\mathcal{A}_{\Z}$-valued functors:
\begin{itemize}
	\item $I_{\cpt} \colon \hat{h}^{\bullet}_{\cpt}(X) \to h^{\bullet}_{\cpt}(X)$;
	\item $R_{\cpt} \colon \hat{h}^{\bullet}_{\cpt}(X) \to \Omega_{\cl, \cpt}^{\bullet}(X; \h^{\bullet}_{\R})$;
	\item $a_{\cpt} \colon \Omega^{\bullet-1}_{\cpt}(X; \h^{\bullet}_{\R})/\IIm(d) \to \hat{h}^{\bullet}_{\cpt}(X)$.
\end{itemize}
These transformations satisfy axioms analogous to (A1)--(A3) of definition \ref{RelDiffExt}, replacing the functors and the natural transformations involved with the corresponding compactly-supported version. The proof of (A1) and (A2) is straightforward and the proof of (A3) is analogous to \cite[Theorem 4.2]{BBSS}.

If $\hat{h}^{\bullet}$ is multiplicative, the module structure stated in definition \ref{MultDiffExt} induces the following natural module structure:
\[\begin{split} \cdot\, \colon & \hat{h}^{m}_{\cpt}(X) \times \hat{h}^{n}(X) \to \hat{h}^{n+m}_{\cpt}(X) \\
& ([\hat{\alpha}], \hat{\beta}) \mapsto [\hat{\alpha} \cdot \hat{\beta}].
\end{split}\]
Here naturality consists in the commutativity of the following diagram for any open embedding $\iota \colon Y \to X$:
\[\xymatrix{
  \hat{h}^{m}_{\cpt}(Y) \times \hat{h}^{n}(X) \ar[rr]^{\id \times \iota^{*}} \ar[d]_{\iota_{*} \times \id} & & \hat{h}^{m}_{\cpt}(Y) \times \hat{h}^{n}(Y) \ar[r] & \hat{h}^{n+m}_{\cpt}(Y) \ar[d]^{\iota_{*}} \\
  \hat{h}^{m}_{\cpt}(X) \times \hat{h}^{n}(X) \ar[rrr] & & & \hat{h}^{n+m}_{\cpt}(X).
}\]
Finally, we have the following natural homomorphism, which in general is neither injective nor surjective:
\[\begin{split}
   S \colon & \hat{h}^{\bullet}_{\cpt}(X) \to \hat{h}^{\bullet}(X) \\
   & [\hat{\alpha}] \mapsto \pi_{X}^{*}\hat{\alpha},
\end{split}\]
where $\pi_{X} \colon (X, \emptyset) \hookrightarrow (X, X \setminus K)$ is the natural inclusion of pairs and the naturality consists in the commutativity of the following diagram for any open embedding $\iota \colon Y \to X$:
\[\xymatrix{
  \hat{h}^{\bullet}_{\cpt}(Y) \ar[r]^{S} \ar[d]_{\iota_{*}} & \hat{h}^{\bullet}(Y) \\
  \hat{h}^{\bullet}_{\cpt}(X) \ar[r]^{S} & \hat{h}^{\bullet}(X). \ar[u]_{\iota^{*}}
}\]

\SkipTocEntry \subsection{Relative cohomology with compact support}\label{RelDiffCohCpt}

We can also define the relative version of compactly-supported cohomology. We only treat the case of a closed embedding $i \colon A \hookrightarrow X$. We define the functor $\Cc_{X, A} \colon \K_{X} \to \M_{2}^{\op}$, assigning to an object $H$ the function $i'_{H} \colon (X \setminus H) \sqcup A \to X$, that acts as $i_{H}$ on $X \setminus H$ and as the embedding $i$ on $A$, and to a morphism $K \subset H$ the natural morphism $i'_{KH} \colon i'_{H} \to i'_{K}$ defined by the following diagram:
\begin{equation}\label{DiagCptXY}
\xymatrix{
  (X \setminus H) \sqcup A \ar@{^(->}[d] \ar[rr]^(.6){i'_{H}} & & X \ar@{=}[d] \\
  (X \setminus K) \sqcup A \ar[rr]^(.6){i'_{K}} & & X.
}
\end{equation}
We call $\hat{h}_{0}^{\bullet}(i'_{K})$ the group formed by the classes $\hat{\alpha}$ such that $\cov(\hat{\alpha})$ vanishes on $X \setminus H$ (but not necessarily on $A$). We define:
\begin{equation}
  \hat{h}^{\bullet}_{\cpt}(X, A) := \colim \bigl( \, \hat{h}_{0}^{\bullet} \circ \Cc_{X, A} \colon \K_{X} \to \A_{\Z} \, \bigr).
\end{equation}
We can extend this definition to the category $\M'_{2}$ whose objects are smooth manifold pairs and whose morphisms are open embeddings. Let us fix a morphism $\iota \colon (Y, B) \hookrightarrow (X, A)$. For any compact subset $K \subset Y$, from the embedding of pairs $\iota_{K} \colon (Y, Y \setminus K) \hookrightarrow (X, X \setminus \iota(K))$, we get the induced morphism ${\iota'}_{K}^{*} \colon \hat{h}_{0}^{\bullet}(i'_{K, X}) \to \hat{h}_{0}^{\bullet}(i'_{K, Y})$, where we denoted by $i'_{K, X}$ the morphism $i'_{K}$ on the manifold $X$. The excision property of parallel classes \ref{ParExcision} easily extends to $\hat{h}_{0}$, hence ${\iota'}_{K}^{*}$ is an isomorphism. If $K \subset H$, the following diagram commutes:
\[\xymatrix{
  \hat{h}_{\ppar}^{0}(i'_{K, Y}) \ar[rr]^{({\iota'}_{K}^{*})^{-1}} \ar[d]^{i_{KH}^{*}} & & \hat{h}_{\ppar}^{\bullet}(i'_{K, X}) \ar[d]^{i_{KH}^{*}} \\
  \hat{h}_{\ppar}^{\bullet}(i'_{H, Y}) \ar[rr]^{({\iota'}_{H}^{*})^{-1}} & & \hat{h}_{\ppar}^{\bullet}(i'_{H, X}),
}\]
therefore we get an induced morphism between the colimits, i.e., $\iota_{*} \colon \hat{h}^{\bullet}_{\cpt}(Y, B) \to \hat{h}^{\bullet}_{\cpt}(X, A)$.

\SkipTocEntry \subsection{Extension with differential long exact sequence}

In \cite{BT} the authors proposed two definitions of relative Cheeger-Simons character. The first one corresponds to the relative differential extension of singular cohomology that we described above. The second one is interesting because it fits into a long exact sequence completely formed by differential cohomology groups and it can be generalized to any cohomology theory in the following way (see \cite{FR3} for Deligne cohomology and \cite{FR2} for the Hopkins-Singer model of any cohomology theory):
\begin{equation}\label{TypeIVDef}
	\hat{\hat{h}}^{\bullet}(\rho) := \frac{\Ker\bigl(\hat{h}^{\bullet}(\rho) \to \hat{h}^{\bullet}(A)\bigr)}{\IIm\bigl(\hat{h}^{\bullet-1}(X) \to \hat{h}^{\bullet}(\rho)\bigr)}.
\end{equation}
The morphism appearing in the numerator is the composition between the map $\hat{h}^{\bullet}(\rho) \to \hat{h}^{\bullet}(X)$, appearing in \eqref{LongExact2}, and the pull-back $\rho^{*} \colon \hat{h}^{\bullet}(X) \to \hat{h}^{\bullet}(A)$. By axiom (A4), it sends a class $\hat{\alpha}$, with curvature $(\omega, \eta)$, to $a(\eta)$. The morphism appearing in the denominator is the composition between the pull-back $\rho^{*} \colon \hat{h}^{\bullet-1}(X) \to \hat{h}^{\bullet-1}(A)$ and the Bockstein map of the sequence \eqref{LongExact2}, described explicitly by formula \eqref{RhoBockstein}.

\begin{Theorem} Given a smooth map $\rho \colon A \to X$, the following sequence is exact:
\begin{equation}\label{LongExact4}
	\cdots \longrightarrow \hat{\hat{h}}^{\bullet-1}(\rho) \longrightarrow \hat{h}^{\bullet-1}(X) \longrightarrow \hat{h}^{\bullet-1}(A) \longrightarrow \hat{\hat{h}}^{\bullet}(\rho) \longrightarrow \cdots.
\end{equation}
\end{Theorem}
\begin{proof} The same of \cite[Theorem 2.1]{FR3}.
\end{proof}

\begin{Rmk} In the papers \cite{FR3} and \cite{FR2} we also considered the extension that we called ``of type III''. For completeness, we briefly show how to reproduce it axiomatically. We need to recover the following exact sequences (see \cite{FR2}, sequences (3) and (5), and \cite{FR3}, sequences (6) and (8)):
\begin{small}
	\[\xymatrix{
	\cdots \ar[r] & \hat{h}^{\bullet-1}_{\fl}(X) \ar[r]^{\rho^{*}} \ar@{=}[d] & \hat{h}^{\bullet-1}(A) \ar[r]^{\beta} \ar@{->>}[d]^{I} & \hat{h}^{\bullet}(\rho) \ar[r] \ar@{->>}[d]^{\pi} & \hat{h}^{\bullet}(X) \ar[r] \ar@{=}[d] & h^{\bullet}(A) \ar[r] \ar@{=}[d] & \cdots \\
	\cdots \ar[r] & \hat{h}^{\bullet-1}_{\fl}(X) \ar[r] & h^{\bullet-1}(A) \ar[r]^{\beta'} & \hat{h}^{\bullet}_{III}(\rho) \ar[r] & \hat{h}^{\bullet}(X) \ar[r] & h^{\bullet}(A) \ar[r] & \cdots.
}\]
\end{small}

It follows by diagram chasing that we have an embedding $\frac{\hat{h}^{\bullet-1}(A)}{\rho^{*}\hat{h}^{\bullet-1}_{\fl}(X)} \hookrightarrow \hat{h}^{\bullet}(\rho)$ whose image contains $\Ker(\pi)$. Moreover, by exactness $\pi(\beta(\hat{\alpha})) = 0$ if and only if $\beta'(I(\hat{\alpha})) = 0$, if and only if $I(\hat{\alpha})$ is a restriction of a torsion class on $X$. Therefore, we define $\hat{h}^{\bullet}_{III}(\rho)$ as the quotient of $\hat{h}^{\bullet}(\rho)$ by the subgroup of $\frac{\hat{h}^{\bullet-1}(A)}{\rho^{*}\hat{h}^{\bullet-1}_{\fl}(X)}$ formed by the classes $[\hat{\alpha}]$ such that $I(\hat{\alpha})$ extends to a torsion class on $X$.
\end{Rmk}

\section{Orientation and integration}\label{SecOrientInt}

Following \cite[sec.\ 4.8-4.10]{Bunke}, we briefly recall the topological notions of orientation and integration and we extend them to the differential case, as we did in \cite[chap.\ 3]{FR}; actually, we slightly modify some definitions, in order to describe in a more unitary way the cases of manifolds without boundary, with boundary and (partially) with corners. We will use the expression ``compact manifold'' to indicate a compact smooth manifold with corners (in particular, with or without boundary). All of the statements can be easily generalized removing the compactness hypothesis, but this is the only case we will need.

\SkipTocEntry \subsection{Topological orientation of a vector bundle}

Let $X$ be a compact manifold and $\pi \colon E \to X$ a real vector bundle of rank $n$. The bundle $E$ is \emph{orientable} with respect to a multiplicative cohomology theory $h^{\bullet}$ if there exists a \emph{Thom class} $u \in h^{n}_{\cpt}(E)$ \cite[p.\ 253]{Rudyak}. We define the \emph{Thom isomorphism} $T \colon h^{\bullet}(X) \to h^{\bullet+n}_{\cpt}(E)$, $\alpha \mapsto u \cdot \pi^{*}\alpha$, and we call \emph{integration map} its inverse $\int_{E/X} \colon h^{\bullet}_{\cpt}(E) \to h^{\bullet-n}(X)$, $u \cdot \pi^{*}\alpha \mapsto \alpha$. If the characteristic of $\h^{\bullet}$ is $0$, the $n$-degree component of $\ch \,u$ defines an orientation of $E$ in the usual sense, hence it is possible to integrate a compactly-supported form fibre-wise. We define the \emph{Todd class} $\Td(u) := \int_{E/X} \ch \, u \in H^{0}_{\dR}(X; \h^{\bullet}_{\R})$. The following formula holds:
\begin{equation}\label{ChIntegral}
	\int_{E/X} \ch \, \alpha = \Td(u) \cdot \biggl( \ch \int_{E/X} \alpha \biggr).
\end{equation}
\begin{Lemma}[2x3 principle]\label{Rule23Top} Given two bundles $E, F \to X$, we call $p_{E} \colon E \oplus F \to E$ and $p_{F} \colon E \oplus F \to F$ the natural projections. Let $(u, v, w)$ be a triple of Thom classes on $E$, $F$ and $E \oplus F$ respectively, such that $w = p_{E}^{*}u \cdot p_{F}^{*}v$. Two elements of such a triple uniquely determine the third one.
\end{Lemma}
For the proof see \cite[prop.\ 1.10 p.\ 307]{Rudyak}.

\SkipTocEntry \subsection{Topological orientation of smooth maps}

\begin{Def}\label{TopOrientedMap} A \emph{representative of an $h^{\bullet}$-orientation} of a smooth neat map between compact manifolds $f \colon Y \to X$ is the datum of:
\begin{itemize}
	\item a neat embedding $\iota \colon Y \hookrightarrow X \times \R^{N}$, for any $N \in \N$, such that $\pi_{X} \circ \iota = f$;
	\item a Thom class $u$ of the normal bundle $N_{\iota(Y)}(X \times \R^{N})$;
	\item a diffeomorphism $\varphi \colon N_{\iota(Y)}(X \times \R^{N}) \to U$, for $U$ a neat tubular neighbourhood of $\iota(Y)$ in $X \times \R^{N}$.
\end{itemize}
\end{Def}
We now introduce a suitable equivalence relation among representatives of orientations. Let us consider a representative $(J, U, \Phi)$ of an $h^{\bullet}$-orientation of $\id \times f: I \times Y \rightarrow I \times X$ and a neighbourhood $V \subset I$ of $\{0, 1\}$. We say that the representative is \emph{proper on $V$} if a vector $(x,v)_{(t,y)} \in N_{\iota(V \times Y)}(V \times X \times \R^{N})_{\iota(t,y)}$ is sent by $\Phi$ to a point $\Phi((x,v)_{(t,y)}) \in V \times X \times \R^{N}$ whose first component is $t$. This means that the following diagram commutes:
\begin{equation}\label{PropRepres2}
	\xymatrix{
	N_{\iota(V \times Y)}(V \times X \times \R^{N}) \ar[r]^(.77){\Phi} \ar[d]_{\pi_{N}} & U \ar[d]^{\pi_{I}} \\
	\iota(V \times Y) \ar[r]^(.6){\pi_{I}} & I.
}\end{equation}
In this case, calling $f_{0} := \id_{\{0\}} \times f$ and $f_{1} := \id_{\{1\}} \times f$, we can define the restrictions $(J, U, \Phi)\vert_{f_{0}}$ and $(J, U, \Phi)\vert_{f_{1}}$.

\begin{Def}\label{HomotopyOrientations} A \emph{homotopy} between two representatives $(\iota, u, \varphi)$ and $(\iota', u', \varphi')$ of an $h^{\bullet}$-orientation of $f \colon Y \to X$ is a representative $(J, U, \Phi)$ of an $h^{\bullet}$-orientation of $\id \times f \colon I \times Y \to I \times X$, such that:
\begin{itemize}
	\item $(J, U, \Phi)$ is proper over a neighbourhood $V \subset I$ of $\{0, 1\}$;
	\item $(J, U, \Phi)\vert_{f_{0}} = (\iota, u, \varphi)$ and $(J, U, \Phi)\vert_{f_{1}} = (\iota', u', \varphi')$.
\end{itemize}
\end{Def}
On the trivial bundle $X \times \R^{N}$ there is a canonical Thom class, defined in the following way. On $pt \times \R^{N}$, whose compactification is $pt \times S^{N}$, we put the class $u_{0} \in \tilde{h}^{N}(S^{N})$ corresponding to the suspension of $1 \in \h^{0}$. Then, we put on $X \times \R^{N}$ the class $\pi_{\R^{N}}^{*}u_{0}$.
\begin{Def}\label{EquivStab} Let us consider a representative $(\iota, u, \varphi)$, with $\iota: Y \hookrightarrow X \times \R^{N}$.
\begin{itemize}
	\item For any $L \in \N$, we define $\iota' \colon Y \hookrightarrow X \times \R^{N + L}$ by $\iota'(y) := (\iota(y), 0)$. Then $N_{\iota'(Y)}(X \times \R^{N+L}) \simeq N_{\iota(Y)}(X \times \R^{N}) \oplus (\iota(Y) \times \R^{L})$.
	\item We put the canonical orientation $u_{0}$ on the trivial bundle $\iota(Y) \times \R^{L}$, and the orientation $u'$ induced by $u$ and $u_{0}$ on $N_{\iota'(Y)}(X \times \R^{N+L})$.
	\item Finally, for $v_{y} \in N_{\iota(Y)}(X \times \R^{N})$ and $w \in \R^{L}$, we define $\varphi'(v_{y}, w) := (\varphi(v_{y}), w) \in X \times \R^{N+L}$.
\end{itemize}
The representative $(\iota', u', \varphi')$ is called \emph{equivalent by stabilization} to $(\iota, u, \varphi)$. \end{Def}

\begin{Def} A \emph{$h^{\bullet}$-orientation} on $f \colon Y \to X$ is an equivalence class $[\iota, u, \varphi]$ of representatives, up to the equivalence relation generated by homotopy and stabilization.
\end{Def}

Because of the uniqueness up to homotopy and stabilization of the tubular neighbourhood and of the diffeomorphism with the normal bundle, the class $[\iota, u, \varphi]$ does not depend on $\varphi$, hence we denote it by $[\iota, u]$. Moreover, any two embeddings $\iota$ and $\iota'$ become equivalent by homotopy and stabilization, therefore the meaningful datum is $u$.

\begin{Rmk}\label{OrInducedBdFunction} If $X$ and $Y$ are manifolds with boundary, an orientation on $f \colon Y \to X$ canonically induces an orientation on $\partial f := f\vert_{\partial Y} \colon \partial Y \to \partial X$. In fact, fixing a representative $(\iota, u, \varphi)$ for $f$, by neatness $\iota$ restricts to $\iota' \colon \partial X \hookrightarrow \partial Y \times \R^{N}$. The normal bundle and the tubular neighbourhood, being neat, restrict to the boundary too, hence we get a representative $(\iota', u', \varphi')$ for $\partial f$. Any homotopy of representatives, being neat, determines a homotopy on the boundary, therefore the resulting orientation of $\partial f$ is well-defined. A similar remark holds when $X$ and $Y$ have corners, but we need to be more careful in defining $\partial f$. We omit the details, since they are irrelevant for the present paper.
\end{Rmk}

\begin{Def}\label{OrientedMapComposition} Let $f \colon Y \to X$ and $g \colon X \to W$ be $h^{\bullet}$-oriented neat maps, with orientations $[\iota, u]$ and $[\kappa, v]$, where $\iota \colon Y \hookrightarrow X \times \R^{N}$ and $\kappa \colon X \hookrightarrow W \times \R^{L}$. There is a naturally induced $h^{\bullet}$-orientation on $g \circ f \colon Y \to W$, that we denote by $[\kappa, v][\iota, u]$, defined in the following way:
\begin{itemize}
	\item we choose the embedding $\xi = (\kappa, \id_{\R^{N}}) \circ \iota \colon Y \hookrightarrow W \times \R^{L+N}$;
	\item on the normal bundle $N_{\xi(Y)}(W \times \R^{L+N}) \simeq N_{\iota(Y)}(X \times \R^{N}) \oplus N_{\kappa(X) \times \R^{N}}(W \times \R^{L+N})\vert_{\xi(Y)} \simeq N_{\iota(Y)}(X \times \R^{N}) \oplus (\pi^{*}_{L}N_{\kappa(X)}(W \times \R^{L}))\vert_{\xi(Y)}$, for $\pi_{L} \colon \R^{L+N} \to \R^{L}$, we put the Thom class $w$ induced from the ones on $N_{\iota(Y)}(X \times \R^{N})$ and $N_{\kappa(X)}(W \times \R^{L})$.
\end{itemize}
We set $[\kappa, v][\iota, u] := [\xi, w]$.
\end{Def}
The following lemma is a consequence of lemma \ref{Rule23Top} and of the uniqueness up to homotopy and stabilization of the embedding $\iota$.
\begin{Lemma}[2x3 principle]\label{Rule23TopMaps} Let $f \colon Y \to X$ and $g \colon X \to W$ be $h^{\bullet}$-oriented neat maps, with orientations $[\iota, u]$ and $[\kappa, v]$, and let $[\xi, w] := [\kappa, v][\iota, u]$ be the orientation induced on $g \circ f$. Two elements of the triple $([\iota, u], [\kappa, v], [\xi, w])$ uniquely determine the third one.
\end{Lemma}
For the proof see \cite[theorem 5.24 p.\ 233]{Karoubi}. Finally, let us consider two smooth neat maps $f, g \colon Y \to X$, with representatives $(\iota, u, \varphi)$ and $(\iota', u', \varphi')$ respectively. A \emph{homotopy} between $(\iota, u, \varphi)$ and $(\iota', u', \varphi')$ is defined as in \ref{HomotopyOrientations}, replacing $\id \times f$ with a smooth neat homotopy $F \colon I \times Y \to I \times X$ between $f$ and $g$.\footnote{A homotopy is usually defined as a function $F' \colon I \times Y \to X$, but we consider the function $F \colon I \times Y \to I \times X$, $(t, y) \mapsto (t, F'(t, y))$.} The existence of such a homotopy only depends on the equivalence classes $[\iota, u]$ and $[\iota', u']$, therefore we can give the following definition.

\begin{Def}\label{HomotopyOriented} Two smooth neat $h^{\bullet}$-oriented maps $f, g \colon Y \to X$ are \emph{homotopic as $h^{\bullet}$-oriented maps} if there exists a homotopy between any two representatives of the orientations of $f$ and $g$.
\end{Def}

\begin{Rmk}\label{RmkHomBd} We remark that, since a homotopy must be neat from $I \times Y$ to $I \times X$ by definition, it restricts to the boundary, thus it is a homotopy of maps of pairs $f, g \colon (Y, \partial Y) \to (X, \partial X)$. In particular, it induces a homotopy between $\partial f$ and $\partial g$.
\end{Rmk}

\SkipTocEntry \subsection{Topological orientation of smooth manifolds}

In this subsection we discuss separately the cases of manifolds without boundary, with boundary and (partially) with corners.
\begin{Def}\label{OrientedManifold} An $h^{\bullet}$-orientation of a manifold without boundary $X$ is an $h^{\bullet}$-orientation of the map $p_{X} \colon X \to \{pt\}$.
\end{Def}
By definition, giving an orientation to $p_{X}$ means fixing an orientation $u$ on the (stable) normal bundle of $X$; when $u$ has been fixed, we set $\Td(X) := \Td(u)$.

Given a manifold with boundary $X$, we recall that a \emph{defining function for the boundary} is a smooth neat map $\Phi \colon X \to I$ such that $\partial X = \Phi^{-1}\{0\}$ (by neatness, it follows that $\Phi^{-1}\{1\} = \emptyset$).
\begin{Def}\label{OrientedManifoldBoundary} An \emph{$h^{\bullet}$-orientation} on a smooth manifold with boundary $X$ is a homotopy class of $h^{\bullet}$-oriented defining functions for the boundary (see def.\ \ref{HomotopyOriented}).
\end{Def}
It easy to verify that any two defining functions are neatly homotopic, therefore the only meaningful datum is again the Thom class $u$.
\begin{Rmk}\label{OrManBdRmk} We set $\Hh^{N} := \{(x_{1}, \ldots, x_{N}) \in \R^{N}: x_{N} \geq 0\}$ (it is the local model of an $n$-dimensional manifold with boundary). Definition \ref{OrientedManifoldBoundary} is equivalent to fixing a neat embedding $\iota \colon X \hookrightarrow \Hh^{N}$, a Thom class on the normal bundle and a difeomorphism with a neat tubular neighbourhood, up to homotopy and stabilization. In fact, if we fix an $h^{\bullet}$-orientation $[\iota, u, \varphi]$ of a defining map $\Phi \colon X \to I$, following definition \ref{OrientedManifoldBoundary}, we have that $\iota \colon X \hookrightarrow I \times \R^{N}$. Since $\Phi^{-1}\{1\} = \emptyset$, the image of $\iota$ is contained in $[0, 1) \times \R^{N} \simeq \Hh^{N+1}$. This confirms that definition \ref{OrientedManifoldBoundary} is natural.
\end{Rmk}

\begin{Rmk}\label{OrInducedBd} It follows from remark \ref{OrInducedBdFunction} that an orientation on a manifold with boundary canonically induces an orientation on the boundary. In particular, let us fix a defining function $\Phi \colon X \to I$ and an orientation $[\iota, u]$, with $\iota \colon X \hookrightarrow I \times \R^{N}$. We call $i_{\partial X} \colon \partial X \hookrightarrow X$ the natural embedding and we set $\iota' := \iota \circ i_{\partial X} \colon \partial X \hookrightarrow \{0\} \times \R^{N}$ and $u' := u\vert_{\partial X}$. We get the orientation $[\iota', u']$ of $\partial X$.
\end{Rmk}

\begin{Rmk}\label{BdPartCase} If we apply definition \ref{OrientedManifoldBoundary} to a manifold without boundary (which is a particular case of a manifold with boundary), we get a function $\Phi \colon X \to I$ whose image is contained in $(0, 1)$, the latter being diffeomorphic to $\R$. A representative $(\iota, u, \varphi)$ of an orientation of $\Phi$ is constructed from the embedding $\iota \colon X \hookrightarrow (0,1) \times \R^{N} \simeq \{pt\} \times \R^{N+1}$, therefore it can be thought of as a representative of an orientation of $p_{X} \colon X \to \{pt\}$. Any two such defining functions are homotopic, $(0, 1)$ being contractible, and a homotopy between them determines a homotopy of representatives of an orientation of $p_{X} \colon X \to \{pt\}$. This shows that definition \ref{OrientedManifold} is (equivalent to) a particular case of definition \ref{OrientedManifoldBoundary}.
\end{Rmk}

With respect to manifold with corners, we just consider the following case, that will be useful in order to define the generalized Cheeger-Simons characters.
\begin{Def}\label{DefSplitBd} A \emph{manifold with split boundary} is a triple of manifolds $(X, M, N)$ such that:
\begin{itemize}
	\item $X$ is a manifold with corners and $M$ and $N$ are manifolds with boundary;
	\item $\partial X = M \cup N$, $M$ and $N$ being embedded sub-manifolds (not neat in general) of $\partial X$ of codimension $0$;
	\item $\partial M = \partial N = M \cap N$;
	\item $\{\textnormal{corners of } X\} \subset M \cap N$.
\end{itemize}
\end{Def}
A \emph{defining function for the boundary} of $(X, M, N)$ is a smooth neat map $\Phi \colon X \to I \times I$ such that $M = \Phi^{-1}(I \times \{0\})$ and $N = \Phi^{-1}(\{0\} \times I)$ (by neatness, it follows that $\Phi^{-1}(I \times \{1\}) = \Phi^{-1}(\{1\} \times I) = \emptyset$). The definition of \emph{$h^{\bullet}$-orientation} is analogous to \ref{OrientedManifoldBoundary}. Remark \ref{OrManBdRmk} keeps on holding, replacing $\Hh^{N}$ by $\Hh^{N,2} := \{(x_{1}, \ldots, x_{N}) \in \R^{N}: x_{N-1}, x_{N} \geq 0\}$. Remark \ref{OrInducedBd} holds in the sense that an orientation of $(X, M, N)$ induces an orientation of $M$ and one of $N$, with defining functions (up to homotopy) $\Phi_{M} := \Phi\vert_{M} \colon M \to I \times \{0\} \approx I$ and $\Phi_{N} := \Phi\vert_{N} \colon N \to \{0\} \times I \approx I$ respectively. Finally, remark \ref{BdPartCase} holds too, in the sense that, setting $N = \emptyset$, we recover the notion of orientation for a manifold with boundary.

\SkipTocEntry \subsection{Topological integration}

Let $f \colon Y \to X$ be a neat map. If we fix a representative $(\iota, u, \varphi)$ of an orientation of $f$, the \emph{Gysin map} $f_{!} \colon h^{\bullet}(Y) \to h^{\bullet - n}(X)$, for $n = \dim Y - \dim X$, is defined as:
\begin{equation}\label{GysinMapTop}
	f_{!}(\alpha) = \int_{\R^{N}}i_{*}\varphi_{*}(u \cdot \pi^{*}\alpha),
\end{equation}
$i$ being the natural inclusion of the tubular neighbourhood $i \colon U \hookrightarrow X \times \R^{N}$, inducing a push-forward in compactly-supported cohomology. The Gysin map $f_{!}$ only depends on the $h^{\bullet}$-orientation $[\iota, u]$, not on the specific representative (\cite[theorem 5.24 p.\ 233]{Karoubi}, \cite[sec.\ 4.9]{Bunke}). If $Y$ and $X$ are oriented, because of the 2x3 principle a map $f \colon Y \to X$ inherits an orientation, hence the Gysin map is well-defined.

\begin{Theorem}\label{fpTopProperties} Let $f \colon Y \to X$ be a neat $h^{\bullet}$-oriented map  of compact manifolds.
\begin{itemize}
	\item The Gysin map $f_{!}$ only depends on the homotopy class of $f$ as an $h^{\bullet}$-oriented map.
	\item The Gysin map is a morphism of $h^{\bullet}(X)$-modules, i.e., given $\alpha \in h^{\bullet}(Y)$ and $\beta \in h^{\bullet}(X)$:
	\[f_{!}(\alpha \cdot f^{*}\beta) = f_{!}\alpha \cdot \beta.
\]
	\item Given another neat $h^{\bullet}$-oriented map $g \colon Z \to Y$ and endowing $f \circ g$ of the naturally induced orientation (def.\ \ref{OrientedMapComposition}), we have $(f \circ g)_{!} = f_{!} \circ g_{!}$.
\end{itemize}
\end{Theorem}

For the proof see \cite[theorem 5.24 p.\ 233]{Karoubi}. If $X$ and $Y$ are manifolds with boundary, considering remark \ref{OrInducedBdFunction}, one has, for every $\alpha \in h^{\bullet}(Y)$:
\begin{equation}\label{RestrictionBoundaryGysin}
	(\partial f)_{!}(\alpha\vert_{\partial Y}) = (f_{!}\alpha)\vert_{\partial X}.
\end{equation}
Such a formula is due to the fact that all the structures involved in the definition of $(\partial f)_{!}$ are the restrictions to the boundary of the corresponding structures for $f_{!}$. A similar result holds when $X$ and $Y$ have corners.

\SkipTocEntry \subsection{Differential orientation of a vector bundle}

If we consider a differential refinement $\hat{h}^{\bullet}$ of $h^{\bullet}$, in order to orient a vector bundle one just has to refine a Thom class $u$ to a \emph{differential Thom class}.
\begin{Def} Let $\hat{h}^{\bullet}$ be a multiplicative differential extension of $h^{\bullet}$. A \emph{differential Thom class} of $E$ is a compactly supported class $\hat{u} \in \hat{h}^{n}_{\cpt}(E)$ such that $I(\hat{u}) \in h^{n}_{\cpt}(E)$ is a Thom class for $h^{\bullet}$.
\end{Def}
Using the product $\hat{h}^{\bullet}_{\cpt}(E) \otimes_{\Z} \hat{h}^{\bullet}(E) \rightarrow \hat{h}^{\bullet}_{\cpt}(E)$, we define the differential Thom morphism, which is not surjective any more, as $\hat{\alpha} \mapsto \hat{u} \cdot \pi^{*}\hat{\alpha}$. We define the \emph{Todd class} $\Td(\hat{u}) := \int_{E/X} R(\hat{u}) \in \Omega^{0}_{\cl}(X; \h^{\bullet}_{\R})$. It follows that $I(\Td(\hat{u})) = \Td(I(\hat{u}))$.

\begin{Def}\label{HomotopyThom} Let $\pi_{X} \colon I \times X \to X$ be the natural projection and $i_{0}, i_{1} \colon X \to I \times X$ the natural embeddings. Two differential Thom classes $\hat{u}, \hat{u}' \in \hat{h}^{n}_{\cpt}(E)$ are \emph{homotopic} if there exists a Thom class $\hat{U} \in \hat{h}^{n}_{\cpt}(\pi_{X}^{*}E)$ such that $i_{0}^{*}\hat{U} = \hat{u}$, $i_{1}^{*}\hat{U} = \hat{u}'$ and $\Td(\hat{U}) = \pi_{X}^{*}\Td(\hat{u})$.
\end{Def}

\begin{Lemma}[2x3 principle]\label{Rule23Diff} Given two bundles $E, F \to X$, with projections $p_{E} \colon E \oplus F \to E$ and $p_{F} \colon E \oplus F \to F$, we consider a triple $(\hat{u}, \hat{v}, \hat{w})$ of differential Thom classes on $E$, $F$ and $E \oplus F$ respectively, such that $\hat{w}$ is homotopic to $p_{E}^{*}\hat{u} \cdot p_{F}^{*}\hat{v}$. Two elements of such a triple uniquely determine the third one up to homotopy.
\end{Lemma}

\begin{Lemma}\label{CanOrientTriv} On the trivial bundle $X \times \R^{N}$ there is a canonical homotopy class of differential Thom classes, refining the canonical topological one.
\end{Lemma}

For the proofs see \cite[prob.\ 4.187]{Bunke} and \cite[cor.\ 3.19]{FR}.

\SkipTocEntry \subsection{Differential orientation of smooth maps}

We define a \emph{representative of an $\hat{h}^{\bullet}$-orientation of $f$} as in definition \ref{TopOrientedMap}, but considering a differential Thom class. Fixing such a representative $(\iota, \hat{u}, \varphi)$, the Gysin map $f_{!} \colon \hat{h}^{\bullet}(Y) \to \hat{h}^{\bullet - n}(X)$ is well-defined via formula \eqref{GysinMapTop} applied to differential classes. Moreover, we have the following natural map on differential forms, called \emph{curvature map}:
\begin{equation}\label{CurvatureMap}
\begin{split}
	R_{(\iota, \hat{u}, \varphi)} \colon & \Omega^{\bullet}(Y; \h^{\bullet}_{\R}) \to \Omega^{\bullet-n}(X; \h^{\bullet}_{\R}) \\
	& \omega \mapsto \int_{X \times \R^{N}/X} i_{*}\varphi_{*}(R(\hat{u}) \wedge \pi^{*}\omega).
\end{split}
\end{equation}
The following definition is analogous to \ref{HomotopyOrientations}, but it takes into account the curvature map.
\begin{Def}\label{HomotopyDiffOrientations} A \emph{homotopy} between two representatives $(\iota, \hat{u}, \varphi)$ and $(\iota', \hat{u}', \varphi')$ of an $\hat{h}^{\bullet}$-orientation of $f \colon Y \to X$ is a representative $(J, \hat{U}, \Phi)$ of an $\hat{h}^{\bullet}$-orientation of $\id \times f \colon I \times Y \to I \times X$, such that:
\begin{itemize}
	\item $(J, I(\hat{U}), \Phi)$ is proper over a neighbourhood $V \subset I$ of $\{0, 1\}$;
	\item $(J, \hat{U}, \Phi)\vert_{f_{0}} = (\iota, \hat{u}, \varphi)$ and $(J, \hat{U}, \Phi)\vert_{f_{1}} = (\iota', \hat{u}', \varphi')$;
	\item $\pi_{X}^{*} \circ R_{(\iota, \hat{u}, \varphi)} = R_{(J, \hat{U}, \Phi)} \circ \pi_{Y}^{*}$.
\end{itemize}
\end{Def}
In particular, it follows that $R_{(\iota, \hat{u}, \varphi)} = R_{(\iota', \hat{u}', \varphi')}$. Thanks to lemma \ref{CanOrientTriv}, we define the equivalence of representatives up to stabilization as in the topological framework (def.\ \ref{EquivStab}).
\begin{Def} An \emph{$\hat{h}^{\bullet}$-orientation} on $f \colon Y \to X$ is an equivalence class $[\iota, \hat{u}, \varphi]$ of representatives, up to the equivalence relation generated by homotopy and stabilization.
\end{Def}
\begin{Rmk} By construction the curvature map \eqref{CurvatureMap} only depends on the orientation, not on the specific representative, therefore we will denote it by $R_{[\iota, \hat{u}, \varphi]}$.
\end{Rmk}

Remark \ref{OrInducedBdFunction} keeps on holding. Now we need to extend to the differential setting the fundamental properties of topological orientation, in particular definition \ref{OrientedMapComposition} and lemma \ref{Rule23TopMaps}. This can be done adding the following hypothesis, that will force us to focus on submersions. Let us consider a vector $v_{y} \in N_{\iota(Y)}(X \times \R^{N})_{\iota(y)}$. It is sent by $\varphi$, as defined in \ref{TopOrientedMap}, to a point $\varphi(v_{y}) \in X \times \R^{N}$. If $f$ is a submersion, we can require that the first component of $\varphi(v_{y})$ is $f(y)$. This means that the following diagram commutes:
\begin{equation}\label{PropRepres}
	\xymatrix{
	N_{\iota(Y)}(X \times \R^{N}) \ar[r]^(.7){\varphi} \ar[d]_{\pi_{N}} & U \ar[d]^{\pi_{X}} \\
	\iota(Y) \ar[r]^{\pi_{X}} & X.
}\end{equation}
\begin{Def}\label{PropDef} A representative of an $\hat{h}^{\bullet}$-orientation of a smooth neat map $f \colon Y \to X$ is \emph{proper} if diagram \eqref{PropRepres} commutes.\footnote{The same definition could be given for a representative of an $h^{\bullet}$-orientation, but it is more relevant in the differential framework.}
\end{Def}

\begin{Lemma} If $(\iota, \hat{u}, \varphi)$ is proper, then:
\begin{equation}\label{IntProp}
	R_{[\iota, \hat{u}, \varphi]}(\omega) = \int_{Y/X} \Td(\hat{u}) \wedge \omega.
\end{equation}
\end{Lemma}
\begin{Corollary}\label{HomotTHomOr} Let $(\iota, \hat{u}, \varphi)$ and $(\iota, \hat{u}', \varphi')$ be two \emph{proper} representatives of an $\hat{h}^{\bullet}$-orientation of a smooth neat map $f \colon Y \to X$, such that $\hat{u}$ and $\hat{u}'$ are homotopic as differential Thom classes. Then the two representatives are homotopic (independently of $\varphi$ and $\varphi'$), thus $[\iota, \hat{u}, \varphi] = [\iota, \hat{u}', \varphi']$.
\end{Corollary}

For the proof see \cite[Lemma 3.25]{FR} and \cite[Problems 4.221 and 4.224]{Bunke}.

\begin{Lemma}\label{SubProp} Let $f \colon Y \to X$ be a neat submersion. For any neat embedding $\iota \colon Y \hookrightarrow X \times \R^{N}$ and any differential Thom class $\hat{u}$ of the normal bundle, there exists a \emph{proper} representative $(\iota, \hat{u}, \varphi)$ of an $\hat{h}^{\bullet}$-orientation of $f$.
\end{Lemma}

For the proof see \cite[Lemma 3.23]{FR} and \cite[paragraph before Problem 4.219]{Bunke}. Because of lemma \ref{SubProp} and corollary \ref{HomotTHomOr}, given a neat submersion $f \colon Y \to X$, a neat embedding $\iota \colon Y \hookrightarrow X \times \R^{N}$ and any differential Thom class $\hat{u}$, the $\hat{h}^{\bullet}$-orientation $[\iota, \hat{u}]$ is well-defined, extending $(\iota, \hat{u})$ to any proper representative $(\iota, \hat{u}, \varphi)$. The orientation $[\iota, \hat{u}]$ only depends on the homotopy class of $\hat{u}$. Moreover, if $f \colon Y \to X$ and $g \colon X \to W$ are $\hat{h}^{\bullet}$-oriented neat submersions, there is a naturally induced $\hat{h}^{\bullet}$-orientation on $g \circ f \colon Y \to W$, defined as in \ref{OrientedMapComposition}. The following lemma is a consequence of lemma \ref{Rule23Diff} and of the uniqueness up to homotopy and stabilization of the embedding $\iota$.
\begin{Lemma}[2x3 principle]\label{Rule23Subm} Let $f \colon Y \to X$ and $g \colon X \to W$ be $\hat{h}^{\bullet}$-oriented neat submersions, with orientations $[\iota, \hat{u}]$ and $[\kappa, \hat{v}]$, and let $[\xi, \hat{w}]$ be the orientation induced on $g \circ f$. Two elements of the triple $([\iota, \hat{u}], [\kappa, \hat{v}], [\xi, \hat{w}])$ uniquely determine the third one.
\end{Lemma}

Finally, definition \ref{HomotopyOriented} can be easily adapted to the differential framework, considering a smooth neat homotopy with a differential orientation. When such a definition holds, two maps $f, g \colon Y \to X$ are \emph{homotopic as $\hat{h}^{\bullet}$-oriented maps}.

\SkipTocEntry \subsection{Differential orientation of smooth manifolds}

We define the notion of differential orientation of a manifold without boundary as in the topological case (def.\ \ref{OrientedManifold}); when the orientation $\hat{u}$ of the stable normal bundle has been fixed, we set $\Td(X) := \Td(\hat{u})$. When $X$ has a boundary, we have to take into account that a defining map for the boundary is not a submersion in general, therefore we cannot apply many results cited above. For this reason, we slightly modify the definition of orientation. Following definition \eqref{CurvatureMap}, the curvature map should be $\omega \mapsto \int_{I \times \R^{N}/I} i_{*}\varphi_{*}(R(\hat{u}) \wedge \pi^{*}\omega)$, but we also integrate on $I$ the result:
\begin{equation}\label{CurvMapBd}
\begin{split}
	R^{\partial}_{(\iota, \hat{u}, \varphi)} \colon & \Omega^{\bullet}(X; \h^{\bullet}_{\R}) \to \Omega^{\bullet-n}(pt; \h^{\bullet}_{\R}) \\
	& \omega \mapsto \int_{0}^{1} \int_{I \times \R^{N}/I} i_{*}\varphi_{*}(R(\hat{u}) \wedge \pi^{*}\omega).
\end{split}
\end{equation}
\begin{Def}\label{OrientedDiffManifoldBoundary} An \emph{$\hat{h}^{\bullet}$-orientation} on a smooth manifold with boundary $X$ is a homotopy class of $\hat{h}^{\bullet}$-oriented defining functions for the boundary, considering the curvature map \eqref{CurvMapBd} in the definition of homotopy.
\end{Def}
This means that the curvature map from $X$ to the point must be constant along the homotopy, not the one from $X$ to $I$, as would follow from the definition without replacing the curvature map. The double integral in \eqref{CurvMapBd} is equivalent to the integral on the whole $I \times \R^{N}$. Considering remark \ref{OrManBdRmk}, we are just integrating on $\Hh^{N+1}$. It follows that:
\begin{equation}\label{RBdTodd}
	R^{\partial}_{(\iota, \hat{u}, \varphi)}(\omega) = \int_{N_{\iota(X)}\Hh^{N+1}} R(\hat{u}) \wedge \pi^{*}\omega = \int_{X} \biggl(\int_{N_{\iota(X)}\Hh^{N+1}/X}R(\hat{u})\biggr) \wedge \omega) = \int_{X} \Td(X) \wedge \omega.
\end{equation}
This result is analogous to formula \eqref{IntProp}, therefore we can state the following corollary, analogous to \ref{HomotTHomOr}.
\begin{Corollary}\label{HomotThomOr} Let $(\iota, \hat{u}, \varphi)$ and $(\iota, \hat{u}', \varphi')$ be two representatives of an $\hat{h}^{\bullet}$-orientation of the defining function $\Phi \colon X \to I$, such that $\hat{u}$ and $\hat{u}'$ are homotopic as differential Thom classes. Then the two representatives are homotopic (independently of $\varphi$ and $\varphi'$), thus $[\iota, \hat{u}, \varphi] = [\iota, \hat{u}', \varphi']$.
\end{Corollary}
It follows that an orientation of $X$ only depends on $\iota$ and $\hat{u}$, therefore an orientation on a neat submersion $f \colon Y \to X$ and an orientation on $X$ induce an orientation on $Y$ by definition \ref{OrientedMapComposition}. Because of corollary \ref{HomotTHomOr} and the uniqueness up to homotopy and stabilization of the embedding $\iota$, we get the following lemma, analogous to \ref{Rule23Subm}.
\begin{Lemma}[2x3 principle]\label{Rule23ManBd} Let $f \colon Y \to X$ be a neat submersions between manifolds with boundary. Let $[\iota, \hat{u}]$ be an orientation of $f$, $[\kappa, \hat{v}]$ an orientation of $X$, and let $[\xi, \hat{w}]$ be the orientation induced on $Y$. Two elements of the triple $([\iota, \hat{u}], [\kappa, \hat{v}], [\xi, \hat{w}])$ uniquely determine the third one.
\end{Lemma}
With this definition of the curvature map, remark \ref{BdPartCase} extends to the differential setting, i.e., an orientation of a manifold without boundary can be thought of as a particular case of an orientation of a manifold with boundary. This confirms the naturality of the definition. As well, remark \ref{BdPartCase} keeps on holding. Finally, in the case of manifolds with split boundary, we define a $\hat{h}^{\bullet}$-orientation in the same way, the curvature map \eqref{CurvMapBd} being defined integrating over $I \times I$.

\SkipTocEntry \subsection{Differential integration} 

The Gysin map $f_{!} \colon \hat{h}^{\bullet}(Y) \to \hat{h}^{\bullet - n}(X)$, for $n = \dim Y - \dim X$, is defined similarly to \eqref{GysinMapTop}, starting from a representative of an $\hat{h}^{\bullet}$-orientation:
\begin{equation}\label{GysinMapDif}
	f_{!}(\hat{\alpha}) = \int_{\R^{N}}i_{*}\varphi_{*}(\hat{u} \cdot \pi^{*}\hat{\alpha}).
\end{equation}
The integration map $\int_{\R^{N}} \colon \hat{h}^{\bullet+N}_{\cpt}(X \times \R^{N}) \to \hat{h}^{\bullet}(X)$ is defined as follows. The open embedding $j \colon \R^{N} \hookrightarrow (S^{1})^{N}$, defined through the embedding $\R \hookrightarrow \R^{+} \simeq S^{1}$ in each coordinate, induces the push-forward $(\id \times j)_{*} \colon \hat{h}^{\bullet}_{\cpt}(X \times \R^{N}) \to \hat{h}^{\bullet}(X \times (S^{1})^{N})$, thus we define
\begin{equation}\label{IntSn}
	\int_{\R^{N}}\hat{\alpha} := \int_{S^{1}} \cdots \int_{S^{1}} (\id \times j)_{*}\hat{\alpha}.
\end{equation}
It is easy to prove from the axioms that:
	\[R(f_{!}\hat{\alpha}) = R_{(\iota, \hat{u}, \varphi)}(R(\hat{\alpha})) \qquad f_{!}a(\omega) = a(R_{(\iota, \hat{u}, \varphi)}(\omega)),
\]
thus the following diagram commutes:
\begin{equation}\label{CurvAMaps}
	\xymatrix{
	\Omega^{\bullet-1}(Y; \h^{\bullet}_{\R})/\IIm(d) \ar[r]^(.65){a} \ar[d]^{R_{(\iota, \hat{u}, \varphi)}} & \hat{h}^{\bullet}(Y) \ar[r]^{I} \ar[d]^{f_{!}} \ar@/^2pc/[rr]^{R} & h^{\bullet}(Y) \ar[d]^{f_{!}} & \Omega_{\cl}^{\bullet}(Y; \h^{\bullet}_{\R}) \ar[d]^{R_{(\iota, \hat{u}, \varphi)}} \\
		\Omega^{\bullet-n-1}(X; \h^{\bullet}_{\R})/\IIm(d) \ar[r]^(.65){a} & \hat{h}^{\bullet-n}(X) \ar[r]^{I} \ar@/_2pc/[rr]_{R} & h^{\bullet-n}(X) & \Omega_{\cl}^{\bullet-n}(X; \h^{\bullet}_{\R}).
	}
\end{equation}
As a consequence of formula \eqref{HomFormula}, $f_{!}$ only depends on the $\hat{h}^{\bullet}$-orientation of $f$, not on the specific representative \cite[sec.\ 4.10]{Bunke}. We now consider a submersion $f \colon Y \to X$. In this case the Gysin map provides a good notion of integration.

\begin{Theorem}\label{fpDifProperties} Let $f \colon Y \to X$ be a neat $\hat{h}^{\bullet}$-oriented submersion between compact manifolds.
\begin{itemize}
	\item The Gysin map $f_{!}$ only depends on the homotopy class of $f$ as an $\hat{h}^{\bullet}$-oriented map.
	\item The Gysin map is a morphism of $\hat{h}^{\bullet}(X)$-modules, i.e., given $\hat{\alpha} \in h^{\bullet}(Y)$ and $\hat{\beta} \in \hat{h}^{\bullet}(X)$:
	\[f_{!}(\hat{\alpha} \cdot f^{*}\hat{\beta}) = f_{!}\hat{\alpha} \cdot \hat{\beta}.
\]
	\item Given another neat $\hat{h}^{\bullet}$-oriented map $g \colon Z \to Y$ and endowing $f \circ g$ of the naturally induced orientation (def.\ \ref{OrientedMapComposition}), we have $(f \circ g)_{!} = f_{!} \circ g_{!}$.
	\item We have that:
	\begin{equation}\label{RAIntSubmersion}
	R(f_{!}\hat{\alpha}) = \int_{Y/X} \Td(\hat{u}) \wedge R(\hat{\alpha}) \qquad f_{!}(a(\omega)) = a \biggl( \int_{Y/X} \Td(\hat{u}) \wedge \omega \biggr).
	\end{equation}
\end{itemize}
\end{Theorem}
For the proof see \cite[Lemmas 3.24 and 3.27]{FR} and \cite[Problems 4.219 and 4.233]{Bunke}. Equations \eqref{RAIntSubmersion} follows from formula \eqref{IntProp} and the commutativity of diagram \eqref{CurvAMaps}. Moreover, formula \eqref{RestrictionBoundaryGysin} keeps on holding.

\begin{Rmk} Let us consider a submersion $f \colon Y \to X$ between $\hat{h}^{\bullet}$-oriented manifolds. If $X$ and $Y$ have no boundary, since $p_{Y} = p_{X} \circ f$, it follows from lemma \ref{Rule23Subm} that $f$ inherits a unique orientation from the ones of $X$ and $Y$. Hence, the integration map $f_{!} \colon \hat{h}^{\bullet}(Y) \to \hat{h}^{\bullet-n}(X)$ is well-defined for submersions between compact $\hat{h}^{\bullet}$-oriented manifolds without boundary. If $X$ and $Y$ have boundary, the same result follows from \ref{Rule23ManBd}.
\end{Rmk}

\SkipTocEntry \subsection{Flat classes}\label{FlGysin}

The Gysin map $f_{!} \colon \hat{h}^{\bullet}(Y) \to \hat{h}^{\bullet - n}(X)$, defined in the previous section, depends on the $\hat{h}^{\bullet}$-orientation of $f$, but, if we restrict it to flat classes, it only depends on the topological $h^{\bullet}$-orientation. In fact, $\hat{h}^{\bullet}_{\fl}(X)$ has a natural graded-module structure over $h^{\bullet}(X)$, i.e., the product $h^{\bullet}(X) \otimes_{\Z} \hat{h}^{\bullet}_{\fl}(X) \to \hat{h}^{\bullet}_{\fl}(X)$ is well-defined. This can be easily proven in the two following steps.
\begin{itemize}
	\item The product of differential classes $\hat{h}^{\bullet}(X) \otimes \hat{h}^{\bullet}(X) \to \hat{h}^{\bullet}(X)$ restricts to the product $\hat{h}^{\bullet}(X) \otimes \hat{h}^{\bullet}_{\fl}(X) \to \hat{h}^{\bullet}_{\fl}(X)$, since, the curvature being multiplicative, if one of the two factors has vanishing curvature, also the result has.
	\item The product $\hat{\alpha} \cdot \hat{\beta}$, when $\hat{\beta}$ is flat, only depends on $I(\hat{\alpha})$. In fact, if $I(\hat{\alpha}) = 0$, then $\hat{\alpha} = a(\omega)$. Because of definition \ref{MultDiffExt}, we have $a(\omega) \cdot \hat{\beta} = a(\omega \wedge R(\hat{\beta})) = a(0) = 0$.
\end{itemize}
We can show in the same way that also the product $\hat{h}^{\bullet}_{\cpt}(E) \otimes_{\Z} \hat{h}^{\bullet}(E) \to \hat{h}^{\bullet}_{\cpt}(E)$ can be refined to $h^{\bullet}_{\cpt}(E) \otimes_{\Z} \hat{h}^{\bullet}_{\fl}(E) \to \hat{h}^{\bullet}_{\fl,\cpt}(E)$, therefore, given a real vector bundle $\pi \colon E \to X$ of rank $n$ with (topological) Thom class $u$, we define the Thom isomorphism:
	\[\begin{split}
	T_{\fl} \colon & \hat{h}^{\bullet}_{\fl}(X) \to \hat{h}^{\bullet + n}_{\fl,\cpt}(E) \\
	& \hat{\alpha} \mapsto u \cdot \pi^{*}\hat{\alpha}.
\end{split}\]
From this it easily follows that the Gysin map $f_{!}$, when applied to a flat class, only depends on the topological orientation of $f$. Lemma \ref{fpTopProperties} keeps on holding, with the same proof (for any $f$, not necessarily a submersion). With respect to the commutativity of diagram \eqref{CurvAMaps} (that, in the case of a submersion, leads to the last item of theorem \ref{fpDifProperties}), obviously the behaviour of the curvature is trivial in the flat case. About the map $a$, the commutativity of the diagram is equivalent to the following lemma (and, in the case of a submersion, to the right-hand side of equation \eqref{RAIntSubmersion}).
\begin{Lemma}\label{GysinC10} Given a $h^{\bullet}$-oriented smooth neat map $f \colon Y \to X$ and a class $\theta \in H^{\bullet-1}_{\dR}(Y; \h^{\bullet}_{\R})$, we have:
\begin{equation}\label{PushForwardDR}
	f_{!}(a(\theta)) = a(f_{!}(\Td(u) \wedge \theta)).
\end{equation}
Equivalently, for any $\alpha \in h^{\bullet}(X) \otimes_{\Z} \R$:\footnote{In equation \eqref{PushForwardChernC} we are considering the Chern character as defined on $h^{\bullet}(X) \otimes_{\Z} \R$, in which case it is an isomorphism. If we consider it as defined on $h^{\bullet}(X)$, then $a(\ch \alpha) = 0$, and formula \eqref{PushForwardChernC} implies coherently that $f_{!}(a(\ch\alpha)) = 0$.}
\begin{equation}\label{PushForwardChernC}
	f_{!}(a(\ch\alpha)) = a(\ch(f_{!}\alpha)).
\end{equation}
\end{Lemma}
\begin{proof} Let us consider a differential Thom class $\hat{u}$ of $N_{\iota(Y)}(X \times \R^{N})$ refining the orientation $u$ induced by the ones of $X$ and $Y$. We have:
\begin{align*}
	f_{!}(a(\theta)) &= \int_{\R^{N}} i_{*}\varphi_{*}(\hat{u} \cdot \pi^{*}a(\theta)) = \int_{\R^{N}} i_{*}\varphi_{*}(a([R(\hat{u})] \wedge \pi^{*}\theta)) = a\left(\int_{\R^{N}} i_{*}\varphi_{*}(\ch\,u \wedge \pi^{*}\theta)\right) \\
	&= a\left(\int_{\R^{N}} i_{*}\varphi_{*}(\ch^{(n)}u \wedge \pi^{*}(\Td(u) \wedge \theta)\right) = a(f_{!}(\Td(u) \wedge \theta)).
\end{align*}
This proves \eqref{PushForwardDR}. In order to prove \eqref{PushForwardChernC}, we observe that, since the fibre-wise integration and the Thom isomorphism are inverse to each other in de-Rham cohomology, we have that $\ch(u) = \ch^{(n)}u \wedge \Td(u)$, where $\ch^{n}(u)$ is the $n$-degree component of $\ch(u)$, that provides an ordinary orientation to $N_{\iota(Y)}(X \times \R^{N})$. Hence:
\begin{align*}
	\ch(f_{!}\alpha) &= \ch\int_{\R^{N}} i_{*}\varphi_{*} (u \cdot \pi^{*}\alpha) = \int_{\R^{N}} i_{*}\varphi_{*} \bigl( \ch(u) \wedge \pi^{*}\ch(\alpha) \bigr) \\
	&= \int_{\R^{N}} i_{*}\varphi_{*} \bigl( \ch^{(n)}u \wedge \Td(u) \wedge \pi^{*}\alpha \bigr) = f_{!}(\Td(u) \wedge \pi^{*}\alpha),
\end{align*}
therefore
	\[a(\ch(f_{!}\alpha)) = a(f_{!}(\Td(u) \wedge \pi^{*}\alpha)) \overset{\eqref{PushForwardDR}}= f_{!}(a(\ch\alpha)). \qedhere
\]
\end{proof}
\begin{Corollary}\label{GysinC10Cor} The Gysin map associated to a $h^{\bullet}$-oriented smooth map $f \colon Y \to X$ induces the following morphism of exact sequences of $\h^{\bullet}$-modules:
	\[\xymatrix{
	\cdots \ar[r] & h^{\bullet}(Y) \ar[r] \ar[d] & h^{\bullet}(Y) \otimes_{\Z} \R \ar[r] \ar[d] & \hat{h}^{\bullet+1}_{\fl}(Y) \ar[r] \ar[d] & h^{\bullet+1}(Y) \ar[r] \ar[d] & \cdots \\
	\cdots \ar[r] & h^{\bullet}(X) \ar[r] & h^{\bullet}(X) \otimes_{\Z} \R \ar[r] & \hat{h}^{\bullet+1}_{\fl}(X) \ar[r] & h^{\bullet+1}(X) \ar[r] & \cdots
}\]
where the map $h^{\bullet}(X) \otimes_{\Z} \R \rightarrow \hat{h}^{\bullet+1}_{\fl}(X)$ is defined by $\alpha \mapsto a(\ch \alpha)$.
\end{Corollary}

\section{Relative integration and integration to the point}\label{RelIntSec}

We are going to show that the Gysin map, both in the topological and in the differential case, can also be defined for classes relative to the boundary. Moreover, in this section we will define the integration of relative classes to the point, that will be used to define the relative generalized Cheeger-Simons characters and the integration map when the fibres have non-empty boundary. As far as we know, these constructions have not been considered in the literature up to now.

\SkipTocEntry \subsection{Relative Thom (iso)morphism}

If $\pi \colon E \to X$ is a real vector bundle of rank $n$ and $A \subset X$ is a topological subspace, fixing a Thom class $u$ of $E$ we also get the relative version of the Thom isomorphism, i.e., $T \colon h^{\bullet}(X, A) \to h^{\bullet+n}_{\cpt}(E, E\vert_{A})$, $\alpha \mapsto u \cdot \pi^{*}\alpha$ \cite[Theorem 11.7.34]{AGP}. When the bundle is smooth and $A$ is a closed submanifold of $X$, the same construction holds in the differential setting, getting the relative Thom morphism $T \colon \hat{h}^{\bullet}(X, A) \to \hat{h}^{\bullet+n}_{\cpt}(E, E\vert_{A})$, $\hat{\alpha} \mapsto \hat{u} \cdot \pi^{*}\hat{\alpha}$. In this section we will apply such an (iso)morphism to the following particular case: $\pi \colon E \to X$ is a smooth vector bundle, $X$ being a manifold with boundary, and $A = \partial X$. It follows that $E\vert_{A} = \partial E$, therefore we get the Thom (iso)morphism for classes relative to the boundary, i.e., $T \colon h^{\bullet}(X, \partial X) \to h^{\bullet+n}_{\cpt}(E, \partial E)$ and $T \colon \hat{h}^{\bullet}(X, \partial X) \to \hat{h}^{\bullet+n}_{\cpt}(E, \partial E)$.

\SkipTocEntry \subsection{Relative topological integration}

Let $f \colon Y \to X$ be a smooth neat map and let us fix a representative of an $h^{\bullet}$-orientation $(\iota, u, \varphi)$. We define the Gysin map on classes relative to the boundary:
\begin{equation}\label{GysinTopBd}
	f_{!!} \colon h^{\bullet}(Y, \partial Y) \to h^{\bullet - n}(X, \partial X)
\end{equation}
where $n = \dim Y - \dim X$. The definition is analogous to \eqref{GysinMapTop}, but applying the \emph{relative} Thom isomorphism on the normal bundle. Even in this case the map $f_{!!}$ only depends on the orientation $[\iota, u]$, not on the specific representative. With the same technique of \cite[Prop.\ 5.24]{Karoubi} one can prove the following theorem.
\begin{Theorem}\label{fpTopRelProperties} Let $f \colon Y \to X$ be a neat $h^{\bullet}$-oriented map  of compact manifolds.
\begin{itemize}
	\item The Gysin map $f_{!!}$ only depends on the homotopy class of $f$ as an $h^{\bullet}$-oriented map (see remark \ref{RmkHomBd}).
	\item The Gysin map is a morphism of $h^{\bullet}(X, \partial X)$-modules, i.e., given $\alpha \in h^{\bullet}(Y, \partial Y)$ and $\beta \in h^{\bullet}(X, \partial X)$:
	\[f_{!!}(\alpha \cdot f^{*}\beta) = f_{!!}\alpha \cdot \beta.
\]
	\item Given $\alpha \in h^{\bullet}(Y)$ and $\beta \in \hat{h}^{\bullet}(X, \partial X)$:
	\[f_{!!}(\alpha \cdot f^{*}\beta) = f_{!}\alpha \cdot \beta.
\]
	\item Given another neat $h^{\bullet}$-oriented map $g \colon Z \to Y$ and endowing $f \circ g$ of the naturally induced orientation (def.\ \ref{OrientedMapComposition}), we have $(f \circ g)_{!!} = f_{!!} \circ g_{!!}$.
\end{itemize}
\end{Theorem}

We remark that, if $L_{X} \colon h^{\bullet}(X) \to h_{n-\bullet}(X, \partial X)$ is the Lefschetz duality \cite{Switzer}, then $f_{!} = L_{X}^{-1} \circ f_{*} \circ L_{Y}$. If we consider the duality in the form $L'_{X} \colon h^{\bullet}(X, \partial X) \to h_{n-\bullet}(X)$, then $f_{!!} = {L'_{X}}^{-1} \circ f_{*} \circ L'_{Y}$.

\SkipTocEntry \subsection{Relative differential integration}

The Gysin map $f_{!!} \colon \hat{h}^{\bullet}(Y, \partial Y) \rightarrow \hat{h}^{\bullet - n}(X, \partial X)$, for $n = \dim Y - \dim X$, is defined similarly to \eqref{GysinMapTop}, starting from a representative of an $\hat{h}^{\bullet}$-orientation. As a consequence of formula \eqref{HomFormula}, it only depends on the corresponding orientation. Considering the following version of diagram \eqref{DiagramFibInt}:
\[\xymatrix{
	\partial X \times \R^{N} \ar@{^(->}[r] \ar[d]_{\pi_{X}\vert_{\partial X \times \R^{N}}} & X \times \R^{N} \ar[d]^{\pi_{X}} \\
	\partial X \ar@{^(->}[r] & X
}
\]
and applying formulas \eqref{FiberInt1} and \eqref{ModWedge}, we define the \emph{relative curvature map}:
\begin{equation}\label{CurvatureMapRel}
\begin{split}
	R^{rel}_{(\iota, \hat{u}, \varphi)} \colon & \Omega^{\bullet}(Y, \partial Y; \h^{\bullet}_{\R}) \to \Omega^{\bullet-n}(X, \partial X; \h^{\bullet}_{\R}) \\
	& (\omega, \rho) \mapsto \int_{X \times \R^{N}/X} i_{*}\varphi_{*}(R(\hat{u}) \wedge \pi^{*}(\omega, \rho)).
\end{split}
\end{equation}
Calling $\partial (\iota, \hat{u}, \varphi)$ the representative induced on the boundary by $(\iota, \hat{u}, \varphi)$, it follows that
\begin{equation}\label{RRelSum}
	R^{rel}_{(\iota, \hat{u}, \varphi)}(\omega, \rho) = \bigl( R_{(\iota, \hat{u}, \varphi)}(\omega), R_{\partial (\iota, \hat{u}, \varphi)}(\rho) \bigr).
\end{equation}
It is easy to prove from the axioms that:
	\[R(f_{!!}\hat{\alpha}) = R^{rel}_{(\iota, \hat{u}, \varphi)}(R(\hat{\alpha})) \qquad f_{!!}a(\omega, \rho) = a(R^{rel}_{(\iota, \hat{u}, \varphi)}(\omega, \rho)),
\]
thus the following diagram commutes:
\begin{footnotesize}
\begin{equation}\label{CurvAMapsRel}
	\xymatrix{
	\Omega^{\bullet-1}(Y, \partial Y; \h^{\bullet}_{\R})/\IIm(d) \ar[r]^(.65){a} \ar[d]^{R^{rel}_{(\iota, \hat{u}, \varphi)}} & \hat{h}^{\bullet}(Y, \partial Y) \ar[r]^{I} \ar[d]^{f_{!!}} \ar@/^2pc/[rr]^{R} & h^{\bullet}(Y, \partial Y) \ar[d]^{f_{!!}} & \Omega_{\cl}^{\bullet}(Y, \partial Y; \h^{\bullet}_{\R}) \ar[d]^{R^{rel}_{(\iota, \hat{u}, \varphi)}} \\
		\Omega^{\bullet-n-1}(X, \partial X; \h^{\bullet}_{\R})/\IIm(d) \ar[r]^(.65){a} & \hat{h}^{\bullet-n}(X, \partial X) \ar[r]^{I} \ar@/_2pc/[rr]_{R} & h^{\bullet-n}(X, \partial X) & \Omega_{\cl}^{\bullet-n}(X, \partial X; \h^{\bullet}_{\R}).
	}
\end{equation}
\end{footnotesize}

We now consider a submersion $f \colon Y \to X$, choosing proper representatives of orientations.

\begin{Theorem}\label{fpDifPropertiesRel} Let $f \colon Y \to X$ be a neat $\hat{h}^{\bullet}$-oriented submersion between compact manifolds.
\begin{itemize}
	\item The Gysin map $f_{!!}$ only depends on the homotopy class of $f$ as an $\hat{h}^{\bullet}$-oriented map (see remark \ref{RmkHomBd}).
	\item The Gysin map is a morphism of $\hat{h}^{\bullet}(X)$-modules, i.e., given $\hat{\alpha} \in h^{\bullet}(Y, \partial Y)$ and $\hat{\beta} \in \hat{h}^{\bullet}(X)$:
	\[f_{!!}(\hat{\alpha} \cdot f^{*}\hat{\beta}) = f_{!!}\hat{\alpha} \cdot \hat{\beta}.
\]
	\item Given $\hat{\alpha} \in h^{\bullet}(Y)$ and $\hat{\beta} \in \hat{h}^{\bullet}(X, \partial X)$:
	\[f_{!!}(f^{*}\hat{\beta} \cdot \hat{\alpha}) = \hat{\beta} \cdot f_{!}\hat{\alpha}.
\]
	\item Given another neat $\hat{h}^{\bullet}$-oriented map $g \colon Z \to Y$ and endowing $f \circ g$ of the naturally induced orientation (def.\ \ref{OrientedMapComposition}), we have $(f \circ g)_{!!} = f_{!!} \circ g_{!!}$.
	\item Considering the following version of diagram \eqref{DiagramFibInt}:
\[\xymatrix{
	\partial Y \ar@{^(->}[r] \ar[d]_{f\vert_{\partial Y}} & Y \ar[d]^{f} \\
	\partial X \ar@{^(->}[r] & X
}
\]
and applying formulas \eqref{FiberInt1} and \eqref{ModWedge}, we have:
	\begin{equation}\label{RAIntSubmersionRel}
	R(f_{!!}\hat{\alpha}) = \int_{Y/X} \Td(\hat{u}) \wedge R(\hat{\alpha}) \qquad f_{!!}a(\omega, \eta) = a \biggl( \int_{Y/X} \Td(\hat{u}) \wedge (\omega, \eta) \biggr).
	\end{equation}
\end{itemize}
\end{Theorem}
Equations \eqref{RAIntSubmersionRel} follows from formula \eqref{IntProp} (in the relative setting) and the commutativity of diagram \eqref{CurvAMapsRel}. 

\SkipTocEntry \subsection{Flat classes}

The relative Gysin map $f_{!!} \colon \hat{h}^{\bullet}(Y, \partial Y) \to \hat{h}^{\bullet - n}(X, \partial X)$, defined in the previous section, depends on the $\hat{h}^{\bullet}$-orientation of $f$, but, if we restrict it to flat classes, it only depends on the topological $h^{\bullet}$-orientation. The reason is the same of section \ref{FlGysin}, applying the relative Thom isomorphism $T_{\fl} \colon \hat{h}^{\bullet}_{\fl}(X, \partial X) \to \hat{h}^{\bullet + n}_{\fl,\cpt}(E, \partial E)$. Lemma \ref{fpTopRelProperties} keeps on holding (for any $f$, not necessarily a submersion). The relative versions of lemma \ref{GysinC10} (in the case of a submersion, the right-hand side of equation \eqref{RAIntSubmersionRel}) and corollary \ref{GysinC10Cor} hold with the same proof.

\SkipTocEntry \subsection{Integration to the point - Manifolds without boundary}

If $X$ is an $h^{\bullet}$-oriented manifold of dimension $n$ without boundary, the integration $(p_{X})_{!} \colon h^{\bullet}(X) \to \mathfrak{h}^{\bullet-n}$ is well-defined applying \eqref{GysinMapTop}. The same holds about the differential extension, defining $(p_{X})_{!} \colon \hat{h}^{\bullet}(X) \to \hat{\mathfrak{h}}^{\bullet-n}$ through \eqref{GysinMapDif}. Since $p_{X}$ is a submersion, formula \eqref{RAIntSubmersion} becomes the following in this case:
\begin{equation}\label{CurvPtTop}
	R((p_{X})_{!}(\hat{\alpha})) = \int_{X} \Td(X) \wedge R(\hat{\alpha}) \qquad (p_{X})_{!}(a(\omega)) = a \biggl( \int_{X} \Td(\hat{u}) \wedge \omega \biggr).
\end{equation}
As a particular case, $X$ can be the boundary of another manifold. Equivalently, we consider a manifold with non-empty boundary $X$ and the integration to the point $(p_{\partial X})_{!}$. We start from the following preliminary lemma, then we will show the behaviour of $(p_{\partial X})_{!}$.

\begin{Lemma}\label{ARoofBoundaryLemma} Let $X$ be a $\hat{h}^{\bullet}$-oriented manifold with non-empty boundary and $\Phi \colon X \to I$ a defining function for the boundary, as a part of the orientation of $X$ (see def.\ \ref{OrientedManifoldBoundary}). For any $\hat{\alpha} \in \hat{h}^{\bullet}(X)$, we have:
\begin{equation}\label{ARoofBoundary}
	\int_{0}^{1}R(\Phi_{!}\hat{\alpha}) = \int_{X} \Td(X) \wedge R(\hat{\alpha}).
\end{equation}
\end{Lemma}
\begin{proof} Let $(\iota, \hat{u}, \varphi)$ be any representative of the orientation of $\Phi$. Because of the commutativity of diagram \eqref{CurvAMaps}, we have that $R(\Phi_{!}\hat{\alpha}) = R_{(\iota, \hat{u}, \varphi)}(R(\hat{\alpha}))$. It follows from definition \eqref{CurvMapBd} that $\int_{0}^{1}R(\Phi_{!}\hat{\alpha}) = R^{\partial}_{(\iota, \hat{u}, \varphi)}(R(\hat{\alpha}))$, hence the result follows from formula \eqref{RBdTodd}.
\end{proof}

\begin{Theorem} Let $X$ be a $\hat{h}^{\bullet}$-oriented manifold with non-empty boundary. For any $\hat{\alpha} \in \hat{h}^{\bullet}(X)$, considering the induced $\hat{h}^{\bullet}$-orientation on $\partial X$, we have:
\begin{equation}\label{StokesPt}
	(p_{\partial X})_{!}(\hat{\alpha}\vert_{\partial X}) = -a \biggl( \int_{X} \Td(X) \wedge R(\hat{\alpha}) \biggr).
\end{equation}
In particular, in the topological framework, $(p_{\partial X})_{!}(\alpha\vert_{\partial X}) = 0$.
\end{Theorem}
\begin{proof} Let $\Phi$ be a defining function for the boundary, as a part of the orientation of $X$. Since $\Phi^{-1}\{1\} = \emptyset$, the map $\partial \Phi \colon \partial X \to \partial I$ can be identified with $p_{\partial X} \colon \partial X \to \{0\}$. Thanks to formula \eqref{RestrictionBoundaryGysin}, one has $(p_{\partial X})_{!}(\hat{\alpha}\vert_{\partial X}) = (\Phi_{!}\hat{\alpha})\vert_{\{0\}}$. Since $(\Phi_{!}\hat{\alpha})\vert_{\{1\}} = 0$, because $\Phi^{-1}(1) = \emptyset$, from the homotopy formula \eqref{HomFormula} we have:
\begin{align*}
	(p_{\partial X})_{!}(\hat{\alpha}\vert_{\partial X}) &= - \bigl( (\Phi_{!}\hat{\alpha})\vert_{\{1\}} - (\Phi_{!}\hat{\alpha})\vert_{\{0\}} \bigr) \overset{\eqref{HomFormula}}= -a\biggl( \int_{I} R(\Phi_{!}\hat{\alpha}) \biggr).
\end{align*}
The result follows from formula \eqref{ARoofBoundary}.
\end{proof}

\SkipTocEntry \subsection{Integration to the point - Manifolds with boundary}

When $X$ has a boundary, neither of the two Gysin maps $(p_{X})_{!}$ and $(p_{X})_{!!}$ is well defined, since $p_{X}$ is not neat. Nevertheless, we can define the integration map to the point for classes relative to the boundary, that we denote anyway by $(p_{X})_{!!}$, in the following way. In the topological framework, we set:
\begin{equation}\label{GysinPointBd}
\begin{split}
	(p_{X})_{!!} \colon & h^{\bullet}(X, \partial X) \to \mathfrak{h}^{\bullet - n} \\
	& \alpha \mapsto \int_{S^{1}} \Phi_{!!}(\alpha),
\end{split}
\end{equation}
where the map $\Phi \colon (X, \partial X) \to (I, \partial I)$ is provided by the orientation of $X$ (see def.\ \ref{OrientedManifoldBoundary}) and the integration over $S^{1}$ is defined as follows. Since
\begin{equation}\label{IsoIS1}
	h^{\bullet + 1 - n}(I, \partial I) \simeq h^{\bullet + 1 - n}(I/\partial I, \partial I/\partial I) \simeq h^{\bullet + 1 - n}(S^{1}, *),
\end{equation}
`$*$' being a marked point on $S^{1}$, we apply the suspension isomorphism
	\[\int_{S^{1}} \colon h^{\bullet + 1 - n}(S^{1}, *) \overset{\!\simeq}\longrightarrow h^{\bullet - n}(S^{0}, **) \simeq \mathfrak{h}^{\bullet - n},
\]
`$**$' being a marked point on $S^{0}$.

In order to define the integration map for differential classes, we can apply a formula analogous to \eqref{GysinPointBd}, but we have to define the integration over $S^{1}$, since the isomorphism \eqref{IsoIS1} does not apply any more. We first show how to integrate a parallel class defined on $(I, \partial I)$. The idea is that, glueing the two boundary points, we get a class on $(S^{1}, *)$ as in the topological case. Nevertheless, we have to take care of the smoothness condition when glueing the two extrema, hence we need a class that vanishes not only on $\partial I$, but also in an open neighbourhood $[0, \varepsilon) \cup (1-\varepsilon, 1]$. In order to achieve this condition, we consider a smooth function $\xi \colon (I, \partial I) \to (I, \partial I)$ such that $\xi\vert_{[0, \varepsilon)} = 0$ and $\xi\vert_{(1-\varepsilon, 1]} = 1$.\footnote{It is natural to think of $\xi$ as an increasing function, but we will see that it is not necessary, since in any case it is smoothly homotopic to the identity of $I$ relatively to $\partial I$.} Given $\hat{\alpha} \in \hat{h}^{\bullet}_{\ppar}(I, \partial I)$, we consider its pull-back $\xi^{*}\hat{\alpha} \in \hat{h}^{\bullet}_{\ppar}(I, \partial I)$. Thanks to the following lemma, the class $\xi^{*}\hat{\alpha}$ induces a well-defined class on $(S^{1}, *)$, that we can integrate, getting a class on the point. Finally, we will have to prove that the latter is independent of the choice of $\xi$.
\begin{Not} We set $I' := [0, \varepsilon) \cup (1-\varepsilon, 1]$ for a fixed $\varepsilon$. Moreover, we denote by $\pi \colon (I, \partial I) \to (S^{1}, *)$ the natural projection and we set $S' := \pi(I')$. We think of $\pi$ as a map of pairs $\pi \colon (I, I') \to (S^{1}, S')$.
\end{Not}
\begin{Lemma}\label{UniqueS1} The pull-back $\pi^{*} \colon \hat{h}_{\ppar}^{\bullet}(S^{1}, S') \to \hat{h}_{\ppar}^{\bullet}(I, I')$ is an isomorphism.
\end{Lemma}
\begin{proof} Given $\hat{\alpha} \in \hat{h}_{\ppar}^{\bullet}(I, I')$, we have to show that there exists a unique class $\hat{\beta} \in \hat{h}_{\ppar}^{\bullet}(S^{1}, S')$ such that $\pi^{*}(\hat{\beta}) = \hat{\alpha}$. We set $\alpha := I(\hat{\alpha})$. Since $\pi^{*}$ is an isomorphism in (topological) cohomology, there exists a unique class $\beta \in h^{\bullet}(S^{1}, S')$ such that $\pi^{*}\beta = \alpha$. We choose any parallel differential refinement $\hat{\beta}' \in \hat{h}^{\bullet}_{\ppar}(S^{1}, S')$ of $\beta$. It follows that $\pi^{*}\hat{\beta}' = \hat{\alpha} + a(\eta, 0)$, with $\eta\vert_{I'} = 0$. There exists a unique form $\bar{\eta}$ on $(S^{1}, S')$ such that $\pi^{*}\bar{\eta} = \eta$, thus we set $\hat{\beta} := \hat{\beta}' - a(\bar{\eta}, 0)$ and we get $\pi^{*}\hat{\beta} = \hat{\alpha}$. About the uniqueness, let us suppose that $\hat{\beta}''$ is another parallel class such that $\pi^{*}\hat{\beta}'' = \hat{\alpha}$. Then $\pi^{*}(\hat{\beta} - \hat{\beta}'') = 0$, thus, since $\pi^{*}$ is an isomorphism in cohomology, $I(\hat{\beta} - \hat{\beta}'') = 0$. It follows that $\hat{\beta} - \hat{\beta}'' = a(\xi, 0)$, with $\pi^{*}a(\xi, 0) = 0$, thus $[(\pi^{*}\xi, 0)] = \ch u$. Since $\pi^{*}$ is an isomorphism in cohomology, there exists $v$ such that $u = \pi^{*}v$, hence $\pi^{*}[(\xi, 0)] = \pi^{*}(\ch v)$. Again since $\pi^{*}$ is an isomorphism in cohomology, it follows that $[(\xi, 0)] = \ch v$, hence $a(\xi, 0) = 0$, therefore $\hat{\beta} = \hat{\beta}''$.
\end{proof}

\begin{Not} Considering the statement of lemma \ref{UniqueS1}, we set $\pi_{*} := (\pi^{*})^{-1}$.
\end{Not}

Summarizing, given a class $\hat{\alpha} \in \hat{h}^{\bullet}_{\ppar}(I, \partial I)$ and a smooth function $\xi \colon (I, \partial I) \to (I, \partial I)$ such that $\xi\vert_{[0, \varepsilon)} = 0$ and $\xi\vert_{(1-\varepsilon, 1]} = 1$, we get $\xi^{*}\hat{\alpha} \in \hat{h}^{\bullet}_{\ppar}(I, I')$. Applying lemma \ref{UniqueS1}, we get $\pi_{*}\xi^{*}\hat{\alpha} \in \hat{h}_{\ppar}^{\bullet}(S^{1}, S')$. We can think of $\pi_{*}\xi^{*}\hat{\alpha}$ as an absolute class on $S^{1}$, applying the pull-back with respect to the natural morphism $\id_{S^{1}} \colon (S^{1}, \emptyset) \to (S^{1}, S')$, therefore we can integrate $\pi_{*}\xi^{*}\hat{\alpha}$ on $S^{1}$. We get the following integration map:
\begin{equation}\label{GysinPointBdDifPar}
\begin{split}
	\int_{I} \colon \; & \hat{h}^{\bullet}_{\ppar}(I, \partial I) \to \hat{\h}^{\bullet - 1} \\
	& \hat{\alpha} \mapsto \int_{S^{1}} \pi_{*}\xi^{*}\hat{\alpha}.
\end{split}
\end{equation}

\begin{Lemma}\label{IndepXi} The integration map \eqref{GysinPointBdDifPar} is independent of the choice of $\xi$.
\end{Lemma}
\begin{proof} Let us fix two maps $\xi$ and $\xi'$, vanishing on $[0, \varepsilon) \cup (1-\varepsilon, 1]$ and $[0, \varepsilon') \cup (1-\varepsilon', 1]$ respectively. We call $\int_{I}$ and $\int'_{I}$ the corresponding integration maps \eqref{GysinPointBdDifPar}. We set $J := I = [0, 1]$, in order to distinguish the two components of $I \times I = I \times J$ (this notation will make clearer the fibre-wise integrations). Let us consider a homotopy $\Xi \colon (I, \partial I) \times J \to (I, \partial I)$ between them. By formula \eqref{HomFormula2}, thinking of $\xi$ and $\xi'$ as relative maps $(\xi, \xi\vert_{\partial I}), (\xi', \xi'\vert_{\partial I}) \colon (I, \partial I) \to (I, \partial I)$, we get
	\[\xi^{*}\hat{\alpha} - (\xi')^{*}\hat{\alpha} = a \biggl( \int_{I \times J/I} R(\Xi^{*}\hat{\alpha}) \biggr).
\]
It follows that:
\begin{align*}
	\int'_{I} \hat{\alpha} - \int_{I} \hat{\alpha} & = \int_{S^{1}} \pi_{*}(\xi^{*}\hat{\alpha} - (\xi')^{*}\hat{\alpha}) = \int_{S^{1}} \pi_{*} a \biggl( \int_{I \times J/I} R(\Xi^{*}\hat{\alpha}) \biggr) \\
	& = a \biggl( \int_{S^{1}} \pi_{*} \int_{I \times J/I} R(\Xi^{*}\hat{\alpha}) \biggr) = a \biggl( \int_{I \times J} \Xi^{*}R(\hat{\alpha}) \biggr).
\end{align*}
Since $\hat{\alpha} \in \hat{h}^{n}_{\ppar}(I, \partial I)$, it follows that $R(\hat{\alpha}) \in \Omega^{n}(I; \h^{\bullet}_{\R}) = \Omega^{0}(I; \h^{n}_{\R}) \oplus \Omega^{1}(I; \h^{n-1}_{\R})$. Integrating over $I \times J$, only the components of degree $2$ or more are meaningful, therefore we get $0$. This shows that $\int_{I} \hat{\alpha} = \int'_{I} \hat{\alpha}$.
\end{proof}

Now we can give the following definition, for an $\hat{h}^{\bullet}$-oriented manifold with boundary $X$ of dimension $n$:
\begin{equation}\label{GysinPointBdDif}
\begin{split}
	(p_{X})_{!!} \colon & \hat{h}^{\bullet}(X, \partial X) \to \hat{\h}^{\bullet - n} \\
	& \hat{\alpha} \mapsto \int_{I} \Psi_{1} \bigl( \Phi_{!!}(\hat{\alpha}) \bigr),
\end{split}
\end{equation}
$\Psi_{1}$ being the first component of the isomorphism \eqref{IsoPsiI} with $Y = \{pt\}$. For any representative $(\iota, \hat{u}, \varphi)$ of an orientation of $\Phi$, we call $\partial (\iota, \hat{u}, \varphi)$ the induced representative on $\partial X$ and we define the following curvature map:
\begin{equation}\label{CurvatureMapPtBd}
\begin{split}
	R^{pt}_{(\iota, \hat{u}, \varphi)} \colon & \Omega^{\bullet}(X, \partial X; \h^{\bullet}_{\R}) \to \Omega^{\bullet-n}(pt; \h^{\bullet}_{\R}) \\
	& (\omega, \rho) \mapsto R^{\partial}_{(\iota, \hat{u}, \varphi)}(\omega) + R_{\partial(\iota, \hat{u}, \varphi)}(\rho).
\end{split}
\end{equation}
It follows from formulas \eqref{IntProp} and \eqref{RBdTodd} that:
\begin{equation}\label{CurvatureMapPtBdTodd}
	R^{pt}_{(\iota, \hat{u}, \varphi)}(\omega, \rho) = \int_{X} \Td(X) \wedge \omega + \int_{\partial X} \Td(\partial X) \wedge \rho.
\end{equation}

\begin{Theorem} The following diagram is commutative:\footnote{In the diagram, observe that $\Omega^{n}(pt; \h^{\bullet}_{\R}) = \Omega^{0}(pt; \h^{n}_{\R}) \simeq \h^{n}_{\R}$. Moreover, every form on the point is closed and only the zero one (in any degree) is exact.}
\begin{footnotesize}
\begin{equation}\label{CurvAMapsPt}
	\xymatrix{
	\Omega^{\bullet-1}(X, \partial X; \h^{\bullet}_{\R})/\IIm(d) \ar[r]^(.63){a} \ar[d]^{R^{pt}_{(\iota, \hat{u}, \varphi)}} & \hat{h}^{\bullet}(X, \partial X) \ar[r]^{I} \ar[d]^{(p_{X})_{!!}} \ar@/^2pc/[rr]^{R} & h^{\bullet}(X, \partial X) \ar[d]^{(p_{X})_{!!}} & \Omega_{\cl}^{\bullet}(X, \partial X; \h^{\bullet}_{\R}) \ar[d]^{R^{pt}_{(\iota, \hat{u}, \varphi)}} \\
		\h^{\bullet-n-1}_{\R} \ar[r]^{a} & \hat{\h}^{\bullet-n} \ar[r]^{I} \ar@/_2pc/[rr]_{R} & \h^{\bullet-n} & \h^{\bullet-n}_{\R}.
	}
\end{equation}
\end{footnotesize}
\end{Theorem}
\begin{proof} Let us consider $(\omega, \eta) \in \Omega^{\bullet-1}(X, \partial X; \h^{\bullet}_{\R})$. Let us compute $(p_{X})_{!!} \circ a(\omega, \eta)$. Applying $(\Phi_{X})_{!!} \circ a(\omega, \eta)$, because of the commutativity of diagram \eqref{CurvAMapsRel}, we get $a(\omega', 0 \sqcup \eta')$, where $\omega' = R_{(\iota, \hat{u}, \varphi)}(\omega)$ and $\eta' = R_{\partial(\iota, \hat{u}, \varphi)}(\eta)$. We extend $0 \sqcup \eta'$ to $\tilde{\eta}'$ on $I$, so that $a(\omega', \eta') = a((\omega', \eta') - d(\tilde{\eta}', 0)) = a((\omega', \eta') - (d\tilde{\eta}', 0 \sqcup \eta)) = a(\omega' - d\tilde{\eta}', 0)$. Since $\cov(a(\omega' - d\tilde{\eta}', 0)) = 0 \sqcup \omega'\vert_{0})$, by definition, $(p_{X})_{!!} \circ a(\omega, \eta) = \int_{S^{1}}\pi_{*}\xi^{*}a(\omega' - d\tilde{\eta}' - (1-t)\omega'\vert_{0}) = a(\int_{I}(\omega' - d\tilde{\eta}' - (1-t)\omega'\vert_{0}))$. The integral of $(1-t)\omega'\vert_{0}$ vanishes since there is not $dt$ leg, hence we get $a(\int_{I}\omega' - \int_{\partial I}\tilde{\eta}') = a(\int_{I}\omega' + \eta') = a(R^{\partial}_{(\iota, \hat{u}, \varphi)}(\omega) + R_{\partial(\iota, \hat{u}, \varphi)}(\eta)) = a(R^{pt}_{(\iota, \hat{u}, \varphi)}(\omega, \eta))$.

The commutativity with $I$ follows from the construction of $(p_{X})_{!!}$, since the differential integration to the point is a refinement of the topological one.

With respect to the curvature, let us consider $\hat{\alpha} \in \hat{h}^{\bullet}(X, \partial X)$, with $R(\hat{\alpha}) = (\omega, \eta)$. Applying $(\Phi_{X})_{!!}$ we get $\hat{\beta} := (\Phi_{X})_{!!}\hat{\alpha}$. Because of the commutativity of diagram \eqref{CurvAMapsRel}, we have that $R(\hat{\beta}) = (\omega', 0 \sqcup \eta')$, where $\omega' = R_{(\iota, \hat{u}, \varphi)}(\omega)$ and $\eta' = R_{\partial(\iota, \hat{u}, \varphi)}(\eta)$. By definition, $(p_{X})_{!!}(\hat{\alpha}) = \int_{S^{1}}\pi_{*}\xi^{*}(\hat{\beta} - a((1-t)\eta', 0))$, whose curvature is $\int_{I}(R'(\hat{\beta}) + dt \wedge \eta') = (\int_{I} \omega') + \eta' = R^{\partial}_{(\iota, \hat{u}, \varphi)}(\omega) + R_{\partial(\iota, \hat{u}, \varphi)}(\eta) = R^{pt}_{(\iota, \hat{u}, \varphi)}(\omega, \eta)$.
\end{proof}

As a particular case, $X$ can be one component of a manifold with split boundary. Equivalently, we consider a manifold with split boundary $(X, M, N)$, with $M \neq \emptyset$, and the integration to the point $(p_{M})_{!!}$. We set again $J := I$, in order to distinguish the two components of $I \times J = I \times I$. In order to achieve a result analogous to formula \eqref{StokesPt} in the relative case, we consider a defining function for the boundary $\Phi \colon X \to I \times J$, we call $\pi_{J} \colon I \times J \to J$ the projection and we set $\Phi' := \pi_{J} \circ \Phi \colon X \to J$. It follows that ${\Phi'}^{-1}\{0\} = M$, hence $\Phi'\vert_{M} = p_{M}$, and ${\Phi'}^{-1}\{1\} = \emptyset$. Of course $\Phi'$ is not neat, for the same reason why $p_{M}$ is not.

Topologically, we can define the following integration map:
\begin{equation}\label{GysinIBd}
\begin{split}
	\Phi'_{!!} \colon & h^{\bullet}(X, N) \to h^{\bullet - n + 1}(J) \\
	& \alpha \mapsto \int_{S^{1} \times J/J} \Phi_{!!}(\alpha),
\end{split}
\end{equation}
considering $\Phi \colon (X, N) \to (I \times J, \partial I \times J)$. The map $\Phi_{!!} \colon h^{\bullet}(X, N) \to h^{\bullet-n+2}(I \times J, \partial I \times J)$ is defined by the same construction of the relative Gysin map \eqref{GysinTopBd}, using the Thom isomorphism relative to $N$. The integration over $S^{1}$ is defined observing that $h^{\bullet - n + 2}(I \times J, \partial I \times J) \simeq h^{\bullet - n + 2}(S^{1} \times J, \{*\} \times J) \simeq \tilde{h}^{\bullet - n + 2}(S^{1} \wedge (J_{+}))$ and applying the suspension isomorphism $\tilde{h}^{\bullet - n + 2}(S^{1} \wedge (J_{+})) \simeq \tilde{h}^{\bullet - n + 1}(J_{+}) \simeq h^{\bullet - n + 1}(J)$.

In order to define the analogous integration map in the differential framework, we consider a smooth function $\xi \colon (I, \partial I) \to (I, \partial I)$ such that $\xi\vert_{[0, \varepsilon)} = 0$ and $\xi\vert_{(1-\varepsilon, 1]} = 1$, and we think of it as a function $\xi \colon (I \times J, \partial I \times J) \to (I \times J, \partial I \times J)$, constant on $J$. Since lemma \ref{UniqueS1} keeps on holding, with respect to the pull-back $\pi^{*} \colon \hat{h}_{\ppar}^{\bullet}(S^{1} \times J, S' \times J) \to \hat{h}_{\ppar}^{\bullet}(I \times J, I' \times J)$, we get the integration map:
\begin{equation}\label{GysinIBdDifPar}
\begin{split}
	\int_{I \times J/J} \colon \; & \hat{h}^{\bullet}_{\ppar}(I \times J, \partial I \times J) \to \hat{h}^{\bullet - 1}(J) \\
	& \hat{\alpha} \mapsto \int_{S^{1}} \pi_{*}\xi^{*}\hat{\alpha}.
\end{split}
\end{equation}
Therefore, we define:
\begin{equation}\label{GysinIBdDif}
\begin{split}
	\Phi'_{!!} \colon & \hat{h}^{\bullet}(X, N) \to \hat{h}^{\bullet - n + 1}(J) \\
	& \hat{\alpha} \mapsto \int_{I \times J/J} \Psi_{1} \bigl( \Phi_{!!}(\hat{\alpha}) \bigr),
\end{split}
\end{equation}
$\Psi_{1}$ being the first component of the isomorphism \eqref{IsoPsiI} with $Y = J$.

\begin{Rmk} The construction of $\Phi'_{!!}$, that we have shown, is completely analogous to the one of $(p_{X})_{!!}$, but there is only one difference, concerning the proof of lemma \ref{IndepXi}. In order to show that the integration map is independent of $\xi$, let us consider a homotopy $\Xi \colon (I, \partial I) \times J' \to (I, \partial I)$ between $\xi$ and $\xi'$, inducing the homotopy $\Xi \colon (I, \partial I) \times J \times J' \to (I, \partial I) \times J$ which is constant along $J$. With the same proof we get $\int'_{I \times J/J} \hat{\alpha} - \int_{I \times J/J} \hat{\alpha} = a \bigl( \int_{I \times J \times J'/J} \Xi^{*}R(\hat{\alpha}) \bigr)$. Now $R(\hat{\alpha})$, being defined on $I \times J$, has also a component of degree $2$, that could be non-vanishing after integrating along $J'$ and $I$. Nevertheless, such an integral is a $0$-form, whose value at $t \in J$ is $\int_{I \times J'} \Xi_{t}^{*}R(\hat{\alpha}\vert_{I \times \{t\}})$. The restriction $R(\hat{\alpha}\vert_{I \times \{t\}})$ of the degree-2 component is a 2-form on $I \times \{t\}$, hence it vanishes.
\end{Rmk}

It follows from the construction that
\begin{equation}\label{RestrictionBoundaryGysinSplit}
	\Phi'_{!!}(\hat{\alpha})\vert_{\{0\}} = (p_{M})_{!!}(\hat{\alpha}\vert_{(M, \partial M)}),
\end{equation}
since all the tools and the operations involved in the definition of $\Phi'_{!!}$ restrict on $M$ to the corresponding ones for $(p_{M})_{!!}$.

\begin{Lemma} Let $(X, M, N)$ be a $\hat{h}^{\bullet}$-oriented manifold with split boundary and $\Phi \colon X \to I \times J$ a defining function for the boundary, as a part of the orientation of $X$. We call $R'$ and $\cov$ the two components of the curvature $R$. For any $\hat{\alpha} \in \hat{h}^{\bullet}(X, N)$, we have:
\begin{equation}\label{ARoofSplitBoundary}
	\int_{J} R(\Phi'_{!!}\hat{\alpha}) = \int_{X} \Td(X) \wedge R'(\hat{\alpha}) + \int_{N} \Td(N) \wedge \cov(\hat{\alpha}).
\end{equation}
\end{Lemma}
\begin{proof} By formulas \eqref{GysinIBdDif} and \eqref{GysinIBdDifPar} we have that
\begin{align}
	\int_{J} R(\Phi'_{!!}\hat{\alpha}) &= \int_{J} R \biggl( \int_{S^{1}} \pi_{*}\xi^{*} \Psi_{1} \bigl( \Phi_{!!}(\hat{\alpha}) \bigr) \biggr) = \int_{J} \int_{S^{1}} \pi_{*}\xi^{*} R \bigl( \Psi_{1} \bigl( \Phi_{!!}(\hat{\alpha}) \bigr) \bigr) \nonumber \\
	& = \int_{I \times J} \xi^{*} R \bigl( \Psi_{1} \bigl( \Phi_{!!}(\hat{\alpha}) \bigr) \bigr) \overset{(\star)}= \int_{I \times J} R \bigl( \Psi_{1} \bigl( \Phi_{!!}(\hat{\alpha}) \bigr) \bigr). \label{SplitBd1}
\end{align}
In order to prove the equality $(\star)$, it is enough to choose $\xi$ as a diffeomorphism from $(\varepsilon, 1-\varepsilon)$ to $(0, 1)$. When we apply $\Psi_{1}$, defined in formula \eqref{GysinIBdDif}, to $\Phi_{!!}(\hat{\alpha})$, we have that $\eta_{1} = 0$, thus $\eta_{0} = \cov(\Phi_{!!}(\hat{\alpha}))$. It follows that
\begin{align}
	& \Psi_{1}(\Phi_{!!}(\hat{\alpha})) = \Phi_{!!}(\hat{\alpha}) - a \bigl( (1-t)\cov(\Phi_{!!}(\hat{\alpha})), 0 \bigr) \nonumber \\
	& R \bigl( \Psi_{1}(\Phi_{!!}(\hat{\alpha})) \bigr) = R' \bigl( \Phi_{!!}(\hat{\alpha}) \bigr) + dt \wedge \cov(\Phi_{!!}(\hat{\alpha})) - (1-t)d\cov(\Phi_{!!}(\hat{\alpha})) \nonumber \\
	& \int_{I \times J} R \bigl( \Psi_{1}(\Phi_{!!}(\hat{\alpha})) \bigr) = \int_{I \times J} R' \bigl( \Phi_{!!}(\hat{\alpha}) \bigr) + \int_{J} \cov(\Phi_{!!}(\hat{\alpha})). \label{SplitBd2}
\end{align}
The term $(1-t)d\cov(\Phi_{!!}(\hat{\alpha}))$ vanishes when integrated on $I$, since there is no $dt$ component. Joining \eqref{SplitBd1} and \eqref{SplitBd2} we get:
\begin{equation}\label{SplitBd3}
	\int_{J} R(\Phi'_{!!}\hat{\alpha}) = \int_{I \times J} R' \bigl( \Phi_{!!}(\hat{\alpha}) \bigr) + \int_{J} \cov(\Phi_{!!}(\hat{\alpha})).
\end{equation}
Let $(\iota, \hat{u}, \varphi)$ be any representative of the orientation of $\Phi$. Because of the commutativity of diagram \eqref{CurvAMapsRel} and formula \eqref{RRelSum}, on $(I \times J, \partial I \times J)$ we have that
	\[R(\Phi_{!!}\hat{\alpha}) \overset{\eqref{CurvAMapsRel}}= R^{rel}_{(\iota, \hat{u}, \varphi)}(R(\hat{\alpha})) \overset{\eqref{RRelSum}}= \bigl( R_{(\iota, \hat{u}, \varphi)}(R'(\hat{\alpha})), R_{(\iota, \hat{u}, \varphi)\vert_{N}}(\cov(\hat{\alpha})) \bigr).
\]
Therefore:
\begin{align*}
	& \int_{I \times J} R' \bigl( \Phi_{!!}(\hat{\alpha}) \bigr) = \int_{I \times J} R_{(\iota, \hat{u}, \varphi)}(R'(\hat{\alpha})) \overset{(\#)}= \int_{X} \Td(X) \wedge R'(\hat{\alpha}) \\
	& \int_{J} \cov(\Phi_{!!}(\hat{\alpha})) = \int_{J} R_{(\iota, \hat{u}, \varphi)\vert_{N}}(\cov(\hat{\alpha})) \overset{\eqref{CurvMapBd}}= R^{\partial}_{(\iota, \hat{u}, \varphi)\vert_{N}}(\cov(\hat{\alpha})) \overset{\eqref{RBdTodd}}= \int_{N} \Td(N) \wedge \cov(\hat{\alpha}).
\end{align*}
The equality $(\#)$ follows again from formula \eqref{RBdTodd}, adapted to the case of a manifold with split boundary.
\end{proof}

\begin{Theorem}\label{StokesPtRelThm} Let $(X, M, N)$ be a $\hat{h}^{\bullet}$-oriented manifold with split boundary. For any $\hat{\alpha} \in \hat{h}^{\bullet}(X, N)$, considering the induced $\hat{h}^{\bullet}$-orientation on $M$, we have:
\begin{equation}\label{StokesPtRel}
	(p_{M})_{!!}(\hat{\alpha}\vert_{(M, \partial M)}) = -a \biggl( \int_{X} \Td(X) \wedge R'(\hat{\alpha}) + \int_{N} \Td(N) \wedge \cov(\hat{\alpha}) \biggr).
\end{equation}
In particular, in the topological framework, $(p_{M})_{!!}(\alpha\vert_{(M, \partial M)}) = 0$.
\end{Theorem}
\begin{proof} Let $\Phi$ be a defining function for the boundary, as a part of the orientation of $X$. The map $\Phi_{M} \colon M \to I$ can be identified with a defining function for the boundary of $M$. Since $(\Phi'_{!!}\hat{\alpha})\vert_{\{1\}} = 0$, because $\Phi^{-1}(I \times \{1\}) = \emptyset$, from formula \eqref{RestrictionBoundaryGysinSplit} and the homotopy formula \eqref{HomFormula} we have:
\begin{align*}
	(p_{M})_{!!}(\hat{\alpha}\vert_{(M, \partial M)}) &= - \bigl( (\Phi'_{!!}\hat{\alpha})\vert_{\{1\}} - (\Phi'_{!!}\hat{\alpha})\vert_{\{0\}} \bigr) \overset{\eqref{HomFormula}}= -a\biggl( \int_{J} R(\Phi'_{!!}\hat{\alpha}) \biggr).
\end{align*}
The result follows from formula \eqref{ARoofSplitBoundary}.
\end{proof}
Finally, we remark that, since the flat theory is a (topological) cohomology theory, we can integrate a flat class over the point just using \eqref{GysinPointBd}. We get the integration map $(p_{X})_{!!} \colon \hat{h}^{\bullet}_{\fl}(X, \partial X) \to \hat{\mathfrak{h}}^{\bullet - n}_{\fl}$, that only depends on the topological $h^{\bullet}$-orientation of $X$.

\section{Flat pairing and generalized Cheeger-Simons characters}\label{SecFlatPCS}

We are going to define the relative version of generalized Cheeger-Simons characters, starting from flat classes.

\SkipTocEntry \subsection{Relative homology}

We extend to the relative case the geometric model for the dual homology theory $h_{\bullet}$, described in \cite{Jakob}. When we say ``relative'', we consider the cohomology of any smooth map, not necessarily the embedding of the boundary as in the previous section. The following definition generalizes the one given in \cite{FR}.
\begin{Def}\label{DualHomology} Given a continuous map $\rho: A \rightarrow X$, between spaces having the homotopy type of a finite CW-complex, we define:
\begin{itemize}
	\item the group of \emph{$n$-precycles} as the free abelian group generated by the quintuples $(M, u, \alpha, f, g)$, with:
\begin{itemize}
	\item $(M, u)$ a smooth compact manifold, possibly with boundary, with $h^{\bullet}$-orientation $u$, whose connected components $\{M_{i}\}$ have dimension $n+q_{i}$, with $q_{i}$ arbitrary;
	\item $\alpha \in h^{\bullet}(M)$, such that $\alpha\vert_{M_{i}} \in h^{q_{i}}(M)$;
	\item $f \colon M \to X$ a continuous function;
	\item $g \colon \partial M \to A$ a continuous function such that $\rho \circ g = f\vert_{\partial M}$;
\end{itemize}
	\item the group of \emph{$n$-cycles}, denoted by $z_{n}(\rho)$, as the quotient of the group of $n$-precycles by the free subgroup generated by elements of the form:
\begin{itemize}
	\item $(M, u, \alpha + \beta, f, g) - (M, u, \alpha, f, g) - (M, u, \beta, f, g)$;
	\item $(M, u, \alpha, f, g) - (M_{1}, u\vert_{M_{1}}, \alpha\vert_{M_{1}}, f\vert_{M_{1}}, g\vert_{\partial M_{1}}) - (M_{2}, u\vert_{M_{2}}, \alpha\vert_{M_{2}}, f\vert_{M_{2}}, g\vert_{\partial M_{2}})$, for $M = M_{1} \sqcup M_{2}$;
	\item $(M, u, \varphi_{!}\alpha, f, g) - (N, v, \alpha, f \circ \varphi, g \circ \varphi\vert_{\partial N})$ for $\varphi \colon N \to M$ a neat submersion, oriented via the 2x3 principle, and $\varphi_{!}$ the Gysin map for \emph{absolute} classes ($\alpha$ is not relative to the boundary);
\end{itemize}
	\item the group of \emph{$n$-boundaries}, denoted by $b_{n}(\rho)$, as the subgroup of $z_{n}(\rho)$ generated by the cycles which are representable by a pre-cycle $(M, u, \alpha, f, g)$ such that there exists a quintuple $((W, M, N), U, A, F, G)$, where $(W, M, N)$ is a manifold with split boundary, $U$ is an $h^{\bullet}$-orientation of $W$ and $U\vert_{M} = u$, $A \in h^{\bullet}(W)$ such that $A\vert_{M} = \alpha$, $F \colon W \rightarrow X$ is a smooth map satisfying $F\vert_{M} = f$ and $G \colon N \rightarrow A$ is a smooth map satisfying $\rho \circ G = F\vert_{N}$ and $G\vert_{\partial N} = g$.
\end{itemize}
We define $h_{n}(\rho) := z_{n}(\rho) / b_{n}(\rho)$.
\end{Def}
There is a natural map:
\begin{equation}\label{HomCohom}
\begin{split}
	\xi^{n} \colon & h^{n}(\rho) \to \Hom_{\h^{\bullet}}(h_{n-\bullet}(\rho), \h^{\bullet}) \\
	& \alpha \mapsto \bigl( \, [M, u, \beta, f, g] \mapsto (p_{M})_{!!}(\beta \cdot (f, g)^{*}\alpha) \bigr),
\end{split}
\end{equation}
where $p_{M} \colon M \to \{pt\}$ (see def.\ \eqref{GysinPointBdDif}) and $(f, g)$ is the following morphism:
	\[\xymatrix{
	\partial M \ar@{^(->}[r]^{\iota} \ar[d]_{g} & M \ar[d]^{f} \\
	A \ar[r]^{\rho} & X.
}\]
In order to multiply $\beta$ and $(f, g)^{*}\alpha$, we used the module structure \eqref{AbsRelMod}, since $\beta$ is an absolute class on $M$, while $(f, g)^{*}\alpha$ is relative to the boundary. We verify that \eqref{HomCohom} is well-defined. If we consider a neat submersion $\varphi \colon N \rightarrow M$ and two representatives $(M, u, \varphi_{!}\beta, f, g)$ and $(N, v, \beta, f \circ \varphi, g \circ \varphi\vert_{\partial N})$ of the homology class, we have:
\begin{align*}
	\xi^{n}(\alpha)[N, v, &\beta, f \circ \varphi, g \circ \varphi\vert_{\partial N}] = (p_{N})_{!!}(\beta \cdot (\varphi, \varphi\vert_{\partial N})^{*} (f,g)^{*}\alpha) \\
	&= (p_{M})_{!!} (\varphi, \varphi\vert_{\partial N})_{!!}(\beta \cdot (\varphi, \varphi\vert_{\partial N})^{*} (f,g)^{*}\alpha) \\
	&= (p_{M})_{!!} (\varphi_{!}\beta \cdot (f,g)^{*}\alpha) = \xi^{n}(\alpha)[M, u, \varphi_{!}\beta, f, g].
\end{align*}
If $(M, u, \beta, f, g) = \partial ((W, M, N), U, B, F, G)$, then, by theorem \ref{StokesPtRelThm}, $(p_{M})_{!!}(\beta \cdot (f, g)^{*}\alpha) = 0$, thus $\xi^{n}(\alpha)$ is well-defined on homology classes. Finally, the image of $\alpha$ is a $\mathfrak{h}^{\bullet}$-module homomorphism, since, for $\gamma \in \mathfrak{h}^{t}$:
	\[\begin{split}
	\xi^{n}(\alpha)([(M, u, \beta, f, g)] \cap \gamma) &= \xi^{n}(\alpha)[M, u, \beta \cdot (p_{M})^{*}\gamma, f, g] = (p_{M})_{!!}(\beta \cdot (f, g)^{*}\alpha \cdot (p_{M})^{*}\gamma) \\
	& = (p_{M})_{!!}(\beta \cdot (f, g)^{*}\alpha) \cdot \gamma = \xi^{n}(\alpha)[M, u, \beta, f, g] \cdot \gamma.
\end{split}\]
Tensorizing with $\R$, we get the isomorphism:
\begin{equation}\label{HomCohomR}
	\xi^{n}_{\R} \colon h^{n}(\rho) \otimes_{\Z} \R \overset{\!\simeq}\longrightarrow \Hom_{\h^{\bullet}}(h_{n-\bullet}(\rho), \h^{\bullet}_{\R}).
\end{equation}
Moreover, thanks to the structure of $h^{\bullet}$-module on $h^{\bullet}(\,\cdot\,; \R/\Z)$, we get the following map:
\begin{equation}\label{HomCohomRZ1}
\begin{split}
	\xi^{n}_{\R/\Z} \colon & h^{n}(\rho; \R/\Z) \to \Hom_{\h^{\bullet}}(h_{n-\bullet}(\rho), \h^{\bullet}_{\R/\Z}) \\
	& \alpha \mapsto \bigl( \, [M, u, \beta, f, g] \mapsto (p_{M})_{!!}(\beta \cdot (f, g)^{*}\alpha) \bigr).
\end{split}
\end{equation}

\SkipTocEntry \subsection{Flat pairing}\label{FlatPairing}

We define the natural $\hat{\h}^{\bullet}_{\fl}$-valued pairing for a map $\rho \colon A \to X$ between $\hat{h}^{\bullet}_{\fl}$ and $h_{\bullet}$, that, in the case of singular differential cohomology, reduces to the holonomy of a flat relative Deligne cohomology class. When $\hat{h}^{\bullet}_{\fl} \simeq h^{\bullet}(\,\cdot\,; \R/\Z)$, such a pairing coincides with formula \eqref{HomCohomRZ1}.
\begin{Def} For $\rho \colon A \to X$ a smooth map (not necessarily neat), we have the following natural pairing:
\begin{equation}\label{FlatPairing2}
\begin{split}
	\xi^{n}_{\fl} \colon & \hat{h}^{n}_{\fl}(\rho) \to \Hom_{\mathfrak{h}^{\bullet}} (h_{n-\bullet}(\rho), \hat{\h}^{\bullet}_{\fl}) \\
	& \hat{\alpha} \mapsto \bigl( \, [M, u, \beta, f, g] \mapsto (p_{M})_{!!}(\beta \cdot (f,g)^{*}\hat{\alpha}) \bigr).
\end{split}
\end{equation}
The invariance by $\h^{\bullet}$ is defined by:
\begin{equation}\label{PairingInvariance}
	\xi^{n}_{\fl}(\hat{\alpha})([M, u, \beta, f, g] \cdot \gamma) = \xi^{n}_{\fl}(\hat{\alpha})([M, u, \beta, f, g]) \cdot \gamma.
\end{equation}
\end{Def}
\vspace{0.2cm}
In order to show that \eqref{FlatPairing2} is well-defined, i.e.\ that it does not depend on the representative $(M, u, \beta, f, g)$, and that formula \eqref{PairingInvariance} holds, we apply the same argument used about \eqref{HomCohom}.
\begin{Lemma}\label{MorphismComplexes} We have the following morphism of complexes of $\h^{\bullet}$-modules (the lower one not being exact in general):
\begin{scriptsize}
	\[\xymatrix{
	\cdots \ar[r]^(.3){r} & h^{n}(\rho) \otimes_{\Z} \R \ar[r]^{a} \ar[d]^{\xi^{n}_{\R}} & \hat{h}^{n+1}_{\fl}(\rho) \ar[r]^{I} \ar[d]^{\xi^{n+1}_{\fl}} & h^{n+1}(\rho) \ar[r]^(.55){r} \ar[d]^{\xi^{n+1}} & \cdots \\
	\cdots \ar[r]^(.3){r'} & \Hom_{\h^{\bullet}}(h_{n-\bullet}(\rho), \h^{\bullet}_{\R}) \ar[r]^(.47){a'} & \Hom_{\h^{\bullet}}(h_{n+1-\bullet}(\rho), \hat{\h}^{\bullet}_{\fl}) \ar[r]^{I'} & \Hom_{\h^{\bullet}}(h_{n+1-\bullet}(\rho), \h^{\bullet}) \ar[r]^(.75){r'} & \cdots
}\]
\end{scriptsize}
\end{Lemma}
\begin{proof} We only have to prove the commutativity of the square under the map $a$. It easily follows from the fact that, for $\alpha \in h^{\bullet}(\rho) \otimes_{\Z} \R$ and $\beta \in h^{\bullet}(X)$:
	\[a(\ch\alpha) \cdot \beta = a(\ch(\alpha\beta)).
\]
That's because, for any differential refinement $\hat{\beta}$ of $\beta$, we have $a(\ch\alpha) \cdot \hat{\beta} = a(\ch\alpha \cdot R(\hat{\beta})) = a(\ch\alpha \cdot \ch\beta) = a(\ch(\alpha\beta))$. \end{proof}

We call $\h^{n}_{\Z}$ the image of the Chern character $\ch \colon \h^{n} \to H^{n}_{\dR}(pt; \h^{\bullet}_{\R}) \simeq \h^{n}_{\R}$, which coincides with $\alpha \mapsto \alpha \otimes_{\Z} \R$.
\begin{Theorem}\label{NoTorsionRZ} If $\h^{\bullet}$ has no torsion, the pairing \eqref{FlatPairing2} is an isomorphism and $\hat{\h}^{\bullet}_{\fl} \simeq \h^{\bullet-1}_{\R}/\h^{\bullet-1}_{\Z}$.
\end{Theorem}
\begin{proof} Same of \cite[Theorem 5.5]{FR}.
\end{proof}

\SkipTocEntry \subsection{Homology via differential cycles}\label{DiffCycles}

We can define pre-cycles, cycles and boundaries as in definition \ref{DualHomology}, but refining each orientation and each cohomology class to a differential one. We call $\hat{z}_{n}(\rho)$ and $\hat{b}_{n}(\rho)$ the corresponding groups of cycles and boundaries. It follows that $\hat{z}_{n}(\rho)$ is generated by classes of the form $[(M, \hat{u}, \hat{\alpha}, f, g)]$, and $\hat{b}_{n}(\rho)$ is generated by cycles with a representative such that $(M, \hat{u}, \hat{\alpha}, f, g) = \partial ((W, M, N), \hat{U}, \hat{A}, F, G)$. We define $h'_{n}(\rho) := \hat{z}_{n}(\rho) / \hat{b}_{n}(\rho)$.

\begin{Theorem}\label{TopDiffIso} The natural group morphism:
	\[\begin{split}
	\Phi \colon & h'_{\bullet}(\rho) \to h_{\bullet}(\rho) \\
	& [(M, \hat{u}, \hat{\alpha}, f, g)] \to [(M, I(\hat{u}), I(\hat{\alpha}), f, g)]
\end{split}\]
is an isomorphism.
\end{Theorem}
\begin{proof} It follows from the same result about absolute classes \cite[Theorem 6.2]{FR} and the five lemma applied to the long exact sequence in homology associated to $\rho$. Alternatively, one can adapt to the relative case the same proof of \cite[Theorem 6.2]{FR}.
\end{proof}

\SkipTocEntry \subsection{Cheeger-Simons characters}

The following definition generalizes to any cohomology theory the one of \cite{BT} and \cite{FR3} (type II).

\begin{Def}\label{GeneralizedCS} A \emph{Cheeger-Simons differential $\hat{h}^{\bullet}$-character} of degree $n$ on $\rho \colon A \to X$ is a triple $(\chi_{n}, \omega_{n}, \eta_{n-1})$, where:
\begin{equation}\label{GeneralizedCSDef}
	\chi_{n} \in \Hom_{\hat{\h}^{\bullet}} (\hat{z}_{n-\bullet}(\rho), \hat{\h}^{\bullet}) \quad\qquad (\omega_{n}, \eta_{n-1}) \in \Omega^{n}(\rho; \h^{\bullet}_{\R})
\end{equation}
such that, if $(M, \hat{u}, \hat{\beta}, f, g) = \partial((W, M, N), \hat{U}, \hat{B}, F, G)$, then:
\begin{equation}\label{CSFormula}
	\chi_{n}[M, \hat{u}, \hat{\beta}, f, g] = -a \biggl( \int_{W} \Td(W) \wedge R(\hat{B}) \wedge F^{*}\omega_{n} + \int_{N} \Td(N) \wedge R(\hat{B}\vert_{N}) \wedge G^{*}\eta_{n-1} \biggr).
\end{equation}
The $\hat{\h}^{\bullet}$-invariance is defined by:
\begin{equation}\label{CSInvariance}
	\chi_{n}(\hat{\alpha})([M, \hat{u}, \hat{\beta}, f, g] \cdot \hat{\gamma}) = \chi_{n}(\hat{\alpha})[M, \hat{u}, \hat{\beta}, f, g] \cdot \hat{\gamma}.
\end{equation}
We denote by $\check{h}^{n}(\rho)$ the group of characters of degree $n$.
\end{Def}

We briefly comment on formula \eqref{CSFormula}. Let us suppose that $[M, \hat{u}, \hat{\beta}, f, g] \in \hat{z}_{n-k}(X)$ and that $M$ is connected. Then $\dim(M) = n - k + q$ and $\hat{\beta} \in \hat{h}^{q}(M)$, hence $\dim(W) = n - k + q + 1$ and $\hat{B} \in \hat{h}^{q}(W)$. Thus, in the r.h.s.\ of \eqref{CSFormula}, we integrate on $W$ a $\h^{\bullet}_{\R}$-valued form of degree $0 + q + n$, hence we get a form on the point of degree $q + n - (n - k + q + 1) = k - 1$. Applying $a$, we get a class belonging to $\hat{\h}^{k}$, as desired.
\begin{Theorem}\label{CShnThm} There is a natural graded-group morphism:
\begin{equation}\label{CShn}
\begin{split}
	CS_{\hat{h}}^{\bullet} \colon & \hat{h}^{\bullet}(\rho) \to \check{h}^{\bullet}(\rho) \\
	& \hat{\alpha} \mapsto (\chi, R(\hat{\alpha})),
\end{split}
\end{equation}
where $\chi$ is defined, for $[M, \hat{u}, \hat{\beta}, f, g] \in \hat{z}_{n-k}(\rho)$, by:
	\[\chi[M, \hat{u}, \hat{\beta}, f, g] := (p_{M})_{!!}(\hat{\beta} \cdot (f,g)^{*}\hat{\alpha}).
\]
\end{Theorem}
\begin{proof} If we consider two representatives $(M, u, \varphi_{!}\beta, f, g)$ and $(N, v, \beta, f \circ \varphi, g \circ \varphi\vert_{\partial N})$ of the same homology class, we have:
\begin{align*}
	\chi[N, \hat{v}, &\hat{\beta}, f \circ \varphi, g \circ \varphi\vert_{\partial N}] = (p_{N})_{!!}(\hat{\beta} \cdot (\varphi, \varphi\vert_{\partial N})^{*} (f,g)^{*}\hat{\alpha}) \\
	&= (p_{M})_{!!} (\varphi, \varphi\vert_{\partial N})_{!!}(\hat{\beta} \cdot (\varphi, \varphi\vert_{\partial N})^{*} (f,g)^{*}\hat{\alpha}) \\
	&= (p_{M})_{!!} (\varphi_{!}\hat{\beta} \cdot (f,g)^{*}\hat{\alpha}) = \chi[M, \hat{u}, \varphi_{!}\hat{\beta}, f, g].
\end{align*}
Let us now suppose that $(M, \hat{u}, \hat{\beta}, f, g) = \partial ((W, M, N), \hat{U}, \hat{B}, F, G)$. From formula \eqref{StokesPtRel}, replacing $X$ by $W$ and $\hat{\alpha}$ by $\hat{\beta} \cdot (f,g)^{*}\hat{\alpha}$, we get formula \eqref{CSFormula}. Finally:
	\[\begin{split}
	\chi([(M, \hat{u}, \hat{\beta}, f, g)] \cap \hat{\gamma}) &= \chi[M, \hat{u}, \hat{\beta} \cdot (p_{M})^{*}\hat{\gamma}, f, g] = (p_{M})_{!!}(\hat{\beta} \cdot (f, g)^{*}\hat{\alpha} \cdot (p_{M})^{*}\hat{\gamma}) \\
	& = (p_{M})_{!!}(\hat{\beta} \cdot (f, g)^{*}\hat{\alpha}) \cdot \hat{\gamma} = \chi[M, \hat{u}, \hat{\beta}, f, g] \cdot \hat{\gamma}. \qedhere
\end{split}\]
\end{proof}

The proof of the following theorem is straightforward from the previous definition.
\begin{Theorem}\label{HolonomyFlat} When $\hat{\alpha}$ is flat, the value of the associated Cheeger-Simons character over $[M, \hat{u}, \hat{\beta}, f, g]$ coincides with the value of \eqref{FlatPairing2} on the corresponding homology class.
\end{Theorem}
Considering the pairing \eqref{FlatPairing2}, we have the following embedding:
	\[j \colon \Hom_{\mathfrak{h}^{\bullet}} (h_{n-\bullet}(\rho), \hat{\h}^{\bullet}_{\fl}) \hookrightarrow \check{h}^{n}(\rho).
\]
In fact, a morphism $\varphi_{n} \in \Hom_{\mathfrak{h}^{\bullet}} (h_{n-\bullet}(\rho), \hat{\h}^{\bullet}_{\fl})$ determines a unique morphism $\chi_{n} \colon \hat{z}_{n-\bullet}(\rho) \to \hat{\h}^{\bullet}$ defined by $\chi_{n}[M, \hat{u}, \hat{\beta}, f, g] := \varphi_{n}[M, I(\hat{u}), I(\hat{\beta}), f, g]$, and we define $j(\varphi_{n}) := (\chi_{n}, 0, 0)$. It follows from formula \eqref{CSFormula} that the image of $j$ is the subgroup of generalized Cheeger-Simons characters with vanishing curvature, that we call $\check{h}^{n}_{\fl}(\rho)$. Let us consider the embedding $i: \hat{h}^{\bullet}_{\fl}(\rho) \hookrightarrow \hat{h}^{\bullet}(\rho)$. The following diagram commutes:
	\[\xymatrix{
	\hat{h}^{n}_{\fl}(\rho) \ar[r]^(.3){\xi^{n}_{\fl}} \ar@{^(->}[d]_{i} & \Hom_{\mathfrak{h}^{\bullet}} (h_{n-\bullet}(X), \hat{\h}^{\bullet}_{\fl}) \ar@{^(->}[d]^{j} \\
	\hat{h}^{n}(\rho) \ar[r]^{CS^{n}_{\hat{h}}} & \check{h}^{n}(\rho).
}\]
Therefore $i$ restricts to the embedding $i' \colon \Ker(\xi_{\fl}^{n}) \hookrightarrow \Ker(CS^{n}_{\hat{h}})$, and $j$ restricts to the embedding $j' \colon \IIm(\xi^{n}_{\fl}) \hookrightarrow \IIm(CS^{n}_{\hat{h}})$. Because of $j$ and $j'$ we can construct a morphism $b \colon \Coker(\xi_{\fl}^{n}) \to \Coker(CS^{n}_{\hat{h}})$. We now show that actually $i'$ and $b$ are isomorphisms.
\begin{Theorem}\label{KerCokerEqual} The following canonical isomorphisms hold:
	\[\Ker(\xi_{\fl}^{n}) \simeq \Ker(CS^{n}_{\hat{h}}), \qquad \Coker(\xi_{\fl}^{n}) \simeq \Coker(CS^{n}_{\hat{h}}).
\]
\end{Theorem}
\begin{proof} If $\hat{\alpha} \in \hat{h}^{n}(\rho)$ is not flat, then $CS^{n}_{\hat{h}}(\hat{\alpha}) \neq 0$, since $CS^{n}_{\hat{h}}(\hat{\alpha}) = (\chi_{n}, R(\hat{\alpha}))$ and $R(\hat{\alpha}) \neq 0$. Hence $\Ker(CS^{n}_{\hat{h}}) \subset \Ker(\xi_{\fl}^{n})$ and the equality follows. Moreover, $\check{h}^{n}_{\fl}(\rho) \cap \IIm(CS^{n}_{\hat{h}}) = \IIm(\xi_{\fl}^{n})$, hence $b \colon \Coker(\xi_{\fl}^{n}) \to \Coker(CS^{n}_{\hat{h}})$ is an embedding. If $(\chi_{n}, \omega_{n}, \eta_{n-1}) \in \check{h}^{n}(\rho)$, we consider a class $\hat{\alpha} \in \hat{h}^{n}(\rho)$ such that $R(\hat{\alpha}) = (\omega_{n}, \eta_{n-1})$, and we call $(\chi'_{n}, \omega_{n}, \eta_{n-1}) := CS_{\hat{h}}(\hat{\alpha})$. Then $(\chi'_{n} - \chi_{n}, 0, 0) \in \check{h}^{n}_{\fl}(\rho)$, and, in $\Coker(CS^{n}_{\hat{h}})$, one has $[(\chi_{n}, \omega_{n}, \eta_{n-1})] = [(\chi'_{n} - \chi_{n}, 0, 0)] \in \IIm\,b$. Therefore $b$ is also surjective.
\end{proof}
\begin{Corollary}\label{CorCSIsom} If $\h^{\bullet}$ has no torsion, \eqref{CShn} is an isomorphism.
\end{Corollary}
\begin{proof} It immediately follows from theorems \ref{KerCokerEqual} and \ref{NoTorsionRZ}.
\end{proof}

Corollary \ref{CorCSIsom} holds in particular for ordinary differential cohomology.\footnote{The expression ``ordinary differential cohomology'' is commonly used to indicate the differential refinement of singular cohomology.} In this case differential characters can be defined on smooth singular cycles, without enriching them with a differential cohomology class. In the absolute setting this is the definition of (ordinary) Cheeger-Simons character, introduced in \cite{CS}, and historically it is one of the starting points of the whole theory of differential cohomology.

\section{Integration relative to the boundary}\label{SecIntBd}

Let us consider a smooth fibre bundle $f \colon Y \to X$, such that $X$ is a manifold without boundary and $Y$ with boundary. It follows that the typical fibre is a manifold with boundary $M$. Moreover, the restriction of $f$ to the boundary, that we call $\partial f \colon \partial Y \to X$, is a fibre bundle too, with typical fibre $\partial M$. Of course $f$ is not neat, therefore we cannot apply the integration map as previously defined, but we can define the following integration map for classes relative to the boundary:
\begin{equation}\label{IntMapRelBd}
	f_{!!} \colon \hat{h}^{\bullet}(Y, \partial Y) \to \hat{h}^{\bullet-m}(X),
\end{equation}
$m$ being the dimension of $M$. When $X$ is a point, we get \eqref{GysinPointBdDif} as a particular case. The map \eqref{IntMapRelBd} generalizes to any cohomology theory the one described in \cite{FR3}.

\SkipTocEntry \subsection{Topological integration} Let us start with the notion of orientation. The idea is the following. We choose a neat embedding $\iota \colon Y \hookrightarrow X \times \Hh^{N}$, such that $\pi_{X} \circ \iota = f$ (restricting $\iota$ to the boundary, we get the embedding $\partial \iota \colon \partial Y \hookrightarrow X \times \R^{N-1}$). This is always possible: for example, we can choose a neat embedding $\kappa \colon Y \hookrightarrow \Hh^{N}$ and define $\iota := f \times \kappa$. Then we choose a Thom class on the normal bundle and a neat tubular neighbourhood, as always. We think of $\iota$ as a map to $X \times \R^{N-1} \times I$, through the embedding $[0, +\infty) \approx [0, 1) \subset I$. In this way, we can first integrate on $\R^{N-1}$, getting a class in $X \times I$, relative to $X \times \partial I$. This is equivalent to getting a class in $X \times S^{1}$, therefore we can integrate on $S^{1}$ and obtain the result.

Let us define this integration map using the same language of sections \ref{SecOrientInt} and \ref{RelIntSec}. Given a fibre bundle $f \colon Y \to X$, such that $\partial X = \emptyset$, a \emph{defining function for the boundary} is a smooth neat map $\Phi \colon Y \to X \times I$ such that $\partial Y = \Phi^{-1}(X \times \{0\})$ (by neatness, it follows that $\Phi^{-1}\{1\} = \emptyset$). In particular, the restriction of $\Phi$ to a fibre $Y_{x} := \pi^{-1}\{x\}$ is a defining function for the boundary of $Y_{x}$.
\begin{Def}\label{OrientedFibreBoundary} An \emph{$h^{\bullet}$-orientation} on $f \colon Y \to X$ is a homotopy class of $h^{\bullet}$-oriented defining functions for the boundary.
\end{Def}
Remarks analogous to \ref{OrManBdRmk}, \ref{OrInducedBd} and \ref{BdPartCase} hold in this case. In particular, the remark analogous to \ref{OrManBdRmk} shows that the idea we sketched at the beginning of this section corresponds to definition \ref{OrientedFibreBoundary}. We set:
\begin{equation}\label{IntegrationRelBd}
\begin{split}
	f_{!!} \colon & h^{\bullet}(Y, \partial Y) \to h^{\bullet-m}(X) \\
	& \alpha \mapsto \int_{S^{1}} \Phi_{!!}(\alpha),
\end{split}
\end{equation}
where the map $\Phi \colon (Y, \partial Y) \to (X \times I, X \times \partial I)$ is provided by the orientation of $f$ and the integration over $S^{1}$ is defined as follows. Since $h^{\bullet + 1 - m}(X \times I, X \times \partial I) \simeq h^{\bullet + 1 - m}(X \times S^{1}, X \times \{*\}) = \tilde{h}^{\bullet + 1 - m}(X_{+} \wedge S^{1})$, `$*$' being a marked point on $S^{1}$, we apply the suspension isomorphism $\tilde{h}^{\bullet + 1 - m}(X_{+} \wedge S^{1}) \simeq \tilde{h}^{\bullet - m}(X_{+}) \simeq h^{\bullet - m}(X)$ and we get the result. The same construction holds for differential integration of flat classes.

\SkipTocEntry \subsection{Differential integration}

We generalize the curvature map \eqref{CurvMapBd} in the following natural way:
\begin{equation}\label{CurvMapFibreBd}
\begin{split}
	R^{\partial}_{(\iota, \hat{u}, \varphi)} \colon & \Omega^{\bullet}(Y; \h^{\bullet}_{\R}) \to \Omega^{\bullet-m}(X; \h^{\bullet}_{\R}) \\
	& \omega \mapsto \int_{X \times \R^{N-1} \times I/X} i_{*}\varphi_{*}(R(\hat{u}) \wedge \pi^{*}\omega)
\end{split}
\end{equation}
and we define:
\begin{equation}\label{CurvatureMapPtFibreBd}
\begin{split}
	R^{pt}_{(\iota, \hat{u}, \varphi)} \colon & \Omega^{\bullet}(Y, \partial Y; \h^{\bullet}_{\R}) \to \Omega^{\bullet-m}(X; \h^{\bullet}_{\R}) \\
	& (\omega, \eta) \mapsto R^{\partial}_{(\iota, \hat{u}, \varphi)}(\omega) + R_{\partial(\iota, \hat{u}, \varphi)}(\eta).
\end{split}
\end{equation}
Requiring that the orientation is proper, i.e. that the fibre of the normal bundle of $\iota(Y)$ in $\iota(y) = (x, t)$ is sent by $\varphi$ to a subset of $\{x\} \times \Hh^{n}$, it follows from formulas analogous to \eqref{IntProp} and \eqref{RBdTodd} that:
\begin{equation}\label{CurvatureMapPtFibreBdTodd}
	R^{pt}_{(\iota, \hat{u}, \varphi)}(\omega, \eta) = \int_{Y/X} \Td(\hat{u}) \wedge \omega + \int_{\partial Y/X} \Td(\hat{u}\vert_{\partial Y}) \wedge \eta.
\end{equation}

\begin{Def}\label{OrientedDiffManifoldFibreBoundary} An \emph{$\hat{h}^{\bullet}$-orientation} on $f \colon Y \to X$ is a homotopy class of $\hat{h}^{\bullet}$-oriented defining functions for the boundary, considering the curvature map \eqref{CurvMapFibreBd} (equivalently, \eqref{CurvatureMapPtFibreBd}) in the definition of homotopy.
\end{Def}

Corollary \ref{HomotThomOr} and lemma \ref{Rule23ManBd} hold with the same proof. Applying the isomorphism \eqref{IsoPsiI} with $Y = X$, we can integrate on $\R^{N}$ and apply the following integration map analogous to \eqref{GysinPointBdDifPar}:
\begin{equation}\label{GysinPointBdDifPar2}
\begin{split}
	\int_{I} \colon \; & \hat{h}^{\bullet}_{\ppar}(X \times I, X \times \partial I) \to \hat{h}^{\bullet - 1}(X) \\
	& \hat{\alpha} \mapsto \int_{X \times S^{1}} \pi_{*}\xi^{*}\hat{\alpha}.
\end{split}
\end{equation}
This defines \eqref{IntMapRelBd}. The following lemma, that generalize \cite[Theorem 47 p.\ 170]{BB} to any cohomology theory, shows the naturality of this integration map.
\begin{Theorem} Let us consider the following diagram, whose rows are segments of the long exact sequences \eqref{LongExact2} associated respectively to $i \colon \partial Y \hookrightarrow Y$ and $\id_{X} \colon X \to X$:
	\[\xymatrix{
	\hat{h}^{\bullet-1}(\partial Y) \ar[r]^{B_{1}} \ar[d]_(.38){(\partial f)_{!}} & \hat{h}^{\bullet}(Y, \partial Y) \ar[r]^(.55){\pi_{1}^{*}} \ar[dl]_(.53){f_{!!}} & \hat{h}^{\bullet}(Y) \ar[dl]_(.55){f_{\#}} \ar[r]^{i^{*}} & \hat{h}^{\bullet}(\partial Y) \ar[dl]_(.55){(\partial f)_{!}} \\
	\hat{h}^{\bullet-m}(X) \ar[r]^(.45){B_{2}} & \hat{h}^{\bullet-m+1}(\id_{X}) \ar[r]^{\pi_{2}^{*}} & \hat{h}^{\bullet-m+1}(X).
}\]
We set $f_{\#}(\hat{\alpha}) := -a\bigl(\int_{Y/X} \Td(Y/X) \wedge  R(\hat{\alpha}), 0 \bigr)$. The left triangle and the right parallelogram of the diagram commute. The central parallelogram commutes provided that we restrict $\hat{h}^{\bullet}(Y, \partial Y)$ to the subgroup of parallel classes.
\end{Theorem}
\begin{proof} Let us fix a defining map $\Phi \colon Y \to X \times I$ that represent the fixed orientation of $f$. It follows that $\partial f = \Phi\vert_{\partial Y} \colon \partial Y \to X \times \{0\}$. Moreover, we call $B_{I} \colon \hat{h}^{\bullet-1}(X \times \partial I) \to \hat{h}^{\bullet}(X \times I)$ the Bockstein map of the sequence \eqref{LongExact2} associated to the embedding $X \times \partial I \hookrightarrow X \times I$. About the left triangle, since all of the steps (S1)--(S6) in section \ref{ConstrBockSec} commute with the Gysin map, it follows that $\Phi_{!!}(B_{1}(\hat{\alpha})) = B_{I}((\partial f)_{!}(\hat{\alpha}))$. Moreover, applying the isomorphism \eqref{IsoPsiI} and the integration map \eqref{GysinPointBdDifPar2} to a class of the form $B_{I}(\hat{\alpha}')$, the class $\pi_{*}\xi^{*}\Psi_{1}(B_{I}(\hat{\alpha}')) \in \hat{h}^{\bullet}(X \times S^{1}, X \times S')$ coincides with the class $\hat{\beta}$ provided by lemma \ref{LiftCurvS1} such that $\int_{X \times S^{1}} \hat{\beta} = \hat{\alpha}'$ and $R'(\hat{\beta}) = \pi_{1}^{*}R(\hat{\alpha}') \wedge dt$. It follows that
\begin{align*}
	f_{!!}(B_{1}(\hat{\alpha})) &= \int_{I} \Psi_{1}\Phi_{!!}(B_{1}(\hat{\alpha})) = \int_{I} \Psi_{1}B_{I}((\partial f)_{!}(\hat{\alpha})) \\
	& = \int_{S^{1}} \pi_{*}\xi^{*}\Psi_{1}(B_{I}((\partial f)_{!}(\hat{\alpha}))) = (\partial f)_{!}(\hat{\alpha}).
\end{align*}
The commutativity of the right parallelogram is essentially formula \eqref{StokesPt} in the case of a fibre bundle. About the central parallelogram, since $\Omega^{\bullet-1}(X) \simeq \hat{h}^{\bullet}(\id_{X})$, $\omega \simeq a(\omega, 0)$, and $\omega = \cov(a(\omega, 0))$, it follows from formula \eqref{CurvBockstein} that $B_{2}(\hat{\beta}) = a(-R(\hat{\beta}))$. Since the integration map commutes with the curvature \eqref{CurvatureMapPtFibreBd}, it follows from formula \eqref{CurvatureMapPtFibreBdTodd} that
	\[B_{2}(f_!!(\hat{\alpha})) = -a\left( \int_{Y/X} \Td(Y/X) \wedge R'(\hat{\alpha}) + \int_{\partial Y/X} \Td(\partial Y/X) \wedge \cov(\hat{\alpha}), 0 \right).
\]
If $\cov(\hat{\alpha}) = 0$, we get precisely $f_{\#}\pi_{2}^{*}(\hat{\alpha})$.
\end{proof}

%%%%%%%%%%%%%%%%%%%%%%%%%%%%%%%%%%%%%%%%%%%%%%%%%%%%%%%%%%%%%%%%%%%%%%%%%%%%%%%%%%%%%%%%%%%%%%%%%

\end{document}